\documentclass[11pt]{article}

\usepackage{amsmath}
\usepackage{amsthm}
\usepackage[dvipsnames]{xcolor}
\usepackage[colorlinks=true,linkcolor=ForestGreen]{hyperref}
\usepackage{svg}
\usepackage[capitalize,nameinlink,noabbrev]{cleveref}
\crefformat{equation}{(#2#1#3)}
\usepackage{amsfonts}
\usepackage{graphicx}
\usepackage{epstopdf}
\usepackage{algorithmic}
\usepackage{enumitem}
\usepackage[margin=1.2in, top=1.2in, bottom=1.2in]{geometry}
\usepackage[utf8]{inputenc}

\newcommand{\TheTitle}{Composite Optimization by Nonconvex Majorization-Minimization} 

\newcommand{\intdom}[1]{\operatorname{int\, dom} #1}
\newcommand{\dom}[1]{\operatorname{dom} #1}

\title{{\TheTitle}}

\author{
	Jonas Geiping\thanks{University of Siegen, Hölderlinstraße 3, 57076 Siegen, Germany
		(\href{jonas.geiping@uni-siegen.de}{jonas.geiping@uni-siegen.de}, \href{michael.moeller@uni-siegen.de}{michael.moeller@uni-siegen.de})} 
	\and
	Michael Moeller\footnotemark[2] 
}

\usepackage{amsopn}

\newcommand{\R}{\mathbb{R}}
\newcommand{\N}{\mathbb{N}}
\DeclareMathOperator*{\argmin}{\arg \min}%
\newtheorem{theorem}{Theorem}

\newtheorem{lemma}{Lemma}
\newtheorem{corollary}{Corollary}
\theoremstyle{definition}
\newtheorem{definition}{Definition}[section]
\newtheorem{remark}{Remark}
\newtheorem{example}{Example}

\newcommand{\importantBox}[1]{%
	\begin{center}%
		\setlength{\fboxrule}{2pt}%
		\fbox{%
			\begin{minipage}{0.95\textwidth}%
				#1
			\end{minipage}%
		}%
	\end{center}%
}

\begin{document}
	
	\maketitle
	\begin{abstract}
\noindent The minimization of a nonconvex composite function can model a variety of imaging tasks. 
A popular class of algorithms for solving such problems are majorization-minimization techniques which iteratively approximate the composite nonconvex function by a majorizing function that is easy to minimize. Most techniques, e.g. gradient descent, utilize convex majorizers in order to guarantee that the majorizer is easy to minimize. 
In our work we consider a natural class of nonconvex majorizers for these functions, and show that these majorizers are still sufficient for a globally convergent optimization scheme. 
Numerical results illustrate that by applying this scheme, one can often obtain superior local optima compared to previous majorization-minimization methods, when the nonconvex majorizers are solved to global optimality. Finally, we illustrate the behavior of our algorithm for depth super-resolution from raw time-of-flight data.
	\end{abstract}

	

	
	
	
\section{Introduction}

Many imaging tasks that can be regarded as the minimization of some objective function, also called energy, can be solved by nonlinear optimization. Unfortunately, many energies arising from the faithful modeling of the data formation process and a state-of-the-art regularization term are inherently nonconvex, coupled, and high dimensional. Since determining the global minimizer of such a cost function is rarely feasible, one frequently turns to (gradient-based) methods that only find a, possibly sub-optimal, critical point of the energy landscape \cite{nesterov_introductory_2004}. 

Interestingly, some high-dimensional nonconvex optimization problems do admit a global solution within reasonable time. Besides problems for which the solution can be determined analytically, the aforementioned class includes \textit{separable} problems on a bounded domain, i.e. problems for which the minimization of an energy $E$ with respect to some variable $u\in \mathbb{R}^n$ decomposes into the minimization of separate low-dimensional energies, e.g. $E(u) = \sum_{i=1}^n E_i(u_i)$. 
Even more remarkably, there are several types of non-separable nonconvex optimization problems which can be reformulated as convex problems, e.g. via convex relaxation techniques \cite{chan_algorithms_2006} or via functional lifting \cite{pock_global_2010}, and still yield a globally optimal solution to the original nonconvex problem. Unfortunately, the aforementioned techniques rely on a special structure of the objective. Even seemingly minor perturbations of the required structure make it impossible to exploit these techniques, and lead practitioners to consider local (gradient-based) methods again.

\begin{figure}
	\centering
	\subfloat[]{\label{fig:non_convex:a}
		\includegraphics[width=0.49\textwidth]{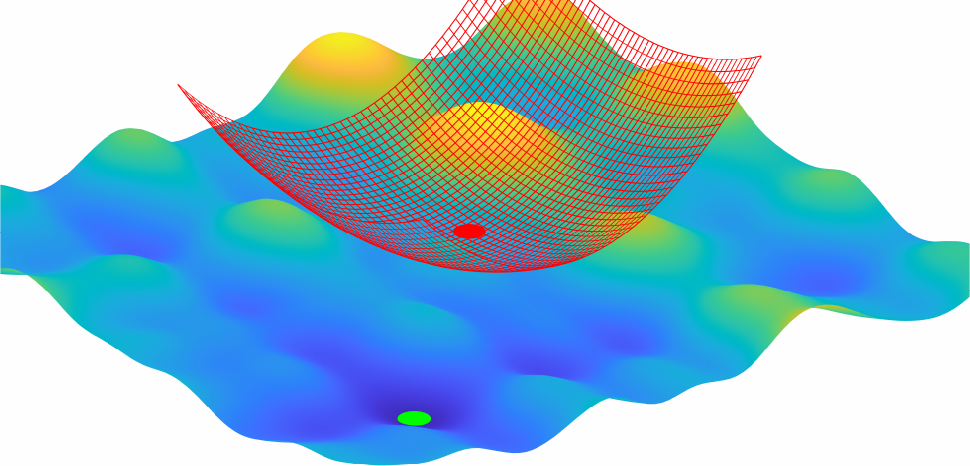}}
	\subfloat[]{\label{fig:non_convex:b}\includegraphics[width=0.49\textwidth]{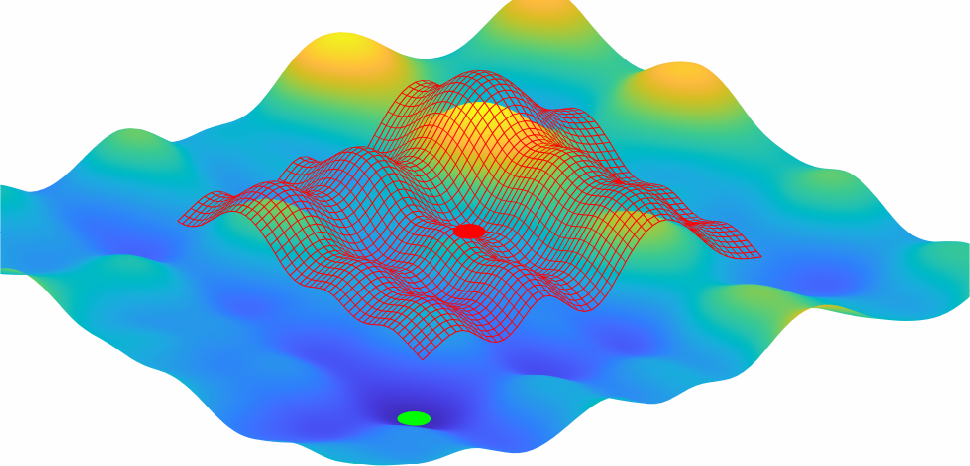}}
	\caption{Nonconvex versus convex majorization. (a) shows an energy of type \cref{eq:problem} with a convex majorizer. (b) shows the same energy, but with a solvable nonconvex majorizer. The initial point is marked in red, the global minimum of the energy in green. We can see that the shown nonconvex majorizer can better represent the given function. \label{fig:non_convex}}
\end{figure}

Interestingly, many of such local methods admit an interpretation in the framework of \textit{majorization-minimization} techniques: In each iteration, the energy $E$ is approximated by a simpler function $E_{u^k}$ which satisfies 
\begin{align*}
E_{u^k}(u) &\geq E(u), \\
E_{u^k}(u^k) &= E(u^k),
\end{align*}
for $u^k$ being the current iterate. By defining the next iterate to be the minimizer of the approximation $E_{u^k}$, 
$$u^{k+1} = \arg \min_u E_{u^k}(u), $$
one automatically obtains monotonically decreasing objective values. 

Common gradient-based methods use simple convex approximation functions $E_{u^k}$, e.g. quadratic functions,
\begin{equation}\label{eq:intro_grad}
E_{u^k}(u) = E(u^k) + \langle \nabla E(u^k),u-u^k \rangle + \frac{1}{2\tau}||u-u^k||^2,
\end{equation}
in the case of gradient descent. While this leads to easy-to-solve subproblems, such approximation functions $E_{u^k}$ are only a crude approximation of the original energy and almost all information about the shape of the original energy landscape is lost. 

In this work we propose a novel majorization-minimization technique with nonconvex functions $E_{u^k}$ with the idea to 
\begin{enumerate}
	\item approximate the original energy landscape much more faithfully, and
	\item still be able to minimize $E_{u^k}$ globally by considering functions $E_{u^k}$ that are either separable or can be minimized via relaxation techniques. 
\end{enumerate}

As illustrated in a simple two-dimensional example in \cref{fig:non_convex}, one can expect a more faithful approximation of the original energy to yield 'better' local minima: While the para-bolic approximation of \cref{fig:non_convex:a} yields a nearby local minimum, the separable nonconvex majorizer in \cref{fig:non_convex:b} allows to skip several local minima. In this example, the minimizer of the nonconvex majorizer is in a close vicinity to the global minimizer after just a single step of the algorithm. 

While our motivation comes from the (somewhat heuristic) idea of finding 'better' local minima, our convergence analysis does \textit{not} depend on the subproblems being solved to global optimality.  
For the remainder of the paper we consider the minimization of composite energies of the form
\begin{equation}\label{eq:problem}
E(u) = G(\rho(u)) +  R (u),
\end{equation}
for suitable functions $G:\R^m \to \R$, $\rho:\R^n \to \R^m$ and $R:\R^n \to \R$, via the iterative minimization of 
\begin{equation}\label{eq:aux}
E_{u^k}(u) = G(\rho(u^k)) + \langle \nabla G(\rho(u^k)),\rho(u)-\rho(u^k)\rangle + R(u) + \frac{1}{2\tau}||\rho(u)-\rho(u^k)||^2.
\end{equation}
The model function $E_{u^k}$ is a naturally global, but nonconvex, majorizer of $E$ for suitable $\tau$ as we will see later.
A typical example for 'simple' functions $\rho:\R^n \to \R^n$ and $R:\R^n \to \R$ is given when both functions are separable, i.e. $\rho(u) = (\rho_1(u_1),\dots,\rho_n(u_n))$ and $R(u) = \sum_{i=1}^n r_i(u_i)$. In this case, the nonconvex majorizer \cref{eq:aux} is then also separable and can be solved in each dimension separately.

We continue summarizing some of the related work for nonconvex and composite optimization problems and illustrate how the proposed majorization-minimization technique \cref{eq:aux} differs from the methods that have been considered in the literature so far. 

\subsection{Related Work}

The current field of nonlinear optimization is quite wide. In the following overview of related work we focus on results, that like our method do not require convexity of the objective function and we limit ourselves to generalizations of first-order methods.
The general framework of majorization-minimization methods has been reviewed widely in the literature of the recent decades, see for example, \cite{hunter_tutorial_2004,mairal_optimization_2013-1,sun_majorization-minimization_2017,wu_convergence_1983}.

The first option for tackling the minimization of \eqref{eq:problem} is to ignore the composite structure of $G\circ \rho$, naturally leading to schemes like the aforementioned gradient descent (GD) \cref{eq:intro_grad} or the closely related forward backward splitting (FBS) \cite{chen_convergence_1997,beck_fast_2009,nesterov_gradient_2013}. As we will see in more detail below, the proposed scheme recovers such algorithms in the special case of $\rho$ being the identity. The convergence\footnote{In the context of first order methods, we consider 'convergence' as implying that the sequence of iterates converges to a stationary point of the objective function.} of a general class of nonconvex first-order descent methods, including GD and FBS, was shown e.g. in \cite{attouch_convergence_2013}. It is important to note that such a convergence is nontrivial for arbitrary nonconvex functions and requires, for example, some algebraic notion of 'tameness' \cite{ioffe_invitation_2009}, that is nevertheless usually present in practice.

The most limiting assumption in these first-order methods is the Lipschitz continuity of the gradient of $F=G \circ \rho$, the first part of the objective function. This class of problems was recently extended in \cite{bauschke_descent_2017,bolte_first_2018,benning_gradient_2017} to L-smooth adaptable function \cite{}, these functions are not necessarily convex or L-smooth, only a Legendre function $h$ must exist, so that $Lh-F$ is convex for some $L>0$. The previously mentioned methods can be extended to a descent 'relative' to these Legendre functions. Defining the Bregman distance of $h$ as $D_h(u,v) = h(u)-h(v) -\langle \nabla h,u-v\rangle$, \cite{bolte_first_2018}'s majorizer can be written as 
\begin{equation}\label{eq:intro_fb}
E_{u^k}(u)  = \langle \nabla F(u^k),u-u^k \rangle + R(u) + \frac{1}{\tau}D_h(u,u^k) .
\end{equation}
They show that the sequence of iterates generated by this type of majorizer converges for appropriate $\tau$ and conditions to $h,F$ and $R$, which include the K\L-property \cite{bolte_clarke_2007} (which follows from the mentioned notion of 'tameness') and the assumption that $\dom{h} = \R^n$.

We can relate \cite{bolte_first_2018} to the earlier approach of \cite{chouzenoux_variable_2014}. Here, the functions $h$ are restricted to induced norms, however they are allowed to change during the sequence of iterations, $h^k = \frac{1}{2}||\cdot||^2_{A^k}$ where each $A^k$ is a symmetric positive definite matrix. These matrices are chosen so that \cref{eq:intro_fb} is a majorizer of $E$ at $u^k$, which is in turn guaranteed if $D_{h^k-F}(u,u^k)\geq 0$. This is a weaker assumption than $h-F$ convex, which is equivalent to $D_{h-f}(u,v) \geq 0$, but limited by the use of induced norms. \cite{chouzenoux_variable_2014} also shows global convergence under the K\L-property.

Recent works have also proposed general frameworks for iteratively replacing the original minimization problem with simple approximation functions $E_{u^k}$ beyond majorization-minimization. \cite{drusvyatskiy_nonsmooth_2016} analyzes approximation functions $E_{u^k}$, satisfying $|E_{u^k}(u)-E(u)|\leq \omega(||u-u^k||)$ for a proper growth function $\omega$. A minimization scheme of these approximation functions  exhibits subsequential convergence to critical points, even if the subproblem evaluations are inexact. These approximation functions need not necessarily be convex, but the distance of their subsequent evaluations must tend to zero.
A slightly different generalization is discussed in  \cite{ochs_non-smooth_2017}, where approximation functions constructed by $E_{u^k} = \bar{E}_{u^k}+ D_h(u,u^k)$ with $|\bar{E}_{u^k}(u)-E(u)|\leq \omega(||u-u^k||)$ are examined. Here $\omega$ is a growth function and $D_h$ a Bregman distance generated by a Legendre function, generalizing the previously discussed \cref{eq:intro_fb}.  Subsequence convergence can again be shown here, under relatively weak conditions. However the approximation function $E_{u^k}$ is taken to be convex in \cite{ochs_non-smooth_2017} to, among other properties, guarantee the success of a backtracking scheme and reach an implementable algorithm.

A review of Majorization-Minimization methods that still allow for a sequence of iterates to converge globally under the K\L-property can be found in \cite{bolte_majorization-minimization_2016}. There, majorizers $E_{v}$ are required, most prominently, to be $m$-strongly convex and to fulfill the abstract descent inequality $\operatorname{dist}(0,\partial E_{v}(u)) \leq c ||v -u ||$. This condition however, will be difficult to fulfill in our setting due to the presence of $\rho$, and we will thus seek convergence under different conditions.

Coming to related work in composite optimization we find that there are two ways to handle problems of type \cref{eq:problem}: Either we linearize the outer function $G$ in each approximation, or the inner function $\rho$. Linearizing the inner function $\rho$ leads to methods that are reminiscent of classical Levenberg-Marquardt algorithms for nonlinear least-squares problems. The approximation function can be written as
\begin{equation}\label{eq:rel_lm}
E_{u^k}(u) = G\left(\rho(u^k)+J_\rho(u^k)(u-u^k)\right) + R(u) + \frac{1}{2\tau}||u-u^k||^2 ,
\end{equation}
where $J_\rho$ denotes the Jacobian of $\rho$. A classical application for this composition are systems of nonlinear equations. Due to the inner linearization, it is in general not required that $G$ is smooth. 
Subsequence convergence follows as a result of \cite{lewis_proximal_2016-1,drusvyatskiy_nonsmooth_2016} or \cite{ochs_non-smooth_2017}. Global convergence for convex $G$ and $ R = 0$ is shown under the K\L-property in \cite{pauwels_value_2016-1}.
Further literature 
can be found under the terms 'prox-linear' or 'prox-descent', e.g. \cite{lewis_proximal_2016-1,drusvyatskiy_efficiency_2016}. 
Linearizing the outer function leads to algorithms related to iterative re-weighting procedures:
\begin{equation}\label{eq:rel_pl}
E_{u^k}(u) = G(\rho(u^k)) + \langle \nabla G(\rho(u^k)),\rho(u)-\rho(u^k)\rangle + R(u) +  \frac{1}{2\tau}||u-u^k||^2.
\end{equation}
Subsequence convergence follows from the general result of  \cite{drusvyatskiy_nonsmooth_2016} under the assumption that the distance of subsequent iterates tends to zero. Further analysis, related to special cases in iterative re-weighting can be found in \cite{ochs_iteratively_2015} or under more general assumptions, but including the convexity of $E_{u^k}$ in \cite{ochs_non-smooth_2017}. The connection to iterative reweighting is immediate for concave $G$, as then $\tau$ can be taken arbitrarily large and the proximal term vanishes.
This formulation is closely related to our work and differs from ours in the way we measure the distance to the previous iterate. We later discuss the implications of this difference.
\begin{figure}
	\centering
	\subfloat[]{\label{fig:majorizers1:a}\includegraphics[width=0.32\textwidth]{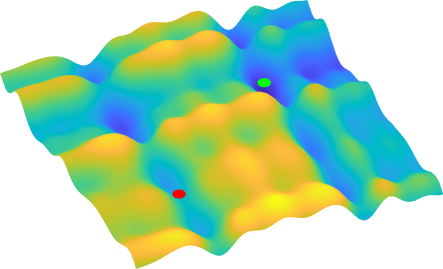}}
	\subfloat[]{\label{fig:majorizers1:b}\includegraphics[width=0.32\textwidth]{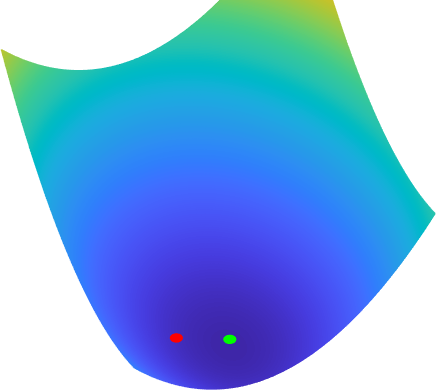}}
	\subfloat[]{\label{fig:majorizers1:c}\includegraphics[width=0.32\textwidth]{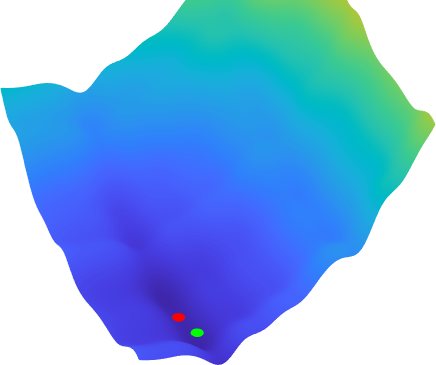}}
	\caption{Visualization of related work. (a) shows the original function of type \cref{eq:problem}, (b) shows a gradient descent majorizer \cref{eq:intro_grad}, (c) shows a forward-backward splitting majorizer \cref{eq:intro_fb}. The point $u^k$ is equal in each figure and shown in red and the minimizer $u^{k+1}$ in green.
\label{fig:majorizers1}}
\end{figure}
\begin{figure}
	\centering
	\subfloat[]{\label{fig:majorizers2:a}\includegraphics[width=0.32\textwidth]{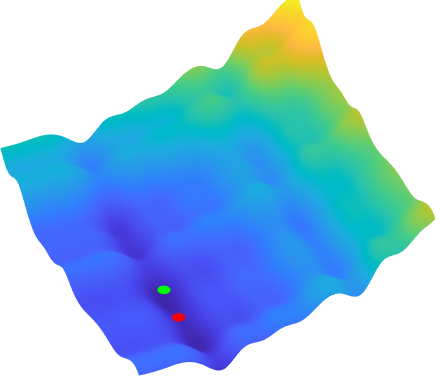}}
	\subfloat[]{\label{fig:majorizers2:b}\includegraphics[width=0.32\textwidth]{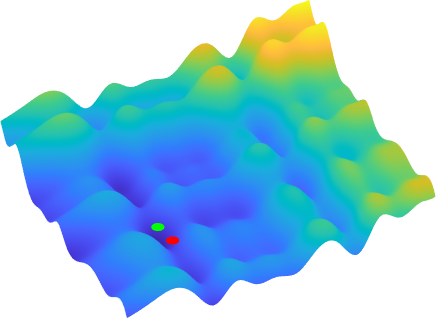}}
	\subfloat[]{\label{fig:majorizers2:c}\includegraphics[width=0.32\textwidth]{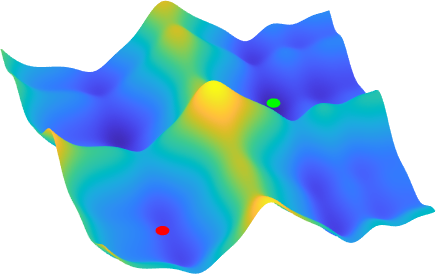}}
	\caption{Visualization of related work. (a) shows a prox-linear type inner linearization \cref{eq:rel_lm}, (b) shows an outer linearization \cref{eq:rel_pl}, (c) shows finally shows the proposed (separable) majorizer \cref{eq:aux}. The point $u^k$ is equal in each figure and shown in red and the minimizer $u^{k+1}$ in green.
		\label{fig:majorizers2}}
\end{figure}

As a first 
 visualization, \cref{fig:majorizers1} and \cref{fig:majorizers2} show these majorization functions in two dimensions. For a nonconvex function of type \cref{eq:problem} in \cref{fig:majorizers1:a}, a gradient descent majorizer is shown in \cref{fig:majorizers1:b} and a forward-backward splitting in \cref{fig:majorizers1:c}. We 
 see that for both majorizers their respective minimizers, marked in green, are located in a close neighborhood to the current iterate, marked in red. Both algorithms will likely converge to a nearby local minimum of the original energy \cref{fig:majorizers1:a}.

The presented related majorizers for composite optimization are shown in \cref{fig:majorizers2}. \Cref{fig:majorizers2:a} and \Cref{fig:majorizers2:b} show both linearization variants, namely \cref{eq:rel_lm} and \cref{eq:rel_pl}. These generally produce more faithful representations of the original energy (\cref{fig:majorizers1:a}), but both minimizers are still far away from the global minimum. 
Finally, \Cref{fig:majorizers2:c} shows our majorizer \cref{eq:aux}. Note that the minimizer of this majorizer can not only be computed efficiently due to its separability, but also allows for a global view of the function and its minimizer almost coincides with the global minimum although the initial point is quite far from it.

Finally, a recent preprint \cite{bolte_nonconvex_2018} proposes to solve composite minimization problems with a different approach, namely a nonlinear splitting variant, reformulating the problem to
\begin{equation}
\min_{u,v \in \R^n} G(v) + R(u) \quad \textnormal{ s.t. } \rho(u) = v,
\end{equation}
and introducing an augmented Lagrangian formulation
\begin{equation} \label{eq:ALM}
\min_{u,v \in \R^n} G(v) + R(u) + \langle w,\rho(u)-v \rangle + \frac{\tau}{2} ||\rho(u)-v||^2.
\end{equation}
with an additional variable $w \in \R^n$ that mimics the dual variable of the convex setting.
 This is a quite interesting result, as it shows that the complementarity of forward-backward splitting and augmented Lagrangian methods can be extended into the composite setting. Whereas our method is a generalization of forward-backward splittings, their work generalizes ADMM \cite{gabay_dual_1976}. Critically both ours and their approach rely on the efficient solution of a nonlinear programming task as intermediate step in the algorithm. For us, this is the nonconvex majorizer \cref{eq:aux}, the corresponding problem in \cite[Eq. (6.3)]{bolte_nonconvex_2018} is the minimization of \cref{eq:ALM} for $u$:
\begin{equation}
u^{k+1} = \argmin_{u} G(v^{k+1})+R(u)+\langle w^k,\rho(u)-v^{k+1}\rangle + \frac{\tau^k}{2}||\rho(u)-v^{k+1}||^2 + \frac{\mu}{2}||u-u^k||^2.
\end{equation}
Both subproblems are in general equally difficult as they are connected for $\mu=0$, identifying $v^{k+1}=\rho(u^k)$ and $w^k = \nabla G(\rho(u^k))$.

Although formulated in less generality in terms of the involved functions but in more generality in terms of the number of nested functions, the update equation of the related work \cite[Eq. (11)]{frerix_proximal_2018} for solving problem \eqref{eq:problem} can be written as
\begin{align*}
u^{k+1} =& \arg\min_u R(u) + \frac{1}{2}\|\rho(u) - \rho(u^k)\|^2 + \sigma\langle \nabla  G(\rho(u^k)), \rho(u) - \rho(u^k)\rangle  + \frac{1}{2 \tau}\|u-u^k\|^2,
\end{align*}
for an affine linear $\rho$. This is similar to the proposed algorithm but also contains the additional proximity term for $u-u^k$. The analysis we provide in this work could make it interesting to revisit \cite{frerix_proximal_2018} in the two-layer case. 

\textbf{Solving the subproblems via lifting.}
While the convergence analysis of our approach will make rather weak assumptions on the global quality of the solution used in each of the subproblems \cref{eq:aux}, we found our method to be particularly effective and successful if the (nonconvex) subproblems are solved to global optimality. This raises the question what types of functions allow to determine globally optimal solutions to such subproblems. 

A rather simple case occurs if the involved functions are separable or separable into blocks of few variables. In these situations we can apply exhaustive search and branch-and bound algorithms to each block separately \cite{kearfott_rigorous_2013,hansen_global_2004}.

More interesting for imaging tasks is the class of functions where the subproblems can be solved by \textit{functional lifting}. It was shown in \cite{pock_global_2010,chambolle_convex_2012} that free discontinuity-type energies, in particular,
\begin{equation}\label{eq:cont_lift}
E(u) = \int_\Omega \nu(x,u(x),\nabla u(x)) \ dx , \quad u \in W^{1,1}(\R^n,\R)
\end{equation}
with $\nu:\R \times \R \times \R^d \to \R$ being continuous in its second argument, and convex and continuous in its third argument, can be optimized globally by \textit{lifting} the problem into a higher dimensional space where it admits a convex representation. Recent works, e.g. \cite{mollenhoff_sublabel-accurate_2016,mollenhoff_sublabel-accurate_2017}, discuss how to discretize the continuous formulation accurately and return to the finite-dimensional setting of this work.

We therefore expect to be able to solve all nonconvex majorizers $E_{u^k}$ that are a discretization of \cref{eq:cont_lift} to (near)-global optimality, allowing us to consider highly non-trivial majorizing functions. 
Similar relaxation methods exist in the discrete community via graph cuts for Markov random fields, see  \cite{kolmogorov_what_2004,boykov_fast_2001-1,ishikawa_segmentation_1998} and the references therein.

\subsection{Organization of this work}
This work introduces an optimization algorithm for the sum of a function and a composite function, which iteratively minimizes a nonconvex majorizing function \cref{eq:aux}. The algorithm is detailed and discussed in \cref{sec:principle} and basic properties are discussed in the first part of  \cref{sec:convergence}. The second part of \cref{sec:convergence} then extends these basic properties to a global convergence under the K\L-property and uniqueness of $R$-minimizing solutions. Several generalizations and implementation details follow in \cref{sec:generalized}. Finally, \cref{sec:experiments} shows some promising numerical results on synthetic examples where the proposed algorithm is able to find better minima than competing first-order methods, while being much more efficient than methods from global optimization applied to the discussed problem class \cref{eq:problem}. We then close \cref{sec:experiments} with an application to depth super resolution from noisy time-of-flight data.

\section{The General Principle}\label{sec:principle}

Before we begin the formal introduction of the necessary context and provide convergence and basic properties in their full generality it is instructive to reduce the problem formulation to a very simple test case.

Let us consider the standard Jacobi-iteration:
\begin{equation}
u^{k+1} = D^{-1}(f-(A-D)u^k)
\end{equation}
which solves the linear equation $Au=f$ for  symmetric $A \in \R^{n \times n}$ whose diagonal is $D$. We can interpret this scheme as successively minimizing the function
\begin{equation}
E_{u^k}(u) =  \langle u,\frac{1}{2}Du+(A-D)u^k-f \rangle - \langle u^k,\frac{1}{2}(A-D)u^k \rangle,
\end{equation}
which is a majorizer to $E(u) = \frac{1}{2} \langle u,Au-f\rangle$, if $D-A$ is positive definite. 
Now we would like to solve the nonlinear equation system $A\rho(u) = f $ for some function $\rho:\R^m \to \R^n$. And we do the same as before and apply our previous majorizer to $\rho(u)$:
\begin{equation}\label{eq:jacobi_nonlin}
u^{k+1} = \argmin_u \langle \rho(u),\frac{1}{2}D\rho(u)+(A-D)\rho(u^k)-f\rangle - \langle \rho(u^k),\frac{1}{2}(A-D)\rho(u^k) \rangle .
\end{equation}
If $\rho$ is separable, then these problems can still be solved efficiently in each dimension, thereby iteratively solving $A\rho(u)=f$.
As we will see in more detail in \cref{ex:jacobi}, \cref{eq:jacobi_nonlin} is a particular instance of the algorithm we propose and study in this paper, yielding nonconvex majorizers that are still easy to minimize. 
While this illustrates the main idea of our algorithm, the situation becomes even more interesting if an additional regularization $R$ makes a substitution like $z=\rho(u)$ impossible. 

\subsection{The Algorithm}\label{sec:the_alg}
Now we are ready to formulate the algorithm in full generality.

We consider the task of minimizing functions $E:\R^n \to \R \cup \lbrace+ \infty \rbrace =: \overline{\R}$ and define the domain of $E$ by $\dom{E} =: \lbrace u \in \R^n \ | \ E(u) < \infty \rbrace$. We denote the closure of this domain by $\overline{\dom{E}}$. A function is proper if $\dom{E} \neq \emptyset$. We call a function lower semi-continuous if we have $\liminf_{u \to \bar{u}} E(u) \geq E(\bar{u})$ for all  $\bar{u} \in \dom{E}$. The distance of a vector $u \in \R^n$ to a subset $S$ of $\R^n$ is defined via $\operatorname{dist}(u,S) = \inf_{x \in S}||u-x||$. We denote the pre-image of a mapping $\rho$ on a set $S$ by $\rho^{-1}(S)$. A proper function is essentially smooth if its convex subdifferential $\partial h$ is locally bounded and single-valued on its domain \cite{bauschke_essential_2001} or equivalently if $\dom{\partial h} = \intdom{h} \neq \emptyset$ \cite[Thm 26.1]{rockafellar_convex_1970}.

We consider the optimization problem 
\begin{equation}\label{eq:the_problem}
\min_{u \in \Delta} E(u) = G(\rho(u)) + R(u)  ,
\end{equation}
minimizing the composite and additive model $E$ over a closed set defined via $\Delta = \rho^{-1}(C)$ for a closed convex set $C \subset \R^m$ with $\operatorname{int} C \neq \emptyset$. We employ a convex function $h$ that mirrors the geometry of the problem and mimics the behavior of $G$.
We make the following assumptions on these functions:

\textbf{Basic Assumptions:}
\begin{itemize}
	\item $h:\R^m \to \overline{\R}$ is a proper, lower semi-continuous,  convex function that is essentially smooth with $ \overline{\dom{h}} = C$, 
	\item $G:\R^m \to \overline{\R}$ is a proper, lower semi-continuous function with $\dom{h} \subset \dom{G}$, which is differentiable on $\intdom{h}$
	\item $R: \R^n \to \overline{\R}$ is a proper, lower semi-continuous function and $\dom{R} \cap \rho^{-1}(\intdom{h}) \neq \emptyset$.
	\item $\rho: \R^n \to \R^m$ is a continuous function.
\end{itemize}
Under these assumptions, $E$ is a proper, lower semi-continuous objective function.
We define the Bregman distance of two vectors $u \in \R^m$ and $v \in \intdom{h} \subset \R^m$ relative to the chosen function $h$ by
\begin{equation*}
D_h(u,v) = h(u) - h(v) - \langle \nabla h(v),u-v \rangle.
\end{equation*} 
and we set $D_h(u,v) = \infty$ if $v \notin \intdom{h}$.
We choose a step size $\tau >0$ to be discussed later, a starting vector $u^0 \in \rho^{-1}(\intdom{h})$, and then apply  the following iterative scheme: 
\importantBox{\textbf{Main Algorithm:}
\begin{equation}\label{eq:the_alg}
u^{k+1} \in \argmin_{u \in \R^n} \frac{1}{\tau}D_h(\rho(u),\rho(u^k)) + \langle \nabla G(\rho(u^k)),\rho(u)-\rho(u^k) \rangle + G(\rho(u^k)) + R(u) 
\end{equation} 
\vspace{0.125cm}
}
\vspace{0.25cm}
Dicussions of well-definedness and convergence will also follow later in \cref{sec:convergence}.
The use of a Bregman distance is an immediate generalization of the usual squared norms, e.g via $h(u) = \frac{1}{2}||u||_2^2$, which allows us a greater level of generality, as we will discuss later in \cref{sec:generalized}.  
\begin{example}\label{ex:jacobi}
Returning to the Jacobi example from before, we now see in particular that setting $G(v) = \frac{1}{2}\langle v,Av-f\rangle$, $R(u) = 0$ and $h(u) = \frac{1}{2}||u||_D^2$ exactly recovers the nonlinear Jacobi updates in \cref{eq:jacobi_nonlin}.
\end{example}
In practice this algorithm is applicable even if the subproblems \cref{eq:the_alg} can only be solved up to a local optimum. However it is especially interesting if \cref{eq:the_alg} can actually be solved globally. In our applications we mainly consider three interesting cases for this, although our theoretical analysis in later chapters is not necessarily limited to those. 

First, if $\rho$ and $R$ are Lipschitz and separable, in the sense that $\rho:\R^n \to \R^n$ can be written as $\rho(u) = (\rho_1(u_1),\dots,\rho_n(u_n))$ and $R:\R^n \to \overline{\R}$ can be written as $R(u) = \sum_{i=1}^n r_i(u_i)$, then \eqref{eq:the_alg} decomposes into one-dimensional subproblems for each $u_i$. We use separable $h(u) = \sum_{i=1}^m h_i(u_i)$, so that $D_{h_i}(u_i,v_i) = h_i(u_i)-h_i(v_i)-h_i'(v_i)(u_i-v_i)$ and find that the majorizer decouples so that
\begin{equation}
 	u_i^{k+1} \in \argmin_{u_i} \frac{1}{\tau}D_{h_i}(\rho_i(u_i),\rho_i(u_i^k)) + \frac{\partial G(\rho(u^k))}{\partial u_i}(\rho_i(u_i)-\rho_i(u_i^k)) + r_i(u_i).
\end{equation}
These univariate nonconvex problems can be solved very efficiently and in parallel by uniform grid searches or more elaborate exhaustive branch-and-bound strategies, due to the Lipschitz properties $R$ and $\rho$ whenever $R$ has a bounded domain. 

A particularly interesting and practically relevant case are energies of the form
\begin{equation}\label{eq:gen_prob}
E(u) = \sum_{i=1}^m F_i \left(\sum_{j=1}^n \rho_{ij}(u_j) \right) + \sum_{i=1}^n r_i(u_i),
\end{equation}
where we have $\rho_{ij}: \R \to \R$ and $r_i : \R \to \overline{\R}$, $F_i:\R \to \overline{\R}$ and we again assume a bounded domain.
These models appear naturally in several nonlinear regression tasks. But again, the problem can be decomposed into one-dimensional subproblems and we apply our algorithm, as the subproblems decouple if we set $G(v) = \sum_{i=1}^m F_i(\sum_{j=1}^n v_{ij})$ and $\rho = (\rho_{11},\dots,\rho_{mn})$.

Remarkably, both of the above examples still yield (near)-globally solvable subproblems, if the separable regularization is replaced by a suitable penalty on the gradient of the unknown. While such subproblems are nonconvex and non-separable they can still be solved efficiently with the lifting techniques discussed in the context of equation \cref{eq:cont_lift}. We detail these types of problems in \cref{sec:related_tv}.

\subsection{Special Cases}

We note several cases, where the method reduces to simpler approaches: First, if $\rho$ is the identity, then we immediately recover a non-composite problem, the setting of \cite{bolte_first_2018}. If $\rho$ is invertible, then we can minimize over $z$ with the regularizer $R(\rho^{-1}(z))$ and again recover a non-composite problem. Further, if $G$ is separable as well, then it would be easier to take the whole problem directly as a nonconvex majorizer, which would converge in a single iteration. 
If the regularizer $R$ is zero, then the algorithm works fine, yet we would like to highlight that it is possibly easier to solve the minimization over $G(v)$ first under the constraint of $v \in \rho^{-1}(\intdom{h})$ (for separable $\rho,h$ this would be an especially easy constraint), and then optimize $D_h(\rho(u),v^*)$.

	 \begin{figure}
	 	\includegraphics[width=0.49\textwidth]{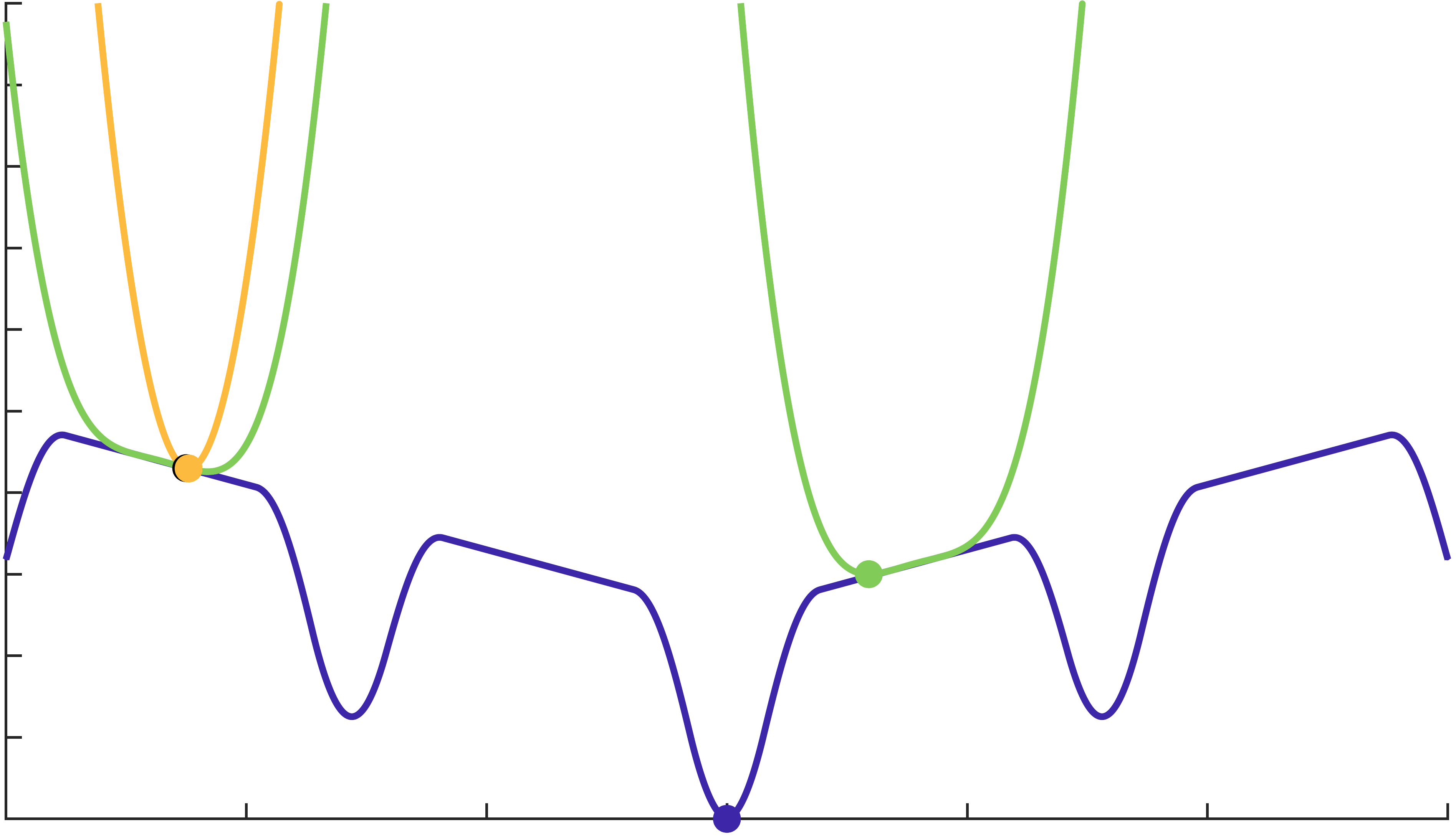}
	 	\includegraphics[width=0.49\textwidth]{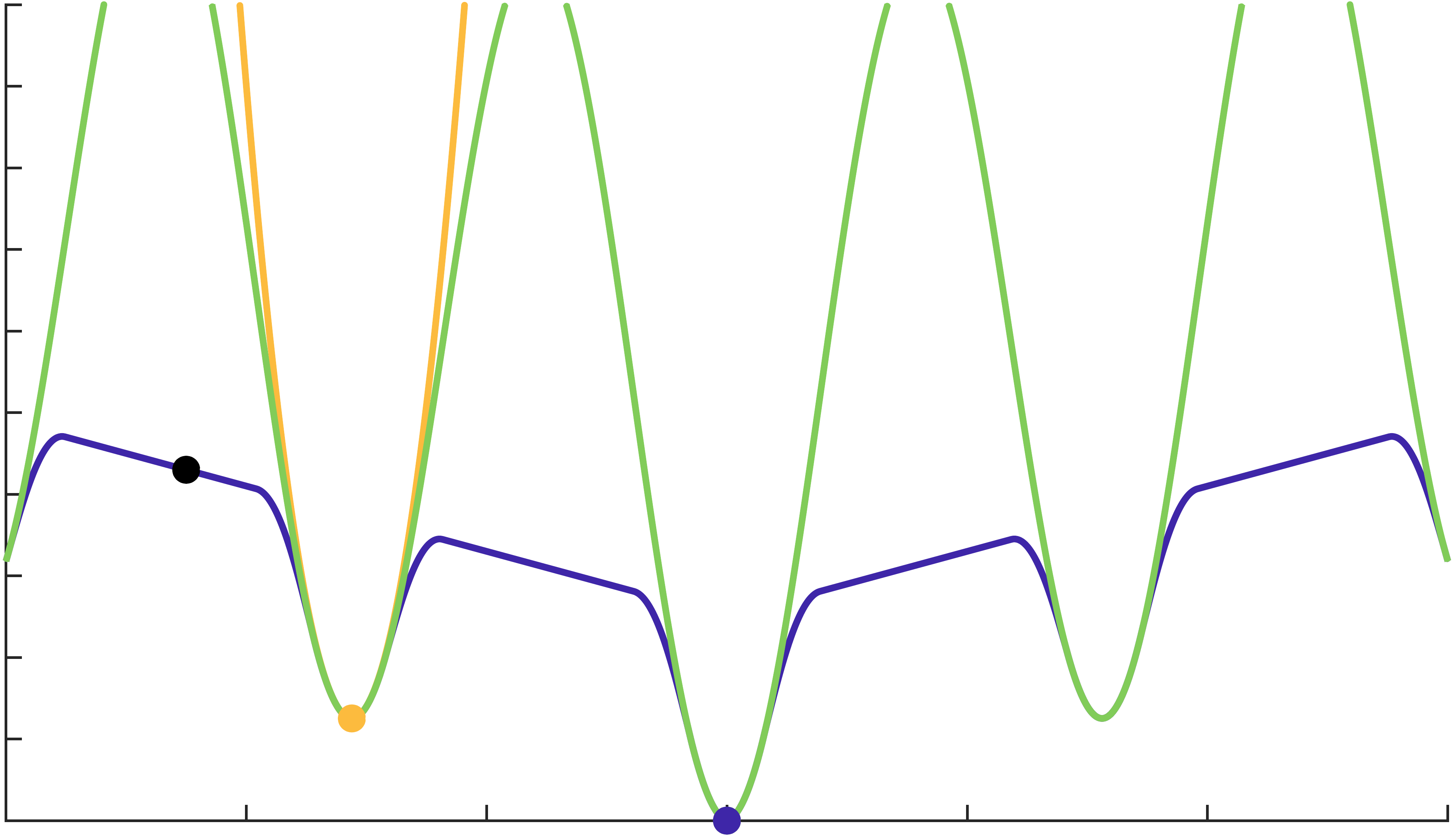} 
	 	\caption{Proximity relative to $\rho$ can be crucial during minimization. Initialization marked in black, global minimum in dark blue. Our majorizer, i.e. \cref{eq:the_alg} and its minimizer are marked in green. A majorizer that measures proximity relative to $u$, i.e. \cref{eq:rel_pl}, shown in yellow. The left figure shows a single step, the right figure shows the algorithm output and final majorizers.}\label{fig:ProxRelative}
	 \end{figure}

\subsection{Proximity relative to the inner function}
	%
	Unlike standard schemes, \cref{eq:the_alg} measures the proximity between $\rho(u)$ and $\rho(u^k)$, instead of $u$ and $u^k$ as in previous works on composite optimization  \cite{drusvyatskiy_efficiency_2016,drusvyatskiy_nonsmooth_2016, bonettini_variable_2016,ochs_non-smooth_2017}. However our choice, motivated by the nonlinear Jacobi example previously mentioned, is advantageous, whenever the subproblems can still be solved efficiently.
	
	The main advantage is the leverage we gain. By updating relative to $\rho$, we are able to directly apply smoothness properties and subsequent descent lemmas for $G$, easily finding a global majorizer in each step. Furthermore, our step size can now be chosen analytically independent of $\rho$ and all new iterates are feasible in the sense that $u^{k+1} \in \textnormal{dom }\rho$ and $\rho(u^{k+1}) \in \textnormal{dom } h $. 
	
	To make a more intuitive argument, we also note that penalizing the direct proximity between $u$ and $u^k$ of course limits the updates $u^{k+1}$ to a neighborhood of $u^k$. 
	If we are able to solve subproblems to global optimality, then limiting our updates in a local area seems unnecessary. If we penalize the proximity in $\rho$, then we only stay in a local area relative to $G$, which is necessary, as we linearized $G$. But otherwise we allow for arbitrarily large updates as long as $\rho(u)$ is similar to $\rho(u^k)$, which is of no issue, as we solve our subproblems globally.
	 By this approach we hope to find interesting stationary points globally and not just locally in a neighborhood around the starting vector. 
	 
	 \Cref{fig:ProxRelative} visualizes this behavior in 1D. Given $G$, a smooth version of $\min(u^2,\lambda)$, $\rho(u)=\sin(u)$ and $R(u)=\alpha|u|$ we majorize around the black mark using $h(u)=u^2$. We see that the proximity relative to $\rho$ is critical for reaching the global minimizer.
	
	It is quite instructive to compute both update steps for a linear composition example, i.e. $E(u) = F(Au)$ for $A\in\R^{m \times n}$. One can check that the updates relative to $\rho_{ij}(u_j):=a_{ij}u_j$ in \cref{eq:gen_prob} then correspond to a gradient descent with diagonal preconditioning, whereas the update in $u$ directly would correspond to standard gradient descent.

\section{Algorithm Discussion and Convergence}\label{sec:convergence}
In the following section we will analyze convergence properties of the proposed algorithm. We will specify the 
assumptions we make on $G$, discuss well-posedness of subproblems and give a descent lemma.  
We will then make further assumptions on tameness of the function and uniqueness of $R$-minimizing solutions to prove global convergence.

During this discussion we will move toward the exact structure of the main algorithms \cref{eq:the_alg} in three steps, the first two being variants, where we first only assume that the subproblems \cref{eq:the_alg} are solved 'sufficiently' and then only assume that the subproblems are solved 'sufficiently' to a stationary point. We do this to highlight precisely when global solutions to the subproblems are necessary and what advantages this confers; knowing that for some problems, solving the nonconvex subproblem to global optimality, might be too difficult.

\subsection{Basic Properties}
To find fixed step sizes for the algorithm we need to assume some bound on the change in the gradient of the function. An appropriate generalization of Lipschitz continuity that gives a bound "relative" to the chosen function $h$ \cite{bauschke_descent_2017,bolte_first_2018}, defines the following property: 
\begin{definition}[L-smooth adaptable]\label{def:lsmad}
	A proper, lower semi-continuous function $G:\R^m \to \overline{\R}$ is called L-smooth adaptable relative to a convex function $h$
	 if there exist $L>0$ so that $Lh-G$ is convex on $\intdom{h}$.	
\end{definition}

As a consequence of the L-smooth adaptability property we have the following descent inequality:
\begin{lemma}[{Descent Lemma, \cite[Lemma 1]{bauschke_descent_2017}}]\label{lem:LCdescent}
	If the proper lower semi-continuous function $G$ is L-smooth adaptable relative to an essentially smooth convex function $h$ so that $\dom{h} \subset \dom{G}$ and $G$ is differentiable on $\intdom{h}$, then
	\begin{equation*}
	 G(z)-G(w)-\langle \nabla G(w),z-w \rangle  \leq LD_h(z,w) \quad \forall z,w \in \intdom{h}.
	\end{equation*}
\end{lemma}
\begin{proof}
	$D_f(z,w)\geq 0 \ \forall z,w \in \R^n $ if and only if $f$ is convex. Hence $D_{Lh-G}(z,w)\geq 0$ which yields $D_G(z,w) \leq L D_h(z,w)$, due to the additivity of the Bregman distance.
\end{proof}

We define the subproblem energy in the following by
\begin{equation}\label{eq:maj}
E_{u^k}(u) = \frac{1}{\tau}D_h(\rho(u),\rho(u^k)) +  G(\rho(u^k)) + \langle \nabla G(\rho(u^k)),\rho(u)-\rho(u^k) \rangle + R(u).
\end{equation}
for some $u^k$ so that $\rho(u^k) \in \intdom{h}$.
For the first part of this section, we now make the following assumptions:

{ \bf Assumptions A:}
\begin{itemize}
	\item Our basic assumptions from \cref{sec:the_alg} hold,
	\item $E$ is bounded from below,
	\item $E$ is coercive,
	\item $Lh-G$ is convex on $\intdom{h}$,
	\item Every iteration is solved sufficiently, so that $E_{u^k}(u^{k+1}) \leq E_{u^k}(u^k) = E(u^k)$ and  $u^{k+1} \in \rho^{-1}(\intdom{h})$.
\end{itemize} 
Under these assumptions we will discuss the validity of \cref{eq:maj} as a majorizer for $E$ and the well-posedness of the minimization of $E_{u^k}(u)$. Note that in practice, we often gain coerciveness by considering functions $E$ with a bounded domain. The most important assumption here is the smoothness assumption on $G$ given by its $L$-smooth adaptability. The fifth assumption is very general and holds, for example, already when each sub-problem is solved only by finite sampling. We also need the technical assumption that $\rho(u^{k+1}) \in \intdom{h}$, which holds e.g. if $\dom{h} = \R^m$ or if $R$ is convex and $\rho$ is continuously differentiable. It can also be guaranteed through a set of constraint qualifications arising from \cite[10.6,10.9]{rockafellar_variational_2009}, yet we omit further discussion of this issue.

\begin{lemma}[Majorization Property]
	Under the assumptions A, given some $\tau \leq \frac{1}{L}$ and $\rho(u^k) \in \intdom{h}$, $E_{u^k}(u)$ is a majorizer of $E$, i.e. it fulfills the properties
	\vspace{5pt}
	\begin{itemize}[leftmargin=40pt]
		\item $E_{u^k}(u) \geq E(u) \quad \forall u\in \R^n$
		\item $E_{u^k}(u^k) = E(u^k)$.
	\end{itemize} 
\end{lemma} 
\begin{proof}
	A quick computation shows that $E_{u^k}(u^k)$ is equal to $E(u^k)$, as the Bregman distance $D_h(x,y)$ is zero if $x=y$ and $x,y \in \intdom{h}$.
	Now, using \cref{lem:LCdescent} for $G$ and inserting arbitrary $\rho(u) \in \intdom{h}$ and $\rho(u^k) \in \intdom{h}$ gives
	\begin{equation*}
	G(\rho(u)) \leq G(\rho(u^k)) + \langle \nabla G(\rho(u^k)),\rho(u)-\rho(u^k) \rangle + L D_h(\rho(u),\rho(u^k)).
	\end{equation*}
	On the other hand, we have, due to $\frac{1}{\tau} \geq L$,
	\begin{equation*}
	E_{u^k}(u) \geq LD_h(\rho(u),\rho(u^k)) + G(\rho(u^k)) + \langle \nabla G(\rho(u^k)),\rho(u)-\rho(u^k) \rangle + R(u).
	\end{equation*}
	Combining both inequalities gives the desired result for any $ u \in \R^n \textnormal{ s.t. }\rho(u) \in \intdom{h}$. If $\rho(u) \notin \intdom{h}$, then $E_{u^k}(u) = \infty$, so that the inequality is trivially fulfilled.
\end{proof}

Now let us consider the set of minimizers of $E_{u^k}$ \cref{eq:maj}:
\begin{equation}\label{eq:solution_set}
	M_\tau(u^k) = \left\lbrace \bar{u} \in \Delta \ | \ \bar{u} \in \argmin_u E_{u^k}(u) \right\rbrace 
\end{equation}

\begin{lemma}\label{lem:well_posed}
	Under the assumptions A, the set $M_\tau(u^k)$ is non-empty and compact for any $u^k \in \text{dom }E$ if $\tau \leq \frac{1}{L}$.
\end{lemma}
\begin{proof}
	We already know that $E$ is coercive. However, as $E_{u^k}$ is a majorizer for $\tau \leq \frac{1}{L}$, $E_{u^k}(u) \geq E(u)$, it is itself coercive. Furthermore $E$ and hence $E_{u^k}$ is bounded from below. As a result $E_{u^k}(u)$ is lower semi-continuous and proper with bounded level sets. \cite[Theorem 1.9]{rockafellar_variational_2009} now guarantees that $\inf E_{u^k}$ is finite and that the set $M_\tau(u^k)$ is non-empty and compact. 
\end{proof}

\subsection{Descent Properties}

As usual for majorization-minimization algorithms, we gain a monotone decrease in the objective function:
\begin{lemma}[Descent Lemma]\label{lem:descent}
	If the assumptions A hold, then the energy $E$ is monotonically decreasing for all iterates $u^k$ if $\tau < \frac{1}{L}$ is chosen as step size. The descent rate is 
	\begin{equation}\label{eq:descent}
	E(u^{k+1})-E(u^k) \leq -\frac{1-\tau L}{\tau}D_h(\rho(u^{k+1}),\rho(u^k)))
.	\end{equation}
\end{lemma}
\begin{proof}
	Using \cref{lem:LCdescent} for $G$ we insert $\rho(u^{k+1})$ and $\rho(u^k)$ so that 
\begin{equation*}
G(\rho(u^{k+1}))-G(\rho(u^k)) \leq \langle \nabla G(\rho(u^k)),\rho(u^{k+1})-\rho(u^k)\rangle + LD_h(\rho(u^{k+1}),\rho(u^k)).
\end{equation*}
Furthermore, because every iteration is solved sufficiently, we know that 
\begin{equation*}
\frac{1}{\tau}D_h(\rho(u^{k+1}),\rho(u^k)) + \langle \nabla G(\rho(u^k)),\rho(u^{k+1})-\rho(u^k)\rangle + R(u^{k+1}) +G(\rho(u^k))\leq G(\rho(u^k)) +R(u^k) .
\end{equation*}
Adding both inequalities yields
\begin{equation*}
G(\rho(u^{k+1}))-G(\rho(u^k)) + R(u^{k+1})-R(u^k) \leq (L-\frac{1}{\tau})D_h(\rho(u^{k+1}),\rho(u^k)),
\end{equation*}
which is the desired result. Due to the convexity of $h$, the right-hand side is always non-negative if $\tau < \frac{1}{L}$.
\end{proof}

From this we gain convergence in function values, subsequence convergence and some notion of convergence speed. For clarity of presentation we define the "outer" sequence $z^k := \rho(u^k)$. 
\begin{corollary}\label{cor:descent}
	
	Under the assumptions A, we see that for a step size of $\tau < \frac{1}{L}$, 
	\begin{itemize}[leftmargin=40pt]
		\item the sequence of function values $(E(u^k))_{k=1}^\infty$ converges to a limit $E^*$,
		\item $\lim_{k \to \infty} D_h(z^{k+1},z^k) = 0$,
		\item there exist converging subsequences $(u^{k_l})_{l=0}^\infty$ and $(z^{k_l})_{l=0}^\infty$,
		\item the sequence $\min_{1 \leq k \leq N} D_h(z^{k+1},z^k)$ converges to $0$ with order $\mathcal{O}(\frac{1}{N})$.
	\end{itemize}
\end{corollary}

\begin{proof}
	Summing both sides of the descent inequality in \cref{eq:descent} for $k=1,\dots,N$ and simplifying the expression gives
	\begin{equation*}
	\alpha \sum_{k=1}^{N} D_h(z^{k+1},z^k) \leq E(u^1) - E(u^N) \leq C
	\end{equation*}
	for $\alpha = \frac{1-\tau L}{\tau}>0$. 
	 $(E(u^k))_{k=1}^\infty$ is a monotone decreasing sequence, that is bounded as $-\infty < \inf E \leq E(u^k) \leq E(u^1)$ and thus converging. We gain the existence of converging subsequences $u^{k_l}$ due to these bounds and the lower semi-continuity and coercivity of $E$. The continuity of $\rho$ allows us to extend this to the existence of converging subsequences $z^{k_l}$. 
	 
	Concerning the convergence rate, define a minimal proximity over all iterates 
	\begin{equation*}
	\mu_N =: \min_{1\leq k \leq N} D_h(z^{k+1},z^k),
	\end{equation*}
	so that
	\begin{equation*}
	\sum_{k=1}^{N} D_h(z^{k+1},z^k) \leq \frac{E(u^1) - E(u^N)}{\alpha}
	\end{equation*}
	implies
	\begin{equation*}
	\mu_N \leq \frac{E(u^1) - E(u^N)}{\alpha N} \leq \frac{C}{ N} .
	\end{equation*}
	\phantom{1}
\end{proof}

For later use we define the set of all accumulation points of the sequence $u^k$, generated by our algorithm for a given starting vector $u^0$ as
\begin{equation}
\textnormal{accum}(u^0) = \lbrace u \in \R^n \ | \  \lim_{l \to \infty} u^{k_l} = u \text{ for a subsequence } u^{k_l} \textnormal{ of } u^k \rbrace.
\end{equation}
This set is non-empty as the sequence is bounded and closed as a set of limit points. 

\subsection{Convergence Properties}\label{sec:convergence_proof}
In this section we want to prove further statements of convergence. Up to now, we only gave assumptions on $G$ and on the relative minimization of the subproblems. For arbitrary $\rho$ and $R$ we can thus not expect a global convergence of the sequence of iterates, $u^k$, mostly because the accuracy up to which each subproblem is solved has not been specified yet. However, even if the subproblems are solved exactly, we need to assume some algebraic properties of $E$.

We will first give an appropriate optimality condition for our subproblems and specify the algebraic notion of 'tameness' discussed previously. Under these assumptions we will show a global convergence of the sequence $z^k=\rho(u^k)$ and convergence to critical points.
This is a natural convergence result as the distance of successive iterates is only measured relative to $D_h(\rho(u),\rho(u^k))$. It can be a conscious modeling choice to allow several equivalent critical points $u^*$ to be found by a single run of the algorithm. However we will also see that the choice of regularizer $R$ directly controls a global convergence in $u^k$, if the subproblems are solved to global optimality.

From now on, we consider a limiting subgradient:
\begin{definition}[Subgradients {\cite[8.3]{rockafellar_variational_2009}}]
	A function $E:\R^n \to \overline{\R}$ has the subgradient $v \in \R^n$ at a point $\bar{u} \in \text{dom }E$, if
	\begin{equation*}
	\liminf_{\substack{u \to \bar{u},\\u \neq \bar{u}}}	\frac{E(u) - E(\bar{u}) + \langle v,u - \bar{u}\rangle}{||u-\bar{u}||} \geq 0
	\end{equation*}
	and we write $v \in \hat{\partial}E(\bar{u})$. $v$ is further an element of the limiting subgradient $\partial E(\bar{u})$ at $\bar{u}$ if  sequences exist so that $u^i \to \bar{u}$, while $E(u^i) \to E(\bar{u})$ and $v^i \to v$ for $v^i \in \hat{\partial}E(u^i)$. 
\end{definition}
Note that in our case there exists some $\bar{u}$ for every $u^k \in \rho^{-1}(\intdom{h})$ so that $\partial E_{u^{k}}(\bar{u})$ is non-empty due to \cref{lem:well_posed} and Fermat's rule \cite[10.1]{rockafellar_variational_2009}. 
Rockafellar's optimality condition for limiting subgradients over a set \cite[8.15]{rockafellar_variational_2009} is a necessary condition for local minima. For our needs we consider the following version:
\begin{lemma}[Optimality Condition]\label{thm:fermat}
	If assumptions A hold and $\rho$ is continuously differentiable, then a local minimum of \cref{eq:the_problem} at $\bar{u} \in \rho^{-1}(\intdom{h})$ implies that
	\begin{equation}
	-J_\rho(\bar{u})^*\nabla G(\rho(\bar{u})) \in \partial R(\bar{u}),
	\end{equation}
\end{lemma}
\begin{proof}
	This follows from \cite[8.15]{rockafellar_variational_2009} and \cite[8.8]{rockafellar_variational_2009}, as the constraint $u \in \rho^{-1}(\dom{h})$ is not active for $\bar{u} \in \rho^{-1}(\intdom{h})$.
\end{proof}
We further call the set of all points $\bar{u}$ that fulfill this condition $\textnormal{crit} E$. 

However considering just the subdifferential of arbitrary functions leaves too many pathological cases for successful analysis of global convergence properties \cite{drusvyatskiy_slope_2013}. We thus follow recent literature on nonconvex optimization and consider functions that further satisfy the Kurdyka-\L ojasiewicz property:
\begin{definition}[Nonsmooth Kurdyka-\L ojasiewicz property \cite{bolte_lojasiewicz_2007}]
	For a proper and lower semi-continuous function $E:\R^n \to \overline{\R}$, we define its local Kurdyka-\L ojasiewicz property (K\L) 
	at a point $\bar{u} \in \textnormal{dom }E$ by the attribution that there exist $\eta > 0$, $\varphi \in C^0[0,\eta) \cap C^1(0,\eta)$ with $\varphi(0)=0$, $\varphi$ concave, $\varphi'>0$ and a neighborhood $U(\bar{u})$, so that
	\begin{equation*}
	\varphi'\left(E(u)-E(\bar{u})\right) \textnormal{dist}\left(0,\partial E(u)\right)  \geq 1,
	\end{equation*}
	for all $u \in U(\bar{u})$ with $E(\bar{u}) < E(u) < E(\bar{u})+\eta$.
\end{definition}
If $E$ is for example semi-algebraic, then it satisfies the K\L -property at any $\bar{u} \in \text{dom }E$. 
A proper semi-algebraic function $E:\R^n \to \R$ has a finite number of critical points \cite{drusvyatskiy_slope_2013}.
We note that any function definable in an o-minimal structure satisfies the K\L -property \cite{bolte_clarke_2007}.
Further, the property can be uniformized to yield
\begin{lemma}[{\cite[Lemma 6]{bolte_proximal_2014}}]\label{lem:KL}
	Let $\Omega$ be a compact set and consider a proper, lower semi-continuous function $E:\R^n \to \overline{\R}$. If $E$ is constant on $\Omega$ and satisfies the K\L -property at every point in $\Omega$, then there exist $\varepsilon > 0$, $\eta > 0$, $\varphi \in C^0[0,\eta) \cap C^1(0,\eta)$ with $\varphi(0)=0$, $\varphi$ concave, $\varphi'>0$ such that for all $\bar{u}$ in $\Omega$ the uniformized K\L -property,
	\begin{equation*}
	\varphi'\left(E(u)-E(\bar{u})\right) \textnormal{dist}\left(0,\partial E(u)\right) \geq 1,
	\end{equation*}
	holds for all $u \in \R^n$ with $dist(u,\Omega)< \varepsilon$ and $E(\bar{u}) < E(u) < E(\bar{u})+\eta$.
\end{lemma}


Now we are ready to collect our set of assumptions.

{ \bf Assumptions B:}
\begin{itemize}
	\item The function $E$ is continuous on its domain and satisfies the K\L-property at every point in the set $\textnormal{accum}(u^0)$,
	\item $h$ is $m$-strongly convex on $\intdom{h}$, 
	\item $\rho:\R^n \to \R^m$ is continuously differentiable,
	\item $\frac{1}{\tau}\nabla h- \nabla G$ is locally Lipschitz continuous on $\intdom{h}$,
	\item Given the set $Z = (z^k)_{k=1}^\infty$, we require that $\bar{Z} \subset \intdom{h}$
	\item every iteration satisfies $0 \in \partial E_{u^k}(u^{k+1})$ and uses a step size $\tau < \frac{1}{L}$.	
\end{itemize}
These extended assumptions now allow us to prove the following statements:

\begin{enumerate}
	\item The sequence of $z^k=\rho(u^k)$ converges globally to a value $z^*$.
	\item All accumulation points of subsequences $(u^{k_l})_{i=1}^\infty$ are stationary points of $E$.
	\item There is a correspondence between the limit point $z^*$ and the accumulation points of the iterates, given by $z^* = \rho(u^*)$ for all $ u^* \in \text{accum}(u^0)$. 
\end{enumerate}


We will see in the proof that the most demanding properties in assumptions B are used to prove a bound on the slope of iterates, i.e. the inequality $\textnormal{dist}\left(0,\partial E(u^{k+1})\right) \leq c ||z^{k+1}-z^k||$ for some $c>0$. Norm convergence to limit points lying on the boundary of $\dom{h}$ from arguments involving the K\L-property is problematic, as the essential smoothness of $h$ implies that $||\nabla h(y^k)|| \to \infty$ for $y^k \to y^* \in (\dom{h}\setminus\intdom{h})$ \cite{bauschke_essential_2001} so that there will be no fixed bound $c$. By requiring $\bar{Z} \subset \intdom{h}$ we strengthen the assumption of $z^k \in \intdom{h}$ from assumption A to the assumption that any prospective limit point $z^*$ will also fulfill $z^* \in \intdom{h}$.
In comparison to \cite{bolte_first_2018}, the assumption $\dom{h} = \R^m$ given therein is a straightforward implication of our more technical statement. Our assumption on the continuity of $E$ and replaces their assumption that $E_{z^k}(z^{k+1})\leq E_{z^k}(z) \, \forall z \in \intdom{h}$ in this subsection.

The ingredients of our proof follow recent literature, e.g. \cite{attouch_convergence_2013,bolte_first_2018}, however special care has to be taken as all estimates of slope and objective value are only relative to the outer sequence of $z^k=\rho(u^k)$.
\begin{lemma}[Slope bound]\label{lem:slope}
	If the assumptions A and B hold, then $c<\infty$ exists, so that
	\begin{equation}\label{eq:slope_bound}
	\textnormal{dist}(0,\partial E(u^{k+1})) \leq c ||z^{k+1}-z^k|| \quad \forall k\in \N.
	\end{equation}

\end{lemma}
\begin{proof}
	Consider the optimality condition of the update equation, as all subproblems are solved exactly:
	\begin{equation*}
	0 \in \partial R(u^{k+1}) + (J_\rho(u^{k+1}))^* \left(\nabla G(\rho(u^k))+\frac{1}{\tau}\nabla h(\rho(u^{k+1})) -\frac{1}{\tau}\nabla h(\rho(u^{k})) \right)
	\end{equation*}
	and reformulate to 
	\begin{align*}
	 (J_\rho(u^{k+1}))^* \biggl( &\nabla G(\rho(u^{k+1})) -\nabla G(\rho(u^k)) \\ &-\frac{1}{\tau}\nabla h(\rho(u^{k+1})) +\frac{1}{\tau}\nabla h(\rho(u^{k})) \biggr) \in \partial R(u^{k+1}) + (J_\rho(u^{k+1}))^* \left(\nabla G(\rho(u^{k+1})\right).
	\end{align*}
	Now we see that the left hand side is an element of $\partial E(u^{k+1})$, so that we can estimate its norm by
	\begin{equation*}
	\textnormal{dist}(0,\partial E(u^{k+1})) \leq \left\Vert J_\rho(u^{k+1})^*\right\Vert_{\operatorname{op}} \left\Vert\nabla \left(\frac{1}{\tau}h-G\right)(z^{k+1})-\nabla \left(\frac{1}{\tau}h-G\right)(z^{k})\right\Vert,
	\end{equation*}
	where we have denoted the induced operator norm  of $||\cdot||$ by $||\cdot||_{\operatorname{op}}$.
	By \cref{lem:descent} and the coerciveness of $E$, we know that $(z^k)_{k=1}^\infty$ is a compact subset of $\dom{h}$. Assumption B then guarantees that the sequence is further contained in $\intdom{h}$. This allows us to extend the local Lipschitz continuity of $\frac{1}{\tau}h-G$, also from assumption B, to Lipschitz continuity on this compact set. 
	Furthermore, we assumed $\rho \in C^1(\R^n,\R^m)$ which implies that its Jacobian $||J_\rho(u)^*||_{\operatorname{op}}$ is also Lipschitz on the compact set $\bar{U}$ for $U = (u^k)_{k=1}^\infty$. These properties allow us to find a fixed constant $c$ so that the inequality \cref{eq:slope_bound} holds for all $k \in \N$.
\end{proof}

Before we now come to the main theorem, we first collect a few properties of the set $\textnormal{accum}(u^0)$, that will allow the application of \cref{lem:KL}.

\begin{lemma}\label{lem:KLprep}
$E$ is constant and finite on $\textnormal{accum}(u^0)$, i.e. $E(u) = E(v) < \infty \ \forall u,v \in \textnormal{accum}(u^0)$ and we have
\begin{equation}\label{eq:distance_to_stat}
\lim_{k \to \infty} \textnormal{dist}(u^k,\textnormal{accum}(u^0)) = 0.
\end{equation}	
\end{lemma}
\begin{proof}
	$(u_k)_{k=1}^\infty$ is bounded due to \cref{cor:descent}. We choose a subsequence $(u^{k_l})$ with $\lim_{l\to\infty}u^{k_l}=u^*$. 
	From the continuity of $E$ on $(u^{k_l})$ we infer $\lim_{l \to \infty} E(u^{k_l}) = E(u^*)$.
	We further know from \cref{cor:descent} that the sequence of function values itself is convergent to some value $E^*$, so that $E$ is finite and constant on all these limit points.
	\Cref{eq:distance_to_stat} is true for all bounded sequences.
\end{proof}

The following proof is now a slight adaptation of usual strategies for convergence under the K\L -property \cite{attouch_convergence_2013} or \cite{bolte_first_2018,ochs_unifying_2016}, with the difference that we apply the K\L-property to the set $\textnormal{accum}(u^0)$ instead of the set of critical points, which nevertheless fulfills $E(u^*) < E(u^k) < E(u^*) +\eta$ for any $u^* \in \textnormal{accum}(u^0)$ as required in \cref{lem:KL}, due to the monotone descent of the algorithm. We then apply our previous results and find a global convergence in $z^k = \rho(u^k)$.

\begin{theorem}[Global Convergence]\label{thm:global}
	Under the assumptions A and B, the sequence $z^k$ either has finite length, $\sum_{k=1}^\infty ||z^{k+1}-z^k|| < \infty$, and converges to a limit $z^* \in \intdom{h}$, or can be terminated after a finite number of steps.
\end{theorem}
\begin{proof}
	 To apply the K\L-property, we need to verify that $E(u^*)< E(u^k)$ for all indices $k$ that we consider and accumulation points $u^* \in \textnormal{accum}(u^0)$. \cref{lem:descent} shows that $E(u^{k+1}) \leq E(u^k)$, due to the convexity of $h$. Now if for some index $l$ we have $E(u^k) = E(u^*)$, then the monotonicity of the sequence of objective values implies $E(u^{k+1}) = E(u^k)$ for the next index $k+1$. Together with \cref{lem:descent} this implies $D_h(z^{k+1},z^k) = 0$ and by the strong convexity of the assumptions B, $z^{k+1} = z^k$. Furthermore it is possible that iterates fulfill $z^{k+1} = z^k$ without fulfilling $E(u^{k+1}) = E(u^k)$, as $\rho$ is not bijective. In both cases, the algorithm can be terminated, as \cref{lem:slope} implies that $0 \in \partial E(u^{k+1})$. 
	 
	 Now, conversely, assume that the algorithm does not terminate after a finite number steps. We may then choose $l \in \N$ large enough so that both $E^* < E(u^l)<E^*+\eta$ and $\text{dist}(u^l,\textnormal{accum}(u^0))<\varepsilon$ are fulfilled. The positive constants $\varepsilon$ and $\eta$ are the ones required by the K\L-property of $E$ w.r.t to the set $\text{accum}(u^0)$. On this set, $E$ is constant and finite, as discussed in \cref{lem:KLprep}.
	From \cref{lem:KL}, we then find that
	\begin{equation*}
	\varphi'\left(E(u^k)-E^*\right)\textnormal{dist}\left(0,\partial E(u^k)\right) \geq 1
	\end{equation*}
	holds for any accumulation point $u^* \in \textnormal{accum}(u^0)$, as $E(u^*)=E^*$ and for all $u^k$ with $k>l$. Now we can apply \cref{lem:slope} to find that 
	\begin{equation}\label{eq:sloping}
	\varphi'\left(E(u^k)-E^*\right) \geq \frac{1}{c||z^k-z^{k-1}||}.
	\end{equation}
	Further, we can consider the descent from \cref{lem:descent} and apply that $h$ is $m$-strongly convex to obtain
	\begin{equation}\label{eq:values}
	E(u^k)-E(u^{k+1}) \geq \frac{1-\tau L}{\tau}D_h(z^{k+1},z^k) \geq \frac{m(1-\tau L)}{2\tau}||z^{k+1}-z^k||^2.
	\end{equation}
	Analogously to \cite{attouch_convergence_2013,bolte_first_2018}, we use the concavity of $\varphi$ to analyze the difference of function values in $\varphi$:
	\begin{align*}
	\Delta_{k,k+1} &=:\varphi\left(E(u^k)-E^*\right) - \varphi\left(E(u^{k+1})-E^*\right) \\
	&\geq \varphi'\left(E(u^k)-E^*\right) \left( E(u^{k})-E(u^{k+1}) \right).
	\end{align*}
	Inserting \cref{eq:sloping} and \cref{eq:values} and denoting constant terms by $c'$ we gain
	\begin{equation*}
	\Delta_{k,k+1} \geq \frac{||z^{k+1}-z^k||^2}{c'||z^k-z^{k-1}||} .
	\end{equation*}
	 Now we are entirely in the setting of \cite[Theorem 6.2]{bolte_first_2018} and likewise reformulate to 
	 \begin{equation*}
		 2\sqrt{||z^k-z^{k-1}|| c'\Delta_{k,k+1}} \geq 2||z^{k+1}-z^k||
	 \end{equation*}
	 and use the inequality $2\sqrt{ab}\leq a+b$ to gain
	 \begin{equation*}
	 2 ||z^{k+1}-z^k|| \leq ||z^k -z^{k-1}|| + c'\Delta_{k,k+1}.
	 \end{equation*}
	 Summing these inequalities for $k=l+1,\dots,n$, then yields
	 \begin{align*}
	 2 \sum_{k=l+1}^n ||z^{k+1}-z^k|| \leq& \sum_{k=l+1}^n ||z^{k}-z^{k-1}||+c'\sum_{k=l+1}^n \Delta_{k,k+1} \\
	 =&  \sum_{k=l+1}^n ||z^{k+1}-z^{k}||-||z^{n+1}-z^n||+||z^{l+1}-z^l|| + c'\sum_{k=l+1}^n \Delta_{k,k+1}\\
	 \leq&  \sum_{k=l+1}^n ||z^{k+1}-z^{k}|| +||z^{l+1}-z^l|| +c'\Delta_{l+1,n+1},
	 \end{align*}
	 where the last inequality is gained by telescoping all $\Delta$. Reinserting the definition of $\Delta_{l+1,n+1}$ and simplifying then results in
	 \begin{equation*}
	 \sum_{k=l+1}^n ||z^{k+1}-z^k|| \leq ||z^{l+1}-z^l|| + c' \varphi(z^{l}-z^{l+1})-c'\varphi(z^{n}-z^{n+1}) < \infty
	 \end{equation*}
	 As $\varphi$ is positive this implies that the whole sequence $z^k$ is a Cauchy sequence and converges due to metric completeness. 
\end{proof}

\begin{remark}[Strong convexity of $h$]
	The $m$-strong convexity might seem like a limiting assumption, yet it is always possible to construct a function $\tilde{h}$ that fulfills this property, if $G$ is L-smooth adaptable for some $h$ (see also the related discussion in \cite{bolte_first_2018}). First if $Lh-G$ is convex, then $L(h+w)-(G+Lw)$ is also convex for any function $w$, especially $w = \frac{m}{2}||\cdot||^2$. Define $\tilde{h} = h+w$, $\tilde{G} = G+Lw$, and $\tilde{R} = R-L (w \circ \rho)$. Now the new energy $\tilde{G} \circ \rho + \tilde{R}$ is equal to the old formulation, but the pair $(\tilde{G},\tilde{h})$ is convex with $\tilde{h}$ being $m$-strongly convex. However we remark that the new function $\tilde{h}$ might make it more difficult to solve the resulting subproblem. 
\end{remark}

From the convergence result of \cref{thm:global} on the sequence $z^k$, we can return to $u^k$:
\begin{corollary}[Convergence to critical points]\label{cor:orbits}
	All accumulation points $u^* \in \textnormal{accum}E$ of $(u^k)_{k=1}^\infty$ are stationary points of $E$, i.e. $u^* \in \textnormal{crit}E$ and belong to the same outer sequence $z^k$ so that $z^*=\rho(u^*)$. 
\end{corollary}

\begin{proof}
	Combining the bound on the slope in \cref{lem:slope} and  global convergence of $z^k$'s from \cref{thm:global} we immediately see that $\textnormal{dist}(0,\partial E(u^{k+1})) \to 0$ as $k \to \infty$. Furthermore, we know that $E(u^k) \to E(u^*)=E^*$ so  that all subsequences of $u^k$ fulfill the definition of the limiting subdifferential and $0 \in \partial E(u^*)$.  	
	We know also that $\lim_{k \to \infty} z^k = \lim_{k \to \infty} \rho(u^k) = z^*$. Let $u^* \in \textnormal{accum}(u^0)$ be arbitrary with the sequence by $u^{k_l} \to u^*$. Due to continuity of $\rho$ we have $\rho(u^*) = \lim_{l \to \infty} \rho(u^{k_l})  = z^*$.
\end{proof}
This result shows the connection between the 'auxiliary' outer sequence of gradient steps $z^k$, which converges globally, due to the K\L-property and the sequence of actual update steps $u^k$. The algorithm converges to a stationary point of $E$ and all accumulation points not only have the same value in $E$, but also in $G \circ \rho$, as $G(\rho(u^*)) = G(z^*) \ \forall u^* \in \textnormal{accum}(u^0)$.

\subsection{Global Convergence of the inner sequence}
A necessary consequence of the previous subsection is that all accumulation points have an equal value $R^*$ in $R$, hence are elements of the set $S = \lbrace u \in \R^n \ | \ R(u)=R^*,\ \rho(u) = z^*\rbrace$. Naturally, if this set is a singleton, then the subsequence convergence of the sequence $(u^k)_{k=1}^\infty$ extends to global convergence.
The cardinality of the set $C$ is however difficult to check a-priori. Nevertheless it turns out that if we finally also assume that the nonconvex subproblems are solved globally, then the convergence result follows from the familiar notion of uniqueness of $R$-minimizing solutions. We further remark that the assumption of global solutions to subproblems also allows us to weaken the continuity assumption made in Assumption B to lower semi-continuity of $E$.

Let us define $R$-minimizing in the following way, as given for example in \cite[Def 3.24]{scherzer_variational_2009}:  
\begin{definition}
	A vector $u^* \in \R^n$ is called $R$-minimizing with respect to a solution set $\{u\in \R^n ~|~ F(u)=v\}$ of an operator $F:\R^n \to \R^m$ and a vector $v \in \operatorname{Im}(F)$, if $F(u^*) = v$ and 
	\begin{equation*}
	R(u^*) \in  \min \left\lbrace R(u) \ | \ u \in \R^n, F(u) = v \right\rbrace .
	\end{equation*} 
\end{definition}

Now the uniqueness of such a vector directly corresponds to global convergence if the subproblems are solved globally:
\begin{theorem}[Global Convergence]
	If the subproblems are solved to global optimality, i.e. $u^{k+1}$ fulfills $E_{u^k}(u^{k+1})\leq E_{u^k}(u) \ \forall u \in \R^n$ and all assumptions hold, then 
	\begin{equation}
		u^* \in \textnormal{accum}(u^0)  ~ \Rightarrow ~ u^* \textnormal{ is $R$-minimizing w.r.t $\rho(u) = z^*$. }
	\end{equation}
	In particular, if the $R$-minimizing element w.r.t. $\rho(u) = z^*$ is unique, then the sequence of iterates $u^k$ converges globally.
\end{theorem}
\begin{proof}
	Let $u \in \{ u \in \R^n ~|~ \rho(u)= z^*\}$ and $u^* \in \text{accum}(u^0)$ be arbitrary. As $z^* \in \intdom{h}$, this implies to $u \in \rho^{-1}(\intdom{h})$.
	Rewriting the optimality assumption $E_{u^k}(u^{k+1})\leq E_{u^k}(u) \ $ $\forall u \in \rho^{-1}(\dom{h})$ results in the inequality
	\begin{equation*}
	R(u^{k+1})-R(u) + \frac{1}{\tau}D_h(\rho(u^{k+1}),\rho(u^k)) - \frac{1}{\tau}D_h(\rho(u),\rho(u^k)) + \langle \nabla G(\rho(u^k)),\rho(u^{k+1})-\rho(u) \rangle \leq 0 .
	\end{equation*}
	Taking the limit of $l \to \infty$ for a subsequence $u^{k_l} \to u^*$ with the knowledge that $\lim_{l \to \infty} \rho(u^{k_l}) = z^*$ and $\rho(u) = z^*$ by assumption, then reveals that 
$R(u^*) \leq R(u)$, showing that $u^* \in \text{accum}(u^0)$ is an $R$-minimizing solution to $\rho(u) = z^*$. If in particular, the set of $R$-minimizing solutions is already a singleton, then the result follows, as the set $\text{accum}(u^0)$ nonempty due to \cref{cor:descent}.
\end{proof}

\begin{example}
	As an example, consider a simple periodic function $\rho:\R^n \to \R^n$, $\rho(u) = (\sin(u_1),\dots,\sin(u_n))$ and $R = ||\cdot||_p$ for $p >0$. The $R$-minimizing solution to $\rho(u) = z$ is then unique for any $z \in \text{Im}(\rho)$. To see this consider that $\sin(x)$ is bijective on $[-\frac{\pi}{2},\frac{\pi}{2}]$. For any level set $z_i \in [-1,1]$ we can find a unique element $u_i$ in this interval $[-\frac{\pi}{2},\frac{\pi}{2}]$, so that $\sin(u_i) = z_i$. Due to the strict monotonicity of $||\cdot||_p$ on either $\R^+ $or $\R^-$, any other element $u_i$ that fulfills $\rho(u_i) = z_i$ must have a greater function value in $R$.  
\end{example}

\section{Implementation Details}\label{sec:generalized}

This section will focus on several interesting special cases of our general composite model \cref{eq:problem} and also discuss possible pairs $G,h$.

\subsection{Modeling}\label{sec:modeling}
For several implementation examples it will be convenient to be a bit more specific with our choices of $G,\rho$ and $R$. One example is the natural extension to several additive terms, 
\begin{equation}\label{eq:gen_problem}
E(u)  = \sum_{i=1}^m G_i \left(\sum_{j=1}^n \rho_{ij}(u_j) \right) + \sum_{j=1}^n r_j(u_j),
\end{equation}
which was already mentioned briefly in \cref{eq:gen_prob}.
This formulation is interesting due to its straightforward interpretation as a way to optimize a function with $n$ measurements of linear combinations of our $m$ variables. Hence the task relates to nonlinear regression models and imaging with nonlinear measurements. It is also a natural discretization of a general nonlinear integral operator as defined for example in \cite{precup_methods_2002-1,bardaro_nonlinear_2003}.

However, defining $G:\R^{m \times n} \to \R$, $G(v) = \sum_{i=1}^m G_i(\sum_{j=1}^n v_{ij})$ and $\rho: \R^n \to \R^{m \times n}$, defined component-wise by $\rho_{ij}(u_j)$, for univariate functions $G_i,\rho_{ij},r_j$, we see that this is just an instance of the general composite model
. The maximal generalization would be achieved by taking $\rho_{ij}:\R^q \to \R^p$, $G_i : \R^p \to \R$, $r_j:\R^q \to \R$, although in practice $q$ would have to be quite small if we wanted to solve the subproblems by an exhaustive search.


Writing out the majorizer to \cref{eq:gen_problem} with univariate functions under the assumption that $\sum_{i=1}^m L_i h_i-G_i$ is a convex function (as required for L-smooth adaptability, \cref{def:lsmad}) gives
\begin{equation}\label{eq:gen_maj}
E_{u^k}(u) = \sum_{i=1}^n D_{h_i}(\rho_{ij}(u_j),\rho_{ij}(u^k)) + \sum_{i=1}^n G_i'\left(\sum_{j=1}^m \rho_{ij}(u_j^k)\right) \sum_{j=1}^m \rho_{ij}(u_j) 
+\sum_{j=1}^n r_j(u_j),
\end{equation}
up to constant terms. This reveals that the majorization function is separable if each $h_i$ is chosen separable so that $h_i(u) = \sum_{j=1}^n h_{ij}(u_j)$, as the Bregman distance to these $h_i$ is then also separable and the summation over all $m$ parameters can be exchanged with the summation over all $n$ 'measurements'. The resulting $m$ independent 1D dimensional subproblems can then be solved efficiently.

\begin{remark}
	This generalization is not only interesting for regression-type problems, where the outer sum naturally sums over all samples and the inner sum over a superposition of parametrized functions, but also for any sort of problem where it would make sense to split a function into the composition of a function and a super-position of simpler functions. As an example consider the 1-dimensional polynomial problem
	\begin{equation}
    P(x) = \left( \sum_{i=0}^p a_i x^i-f \right) ^2 + \sum_{i=0}^q b_i x^i .
	\end{equation}
	While it would be natural to choose $\rho:\R \to \R, \rho(x) = \sum_{i=0}^p a_i x^i-f$, i.e the inner polynomial, another possibility would be to choose $\rho:\R \to \R^p, \rho(u) = (a_0,\dots,a_n x^n)$ and likewise to set $G:\R^p \to \R, G(v) = (\sum_{j=1}^p v_j  - f)^2$. A separable majorizer for this $G$ using \cref{eq:gen_maj} would lead to different subproblems than before.
\end{remark}

An interesting fact about the general composite model \cref{eq:problem} is that we are actually allowed a great deal of freedom, as both $G$ and $\rho$ can be nonconvex. It is possible to insert any invertible function $f$ and its inverse $f^{-1}$ on $\text{dom }\rho$ and solve the equivalent problem with $\tilde{G} = G \circ f$ and $\tilde{\rho} =  f^{-1} \circ \rho$. As an example, consider the following model
\begin{equation*}
E(u) = F \left(\prod_{j=1}^n g_j(u_j) \right),
\end{equation*}
where we have a product of parametrized functions $g_j:\R \to \R^+$. We can set $G:\R^n \to \R, G(v) = F\left(\exp(\sum_{j=1}^n v_j)\right)$ and $\rho_j : \R \to \R, \rho_j(u_j) = \log(g_j(u_j))$, and recover the additive superposition of parameters in \eqref{eq:gen_problem}.

We may freely use these possibilities due to \cref{cor:orbits}. The proposed algorithm converges to the set of stationary solutions of $E$. This result is independent of the actual decomposition of $E$ into $G,\rho,R$ as long as the chosen triple $G,\rho,R$ still fulfills all required conditions.

\subsection{Choices for the Bregman Distance}
Up to now, we always considered an abstract pair $(G,h)$ fulfilling the conditions that both functions are smooth and $Lh-G$ and $h$ are convex. Now we will detail several tangible instances of these functions.

The trivial case is present when $G$ is a concave function. We are then allowed to choose an arbitrary convex function $h$, as $-G$ is itself convex. A natural choice is then to choose $h$ as a linear function, as its Bregman distance then vanishes,
\begin{equation*}
D_h(u,v) = \langle h,u \rangle - \langle h,v\rangle - \langle h,u-v\rangle = 0.
\end{equation*}
The resulting majorizer,
\begin{equation}\label{eq:concave_maj}
E_{u^k}(u) = \langle \nabla G(\rho(u^k)),\rho(u)-\rho(u^k) \rangle + G(\rho(u^k)) + R(u),
\end{equation}
is an instance of iterative reweighting related to variants discussed in \cite{ochs_iteratively_2015,ochs_iterated_2013}. If $R$ is convex and $\rho$ is coordinate-wise convex, then the majorizer is even convex. When $R$ is a convex function and $\rho$ is an affine function, then we recover an instance of the difference of convex functions (DC) algorithm \cite{tao_convex_1997}. Note that for the second part of our analysis in \cref{sec:convergence} to hold, we need to choose $h$ strongly convex. We mention in passing that the results of \cref{cor:descent} also hold relative to the Bregman distance $D_{-G}(\cdot,\cdot)$.
%

The standard case is present when $G$ is $L$-smooth. We then choose $h=\frac{1}{2}||\cdot||^2$ and recover the usual Euclidean distance measure via $D_h(u,v) = \frac{1}{2}||u-v||^2$. Note that $G$ can be $L$-smooth without being convex, for example when considering a smooth truncated quadratic function \cite{artina_linearly_2013}.

However, even when $G$ is L-smooth, more advantageous functions $h$ might exist. Consider the function $G:\R^m \to \R$,
\begin{equation}\label{eq:l2}
G(v) = \frac{1}{2}||Av-f||^2 ,
\end{equation}
for a matrix $A \in \R^{p \times m}$. The function is $L$-smooth with $L=||A^TA||_{\operatorname{op}}$. However we can also inspect $Lh-G$ via its second derivative condition\footnote{We follow the notation that A is positive semi-definite if $A\succeq 0$.},
\begin{equation}
L \nabla^2 h(v) - A^TA \succeq 0 \quad \forall v \in \R^m.
\end{equation}
We could of course choose $h=G$, as \cref{eq:l2} is convex, thereby solving the original problem in each subproblem, but we are looking for functions $h$ so that the subproblems are easy to solve. Such a choice is presented by $h = \frac{1}{2}||\cdot||^2_D$ with a diagonal matrix $D$. Choosing $D$ so that $D-A^TA \succeq 0$ yields a diagonal preconditioning - we intrinsically find vector-valued step sizes by an appropriate choice of $h$.

An important and motivating property of the L-smooth adaptable property is however the inclusion of logarithmic functions, most prominently the Kullback-Leibler divergence as possible terms for $G$, even though this function is not globally $L$-smooth \cite{bauschke_descent_2017}. Consider the function $G:\R^m \to \overline{\R}$:
\begin{equation}
G(v) = \sum_{i=1}^p (Av)_i-f_i+f_i\log \left(\frac{f}{(Av)_i}\right) ,
\end{equation}
for $A \in \R^{p \times m}_+$ and $f \in \R^p_{++}$ and the set $C = [0,\infty)^m$. An appropriate function $h$ is given by $h(v) = -\sum_{i=1}^m \log(v_i)$. \cite[Lemma 7]{bauschke_descent_2017} reveals that the appropriate constant is $L=||f||_1$ so that $Lh-G$ is convex on $\intdom{h} = (0,\infty)^m$. This function and related 'entropy' functions are possible choices for $h$, yet, as now the domain of $h$ is strictly smaller than $\R^n$, one has to check, whether the energy fulfills $z^{k+1} \in \intdom{h}$ and $z^* \in \intdom{h}$ to guarantee well-posedness of the iterations and global convergence, respectively.

A general observation, see \cite{chouzenoux_variable_2014} or \cite[Example 33]{ochs_non-smooth_2017}, is that once we have gained a Bregman distance $D_h$ from $h$, we may actually use a whole family of functions $h_k$ as long as they majorize $h$ while being convex,
\begin{equation*}
h^k \in \lbrace h^k \textnormal{ essentially smooth}, \overline{\dom{h^k}}=C \ | \ h^k-h \text{ convex } \rbrace.
\end{equation*}
The induced Bregman distance then fulfills
\begin{equation*}
D_{h^k}(u,v) \geq D_h(u,v) \quad \forall u,v \in \R^n.
\end{equation*}
A specific instance of this observation is choosing $h$ first and then implementing adaptive step-sizes in this fashion by varying $h^k$ or the approach of \cite{chouzenoux_variable_2014} where a sequence  $h^k=\frac{1}{2}||\cdot||^2_{A^k}$  is constructed with symmetric positive definite matrices $A^k$.

This is of course only a short overview of possible pairs $(G,h)$, further examples can be found in \cite{bauschke_legendre_1997,bauschke_descent_2017,bolte_first_2018,ochs_non-smooth_2017,chouzenoux_variable_2014}.

\subsection{An example of a non-separable, solvable subproblem}\label{sec:related_tv}

This section will continue the discussion started in the introductory section about specific subproblems. We have noted that the presented approach is especially interesting, if the considered subproblems  \cref{eq:the_alg} can still be solved to global optimality. An interesting for this are cases where the subproblems can be solved to global optimality by lifting \cite{alberti_calibration_2003,pock_global_2010} or other relaxation strategies \cite{kolmogorov_what_2004,boykov_fast_2001-1,ishikawa_segmentation_1998}

Models that include the total variation norm in place of the regularization term $R$ are ubiquitous in imaging tasks \cite{burger_guide_2013,rudin_nonlinear_1992} and have been a major motivation in applications of our work. These models will appear as special instances of the discussed 'liftable' subproblems.

We will start with a basic representation of functions that are amenable to lifting,
\begin{equation}\label{eq:pointwise_lift}
E(u) = \sum_{i=1}^n \nu_i(u_i,(Du)_i),
\end{equation}
with continuous functions $\nu: \R \times \R^d \to \R$ that are convex in their second argument and where $D$ denotes a finite-difference gradient. This is a natural discrete representation of the continuous model \cref{eq:cont_lift}, which can be efficiently solved by functional lifting \cite{pock_global_2010}. The choice of $\nu_i(x,y) = g_i(x) + ||y||_2$ then recovers a composite model with some term $G = \sum_{i=1}^n g_i$ and a regularizer $R$ which is total variation \cite{rudin_nonlinear_1992}:
\begin{equation}\label{eq:tv_functional}
E(u) = \sum_{i=1}^n g_i(u_i) + ||Du||_1,
\end{equation}
This is a successful strategy, but its application is limited by the fact that separability is needed. A much more general model would be
\begin{equation}\label{eq:gen_lift}
E(u) = G(\rho(u)) + \sum_{i=1}^n \gamma_i((D u)_i),
\end{equation}
with $\rho$ separable as before, $G$ L-smooth adaptable and $\gamma_i$ convex.
Yet while the lifting scheme of \cite{pock_global_2010} is not applicable due to the non-separability of $G$, this is nevertheless a special instance of our general problem \cref{eq:problem}. Indeed we can write down the majorizer, assuming a separable $h(u)=\sum_{i=1}^n h_i(u_i)$, as
\begin{equation}\label{eq:gen_tv}
E_{u^k}(u) = D_{h}(\rho(u),\rho(u^k)) + \langle \nabla G(\rho(u^k)),\rho(u^k)\rangle + \sum_{i=1}^n \gamma_i((D u)_i),
\end{equation}
up to constants. Now this majorizer is in turn a particular instance of \cref{eq:pointwise_lift}, as the first two terms are separable, and we can now solve \cref{eq:gen_lift} by iteratively solving the lifting subproblems.

Furthermore, we can even exchange convexity of $\gamma_i$ for differentiability. For arbitrary $\gamma_i$ that are L-smooth adaptable, we can linearize the second term as well, in full analogy to \cref{eq:gen_maj}, giving the majorizer
\begin{align*}
E_{u^k}(u) = &\sum_{i=1}^n d_{h_G}(\rho(u),\rho(u^k))
+ \langle \nabla G(\rho(u^k)),\rho(u^k)\rangle \\+ &\sum_{i=1}^n d_{h_{\gamma_i}}( (Du)_i,(Du^k)_i) + \gamma'_i((Du^k)_i)(Du)_i .
\end{align*}
which is again an instance of \cref{eq:pointwise_lift}. For concave $\gamma_i$ this is particularly attractive as we can choose $h_{\gamma_i}$ a as linear functions and just keep the linearization in full analogy to iterative reweighting as discussed in the previous subsection in \cref{eq:concave_maj}.

While this approach greatly increases the applicability of lifting schemes, it is important to keep in mind that previous global optimality considerations for lifting schemes \cite{pock_global_2010} do not translate to these generalized problems. From \cref{sec:convergence} we only gain convergence to stationary points. We will use the next section to analyze the quality of solutions that we receive numerically.

\subsection{Inertia}
A small side note to the previous investigations that is nevertheless quite interesting in the context of nonconvex optimization is inertia.
Once we have \eqref{eq:the_alg}, we can just as well consider 
\begin{align}\label{eq:inertia}
\begin{split}
E_{u^k,u^{k-1}}(u) = &\frac{1}{\tau}D_h(u,u^k) + \langle \nabla G(\rho(u^k)),\rho(u)-\rho(u^k)\rangle + G(\rho(u^k)) \\
&+ R(u) + \frac{\beta}{\tau}\left(D_h(\rho(u),\rho(u^k))-D_h(\rho(u),\rho(u^{k-1})) \right),
\end{split}
\end{align}
inserting an inertial term into the generalized forward-backward equation analogous to \cite{ochs_ipiano:_2014}.

Inertia can be quite valuable for first-order optimization, especially as we allow $G$ to be nonconvex, but only utilize its gradient, i.e local information in each step. In practice we observe that inertia can sometimes help the algorithm to reach better minima or speed up the initial convergence for badly conditioned $G$. Also spurious stationary points can often be overcome. Furthermore, the additional cost of solving \cref{eq:inertia} is minuscule compared to the non-inertial variant.

However, due to the non-existence of the triangular inequality for Bregman distances, we cannot bound the iterations by a Lyapunov function in general as in previous work \cite{ochs_ipiano:_2014} and continue the proof of convergence with this Lyapunov function as a majorizer analogous to \cref{sec:convergence}. A related discussion and solution in the convex setting can be found in \cite{hanzely_accelerated_2018}. For the special cases of induced squared norms, i.e. $h= ||u||_A^2$, convergence still follows by adapting \cref{sec:convergence_proof} to the results of the recent work \cite{ochs_unifying_2016}, but we omit a further discussion.

In practice this modification still works well in many cases, setting $\beta < 0.5$. It is also possible to backtrack in case of violations of Lyapunov function bounds.

\section{Experimental results}\label{sec:experiments} \label{list:G}
In this section we analyze the proposed algorithm numerically. We will first consider a synthetic example, where we will be able to compare the algorithm to other methods easily. 
We will then move to an imaging application, the depth super-resolution from raw time-of-flight  data.
\subsection{Synthetic experiments}
We analyze the following energy
\begin{equation}\label{eq:test}
\min_{u \in \R^n, u_i \in [a,b]} F_f(A\rho(u)) + R(u),
\end{equation}
where we have a bounded interval $[a,b]$, an L-smooth adaptable function $F_f \circ A : \R^n \to \R$ and regularizer $R(u) = \sum_{i=1}^n r(u_i-u^*_i)$ with $r:\R \to \R$. $\rho$ is chosen separable so that $\rho(u) = (\rho_1(u_1),\dots,\rho_n(u_n))$ with $\rho_i:\R \to \R$, whom we will in general choose equal, and omit the subscript. The nonlinearity $r$ is aligned so that $\argmin_x r(x) = 0$. We first draw $u^* \in [a,b]^n$ at random and then set $f=A\rho(u^*)$. We choose $F_f$ as a measure of distance between $f$ and $A\rho(u)$ that fulfills $u^* \in \argmin_u F_f(A\rho(u))$ and $F_f(A\rho(u^*)) = 0$.
Through this construction, we can guarantee that the drawn $u^*$ will be a global minimizer of \eqref{eq:test}.

\begin{figure}
	\centering
	\subfloat[]{\label{fig:test_nonlin:a}\includegraphics[width=0.32\textwidth]{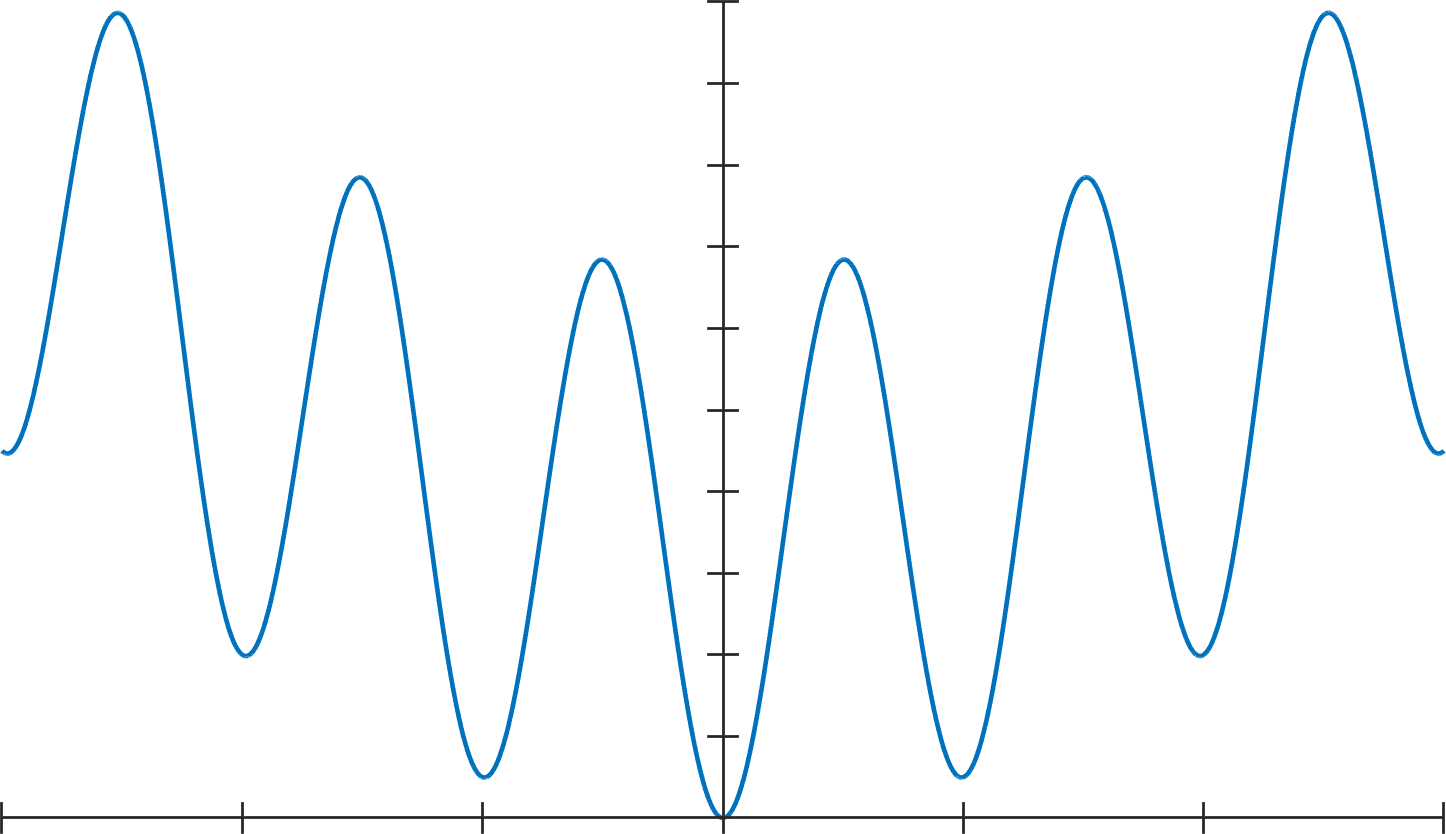}}
	\hspace{1pt}
	\subfloat[]{\label{fig:test_nonlin:b}\includegraphics[width=0.32\textwidth]{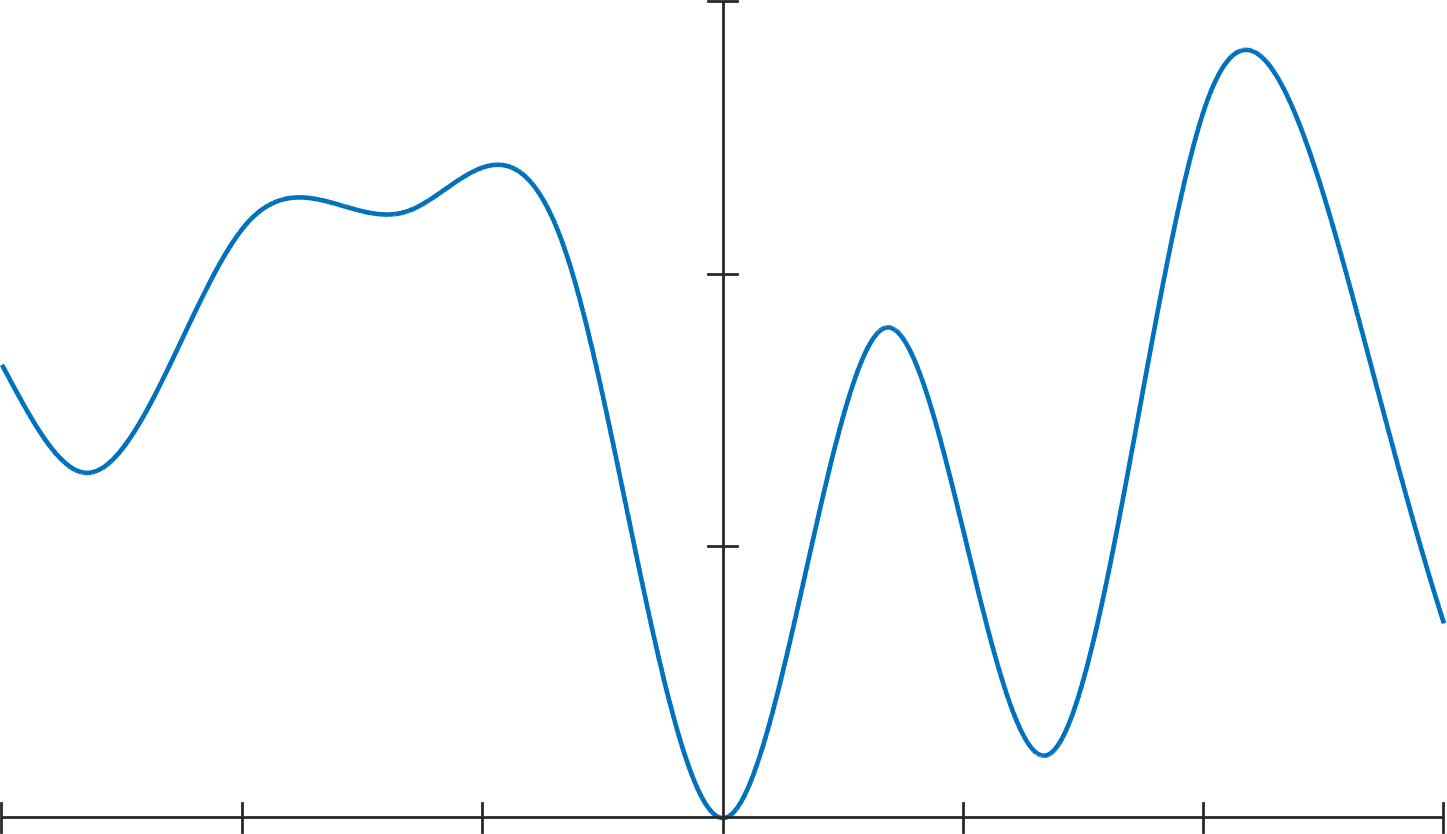}}
	\hspace{1pt}
	\subfloat[]{\label{fig:test_nonlin:c}\includegraphics[width=0.32\textwidth]{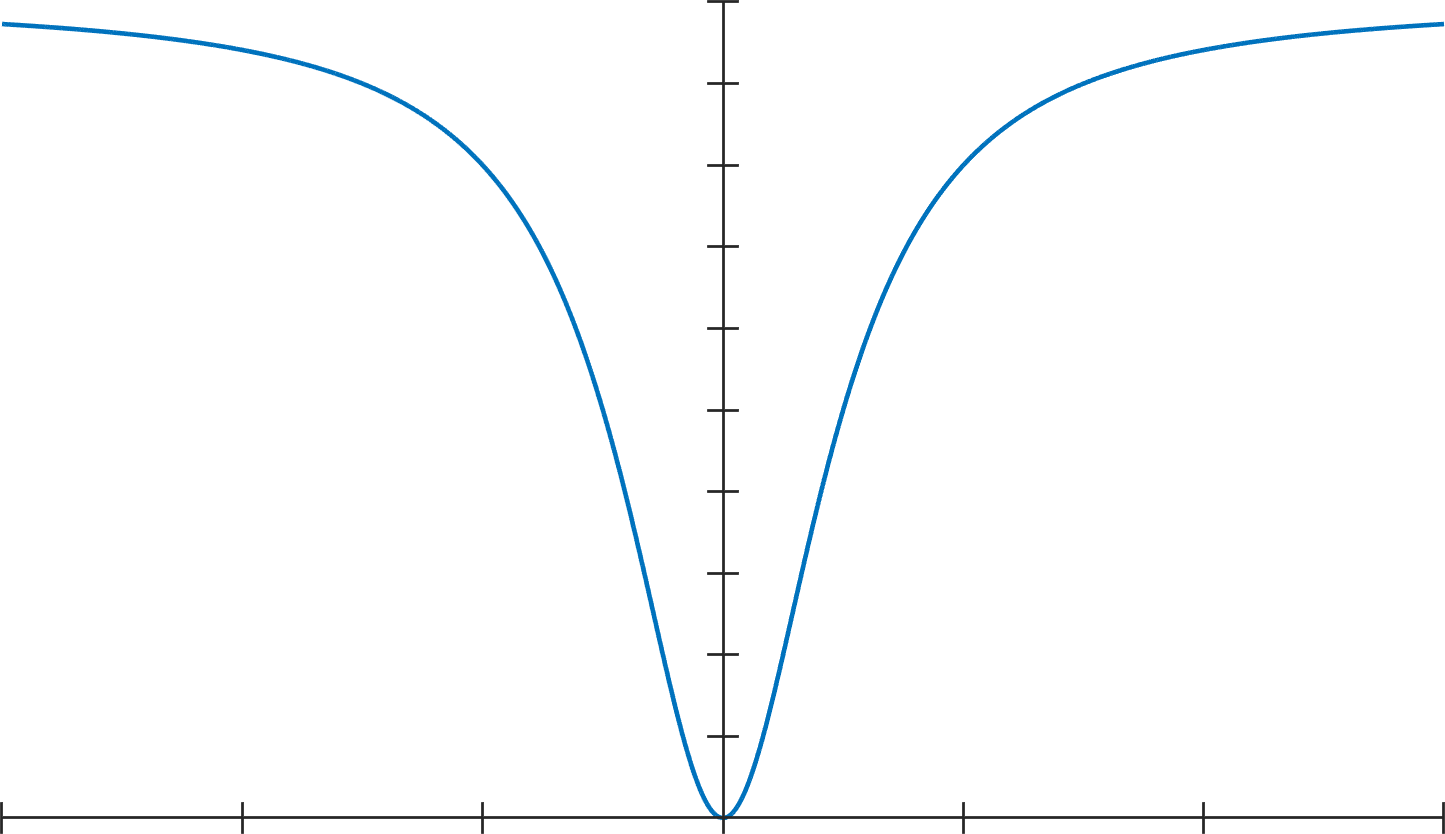}}
	\caption{Example of nonlinearities used in synthetic experiment, (a) $x^2-10\cos(2\pi x)$ (Rastrigin's function \cite{muhlenbein_parallel_1991}), (b) random spline function with 12 queries and (c) $\frac{x^2}{1+x^2}$. \label{fig:test_nonlin}}
\end{figure}	

Now we vary the difficulty of this possibly nonconvex optimization problem in two ways. 
First we choose nonlinearities $\rho,r$ as either
\begin{enumerate}[label=(\arabic*)]
	\item \emph{Simple} $\rho(x) = \exp(x)$, $r(x) = x^2$
	\item \emph{Doable} $\rho(x) = x^2-10\cos(2\pi x)$ \cite{muhlenbein_parallel_1991}, cf. \cref{fig:test_nonlin:a}, $r(x) = \frac{x^2}{1+x^2}$, cf. \cref{fig:test_nonlin:c}
	\item \emph{Difficult} $\rho$ is a (coercive) piecewise cubic polynomial drawn by interpolating 12 values in $[a,b]$, $r(x) = -\text{sinc}(x)$
	\item \emph{Very Difficult}$\rho$ is a (coercive) piecewise cubic polynomial drawn by interpolating 12 values in $[a,b]$, $r(x) = x^2-10\cos(2\pi x)$. 
\end{enumerate}
This allows us to move from a nicely behaved, almost convex test case (1) to a nonconvex problem with a well-behaved minimizer (2), adding further oscillations in (3) and finally arriving at two "very nonconvex" functions in (4).

Then we vary the function $G = F_f \circ A$. Note that this function critically determines the interconnection of variables. If $A$ is a diagonal matrix, then the problem is fully separable and can by solved by $n$ separate 1D optimizations with a single step of the algorithm, but if $A$ is a full matrix, then all variables are interdependent. Further, when $A$ is a rectangular matrix, then the system of nonlinear equations is under-determined and the function landscape is (intuitively) not as well-behaved. Also, we are allowed to choose nonconvex functions $F$ as long as $G$ is still $L$-smooth adaptable.

\begin{enumerate}[label=(\alph*)]
	\item \emph{Convex, local:} $F_f(v) = \frac{1}{2}||v-f||^2$,$A \in \R^{n \times n}$ is a random matrix whose entries are normally distributed relative to its diagonal. 
	An appropriate essentially smooth function is $h(v) = \frac{1}{2}||v||_D^2$, where $D$ is a diagonal matrix with entries $d_i = \sum_{j=1}^n |A^TA|_{ij}$
	\item \emph{Convex, non L-smooth, local:} $F_f$ is the KL-divergence $F_f(v) = \sum_{i=1}^n v_i-f_i\log(v_i)$, $A \in \R^{n \times n}$, is chosen as in (a). Here $h$ is given by Burg's entropy $h(v) = \sum_{j=1}^n -\log(v_j)$ \cite{bauschke_descent_2017}.
	\item \emph{Convex, full:} $F_f(v) = \frac{1}{2}||v-f||^2$,$A \in \R^{m \times n}$ is a full random matrix with singular values in $[\frac{1}{\log(n)},1]$. $m= \frac{n}{3}$. 
	Choose $h$ as in (a).
	\item \emph{Nonconvex, full:} $F_f$ is a smooth-truncated quadratic \cite{artina_linearly_2013}, i.e. a smoothed version of $F_f(v) = \sum_{i=1}^m \min((v_i-f_i)^2,\lambda)$, $A \in \R^{m \times n}$ is a full random matrix with singular values in $[\frac{1}{\log(n)},1]$, $m= \frac{n}{3}$.
	Choose $h$ as in (a).
	
\end{enumerate}
\begin{figure}[h]
	\centering
	\subfloat[Proposed method\cref{eq:the_alg}]{\includegraphics[width=0.49\textwidth]{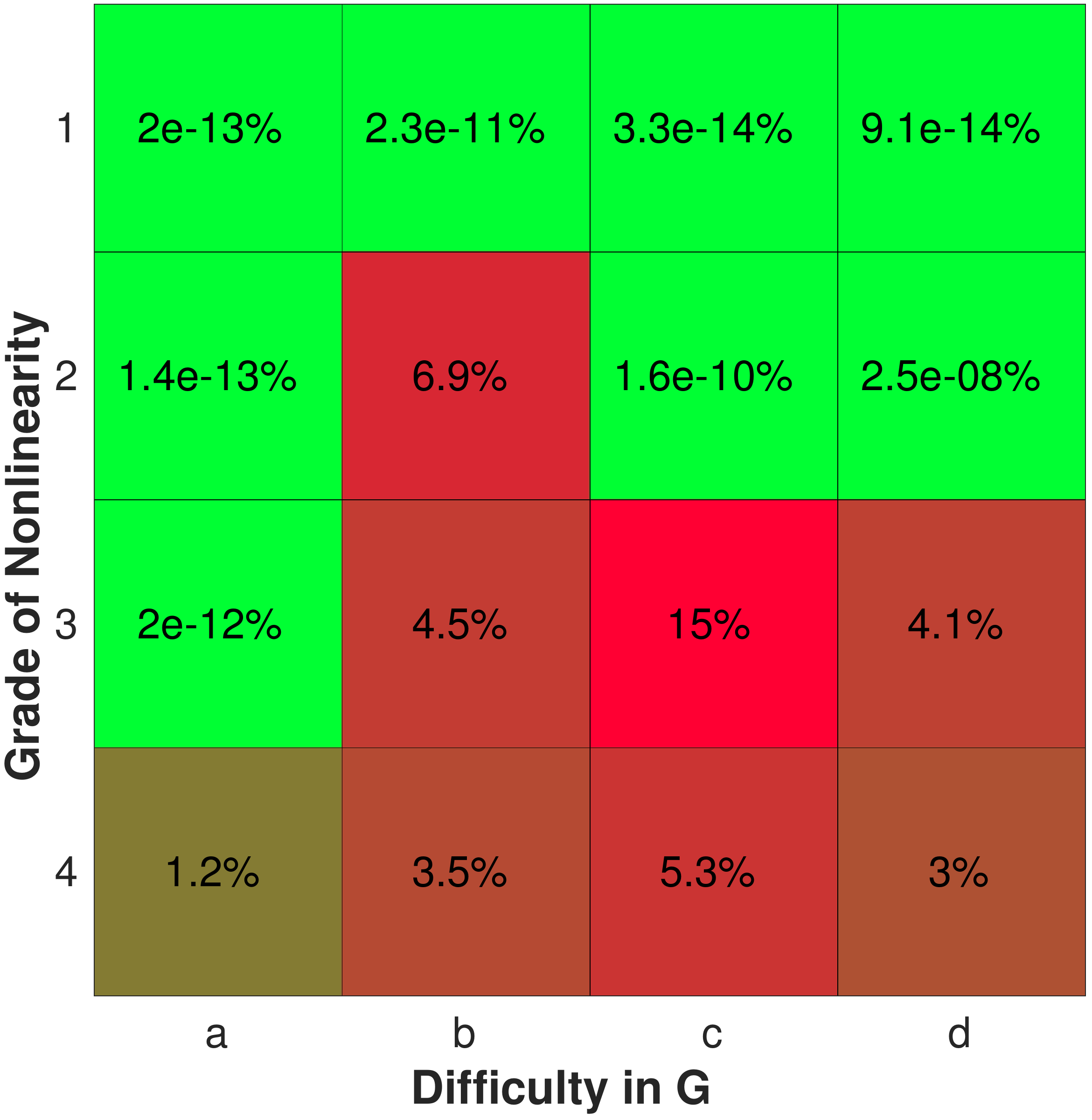}}
	\subfloat[Proposed (with inertia) \cref{eq:inertia}]{\includegraphics[width=0.49\textwidth]{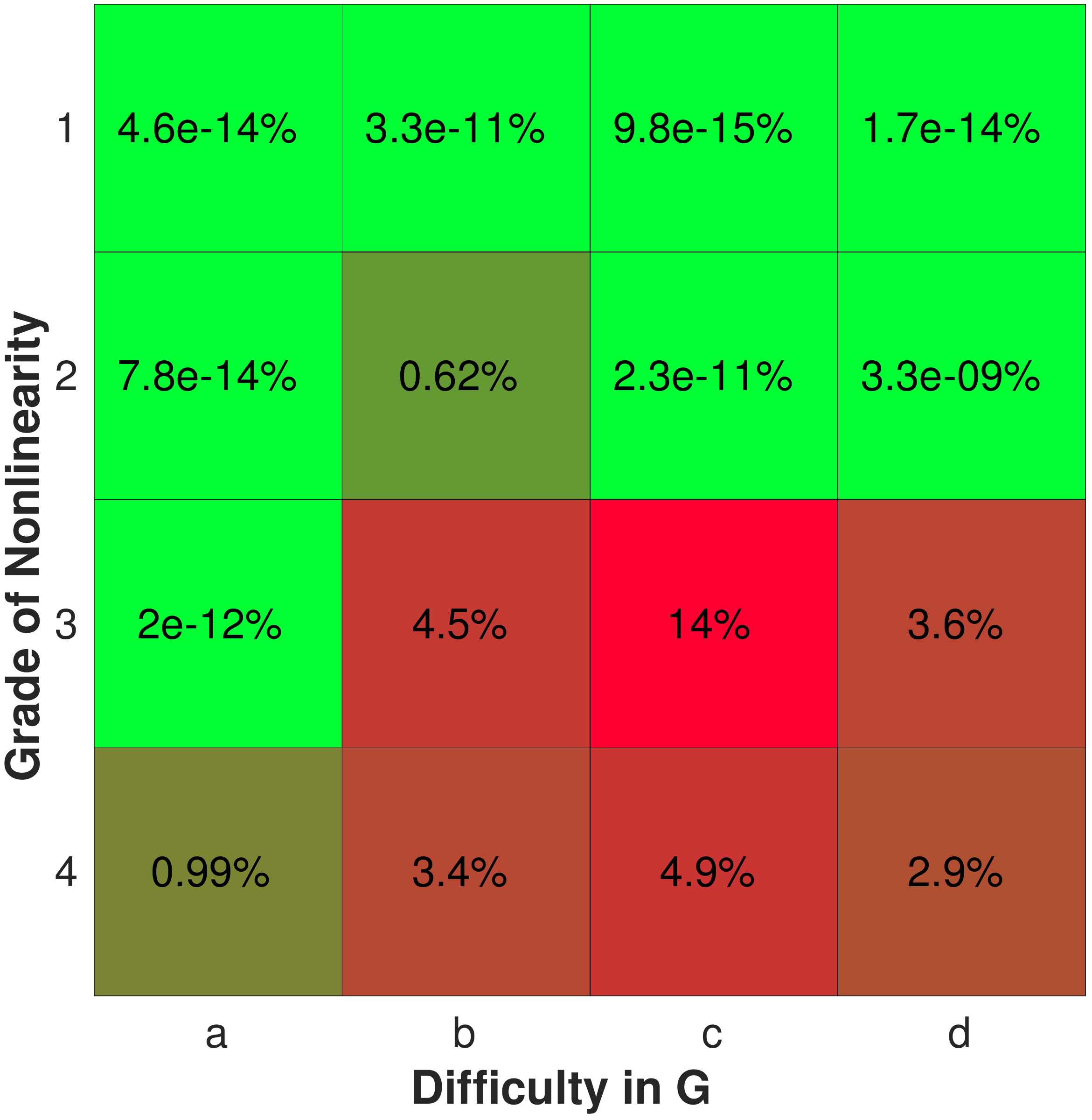}}
	\caption{Examination of different synthetic experiments. We increase the difficulty in $G$ from left to right and in $\rho,r$ from top to bottom, as detailed in \cref{list:G}. In each cell we show the value $\frac{E(\bar{u})-E^*}{\tilde{E}}$, the value $E(\bar{u})$ reached by the algorithm minus the global minimum $E^*$, normalized by $\tilde{E}$, a (sampled) median of the energy values of the function, indicating the quality of the solution relative to the overall energy landscape.
		\label{fig:Self-comparison}}
\end{figure}
We run our method \eqref{eq:the_alg} without and with inertia, $\beta=0.4$ \eqref{eq:inertia}. The subproblems in each iterations are fully separable, so we solve the 1D problems in parallel by exhaustive search with a sufficient amount of trial points and a parabolic refinement around the approximate minimizer to desired precision. 
The refinement is a standard technique, e.g. \cite{huyer_global_1999} for 1D local optimization and further references can be found, for example in book of Luenberger \cite[pp. 217, 224]{luenberger_linear_2015}. This technique has also been used previously in imaging, for example, to refine exhaustive search procedures in the context of quadratic decoupling for stereo in \cite{kuschk_fast_2013-1}.

To mitigate the risk of lucky initializations, we run the algorithm 25 times with random starting vectors for each test case and show the result which reached a median energy value. We set $n=150$ and $[a,b]$ =  $[-3,3]$, respectively $[a,b] = [\epsilon,3]$ for the Poisson case and implement the proposed method in MATLAB.

The results for all test cases can be found in \cref{fig:Self-comparison}. 
We see that either increasing the difficulty in $G$ or the difficulty of the nonlinearity makes the overarching optimization problem more difficult. 
Remarkably, our algorithm was able to find near-global optima for many test cases, especially the performance in row (2) is very good. However we see that the increased oscillations in (3) eventually degrade the quality of solutions. We also notice that differences can appear even for convex functions $G$ in (a) and (b). The squared $l^2$ norm in (a) seems to be easier to optimize globally, although the disparity to (b) is also connected to the analytical step sizes, that we choose. Choosing larger stepsizes for (b), e.g. by backtracking, would recover a similar behavior to (a).

\begin{figure}
	\subfloat[1d]{\label{fig:convergence_comparison:a}\includegraphics[width=0.5\textwidth]{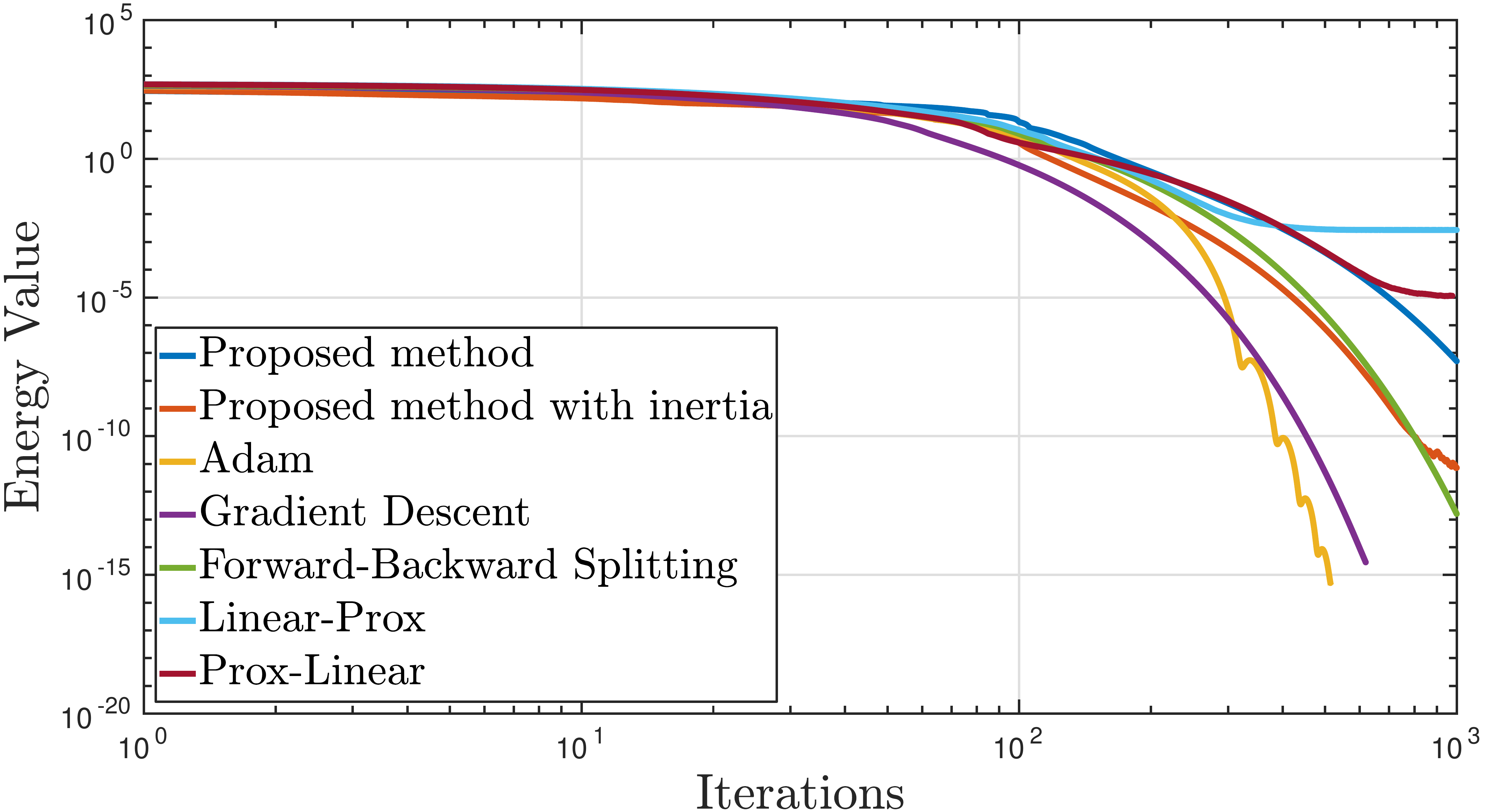}}
	\subfloat[3a]{\label{fig:convergence_comparison:b}\includegraphics[width=0.5\textwidth]{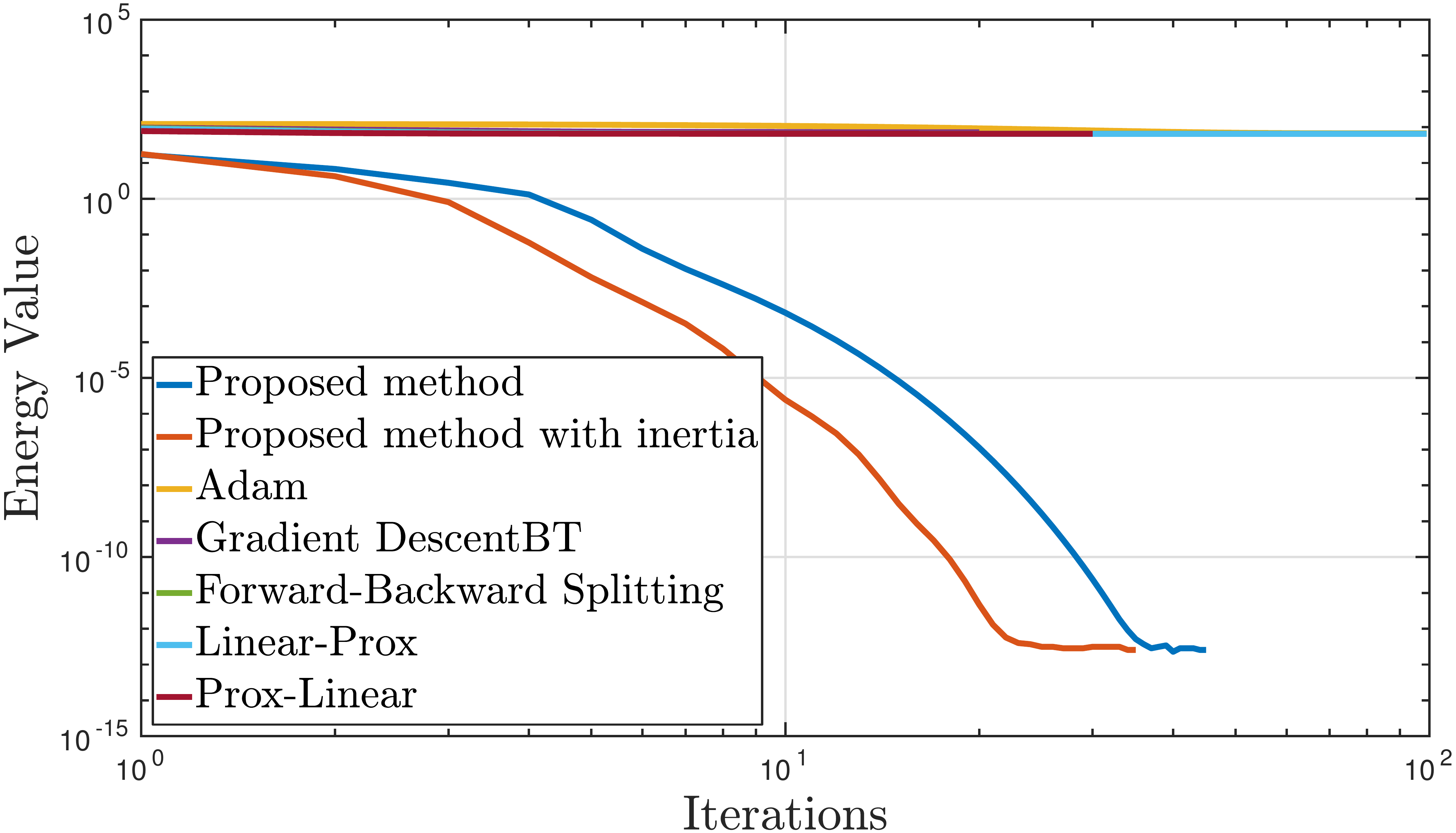}}
	\caption{Convergence of various first-order methods for composite optimization to global optimality. This figure shows problem 1d, i.e. a problem where $R,\rho$ are easy and $G$ difficult, and 3a, the opposite case. Only our method provides favorable results for the difficult case 3a. Shown are the proposed method \cref{eq:the_alg}, proposed with inertia: \cref{eq:inertia}, Adam: \cite{kingma_adam:_2015}, gradient descent: \cref{eq:intro_grad}, forward-backward splitting: \cref{eq:intro_fb}, Linear-Prox: \cref{eq:rel_lm}, Prox-Linear: \cref{eq:rel_pl} \label{fig:convergence_comparison}}
\end{figure}
We now compare with other first-order nonlinear optimization methods, 
namely as mentioned in the related work section, gradient descent \cref{eq:intro_grad}, forward-backward splitting \cref{eq:intro_fb}, the inner linearization, 'prox-linear', \cref{eq:rel_lm} and the outer linearization \cref{eq:rel_pl}.

To fairly evaluate all majorizers we generally solve the subproblems in forward-backward splitting \eqref{eq:intro_fb}, and outer linearization \eqref{eq:rel_pl} to global optimality, again with exhaustive search and parabolic fitting. For prox-linear \eqref{eq:rel_lm}, the subproblems do not decouple and we apply a standard interior point solver in each iteration. We otherwise apply the same methodology as before. We compare the convergence of all algorithms to the global minimum for two characteristic cases, '1d' and '3a'. While the first case denotes simple $\rho,R$ and difficult $G$, the second case denotes the opposite. We expect most methods to do well in the first case, but the second case is not as clear. \Cref{fig:convergence_comparison:a} and \cref{fig:convergence_comparison:b} show the results. It turns out that indeed all methods can reliably solve the first case, \cref{fig:convergence_comparison:a}, but only our method can find near-optimal solutions for the second test case, \cref{fig:convergence_comparison:b}. To compare our method to a modern 'aggressive' inertial variant that does not admit to a majorization-minimization framework, we also include the Adam optimizer \cite{kingma_adam:_2015}, however while this optimizer can find better minima, it is still far off from the global solution in test case '3a'.

\begin{figure}
	\includegraphics[width=\textwidth]{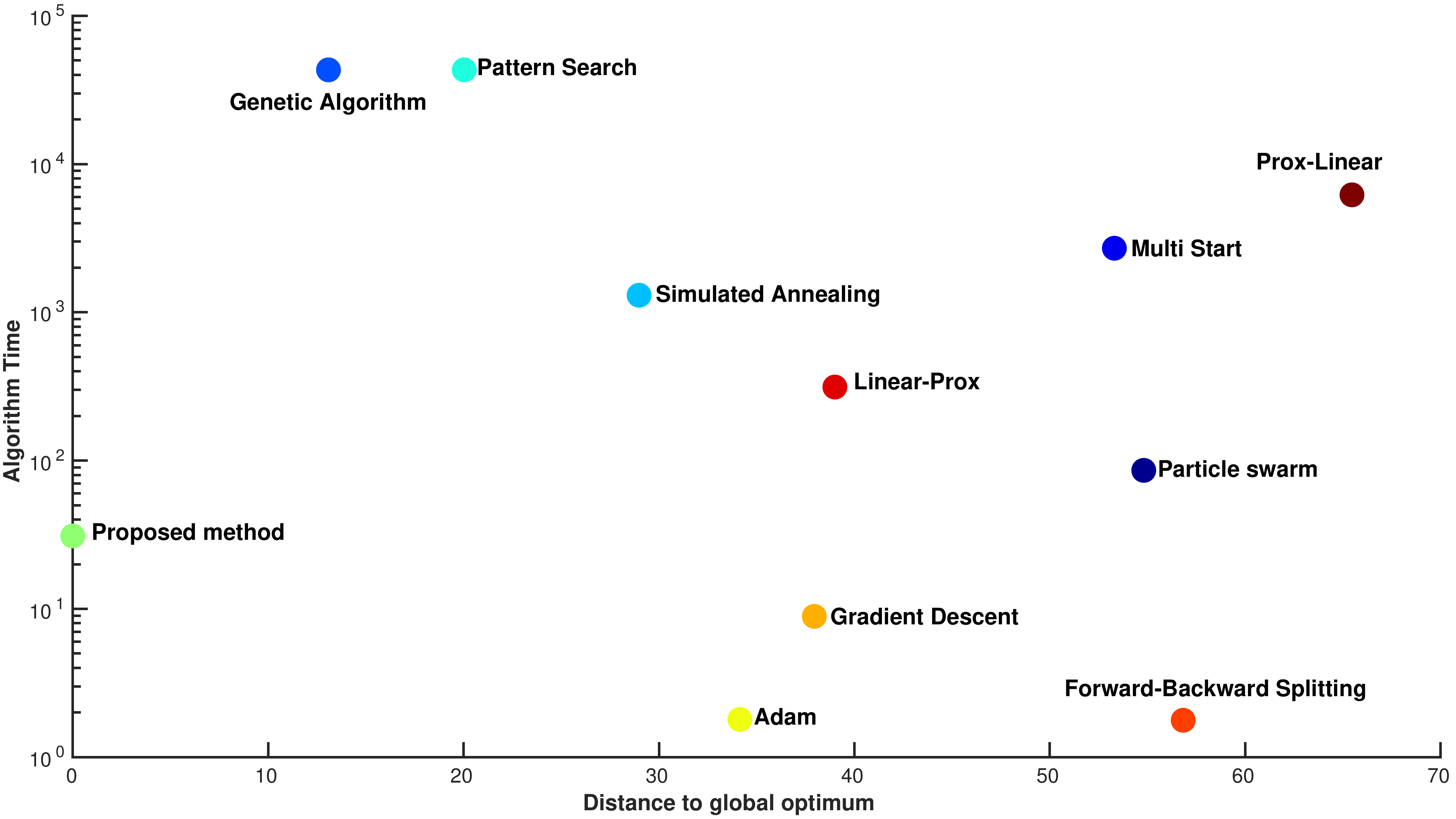}
	\caption{Evaluation of various optimization methods for test case '2c', in terms of energetic difference to global optimum vs. runtime in seconds. \label{fig:efficiency}}
\end{figure}

We can go a step further and compare the proposed algorithm to global optimization methods. 
See \cref{fig:efficiency} for a plot showing test case '2c' and the energetic difference to the global minimum versus the runtime of each the algorithms, with time in a log-scale. Previously mentioned algorithms are shown, as well as the MATLAB default implementations of a genetic algorithm \cite{goldberg_genetic_1989}, particle swarm \cite{kennedy_particle_1995}, pattern search \cite{audet_analysis_2002}, simulated annealing \cite{ingber_adaptive_1996} and multi-start methods \cite{ugray_scatter_2007}. All of these methods can reliably find global near-optimal solutions for low dimensions
, however in our setting of $n=150$ these methods fail to find a global minimizer within reasonable time constraints, as their efficiency decreases with the number of variables. In contrast, our method exploits the structure of the objective function, linearizing the convex outer function and solving the separable subproblems globally, and scales well into higher dimensions.
We mention briefly that the apparent slow runtime of 'prox-linear' and 'linear-prox' for this test case is partly implementation related, as we  solve 'prox-linear' with a generic interior point solver in each iteration, but also because both algorithms converge to very flat local minima.
\begin{figure}[h]
	\centering
	\subfloat[]{\includegraphics[width=0.42\textwidth]{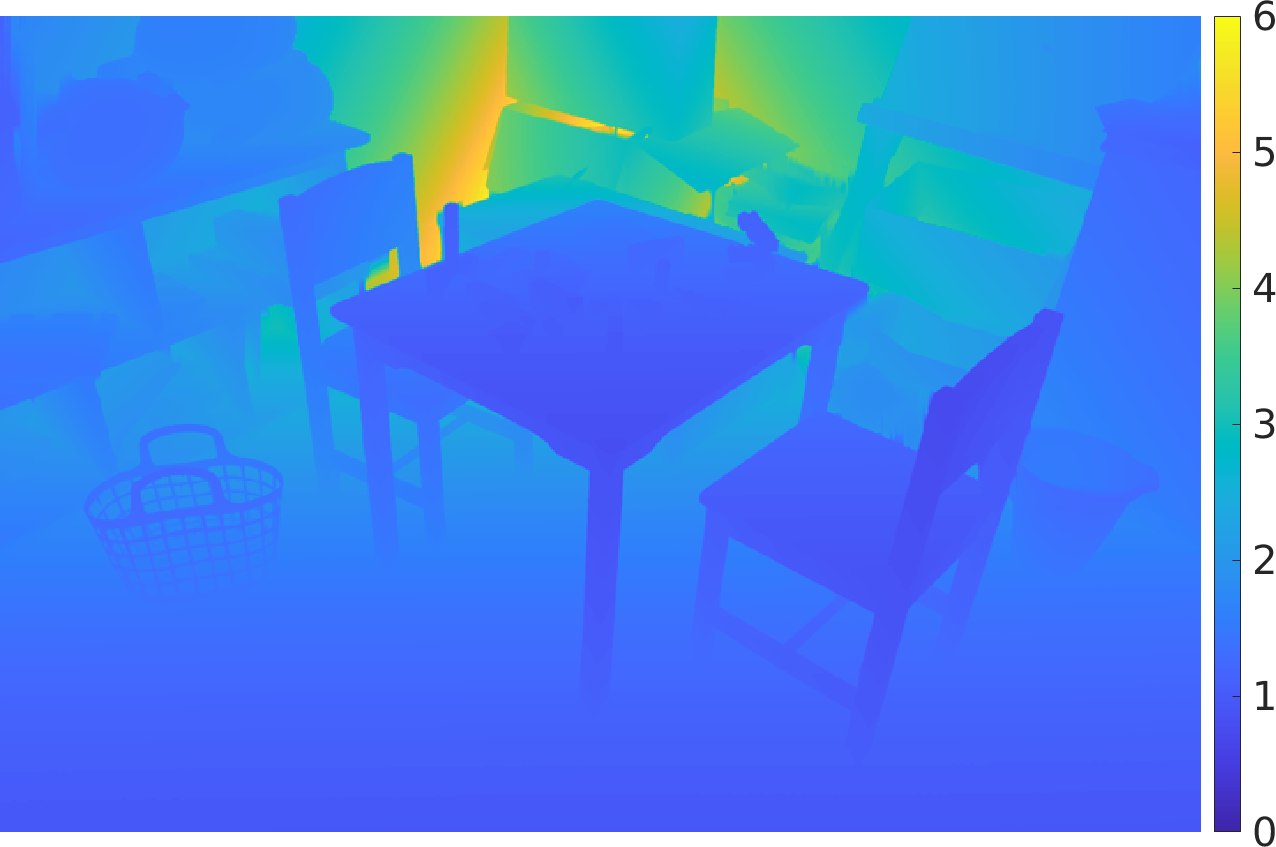}}
	\hspace{1pt}
	\subfloat[]{\includegraphics[width=0.49\textwidth]{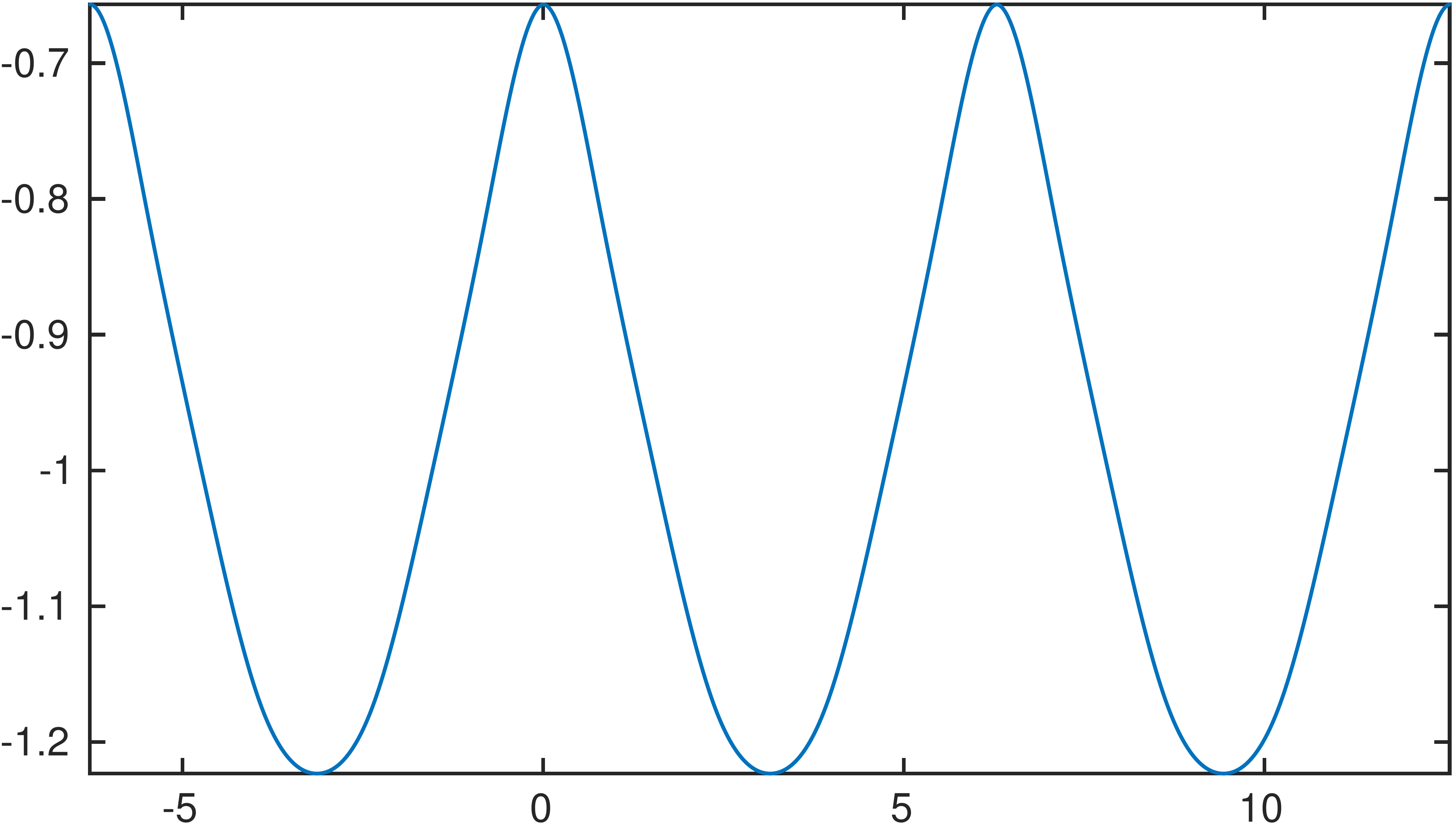}}
	\caption{Ground Truth depth data \cite{scharstein_high-resolution_2014} shown to the left and synthetic autocorrelation function generated from trapezoidal signal to the right. \label{fig:ground_truth}}
\end{figure}
\subsection{Time-of-Flight Depth Reconstruction}
Time-of-Flight cameras are used to recover depth images of a scene. They illuminate the scene with a continuous wave and measure the time of flight of reflecting waves. A modern hardware for this task are correlation photo-sensors, which directly measure the correlation of the incoming wave with a reference wave \cite{lange_3d_2000}. The inversion of this correlation computes the depth, however the process is highly non-linear and in practice often solved by  assuming the incoming waves to be purely sinosoidal and computing the analytical inversion at each pixel separately. This introduces several systematic errors into the depth measurements, especially at lower frequencies \cite{lindner_lateral_2006}. Further, these sensors are have a relatively low resolution due to their complexity and measurements contain a significant amount of noise.

A recent work on time-of-flight super-resolution \cite{xiao_defocus_2015} shows a variational model which includes the precise reference wave, downsampling, blur and noise effects. They model the incoming wave as arbitrary periodic function and find it by thorough calibration. 
In \cite{xiao_defocus_2015}, the resulting nonconvex energy model is solved by alternating local optimization in all variables. We will show that the problem can be solved with our approach and test on synthetic data.

\begin{figure}
	\includegraphics[width=0.24\textwidth]{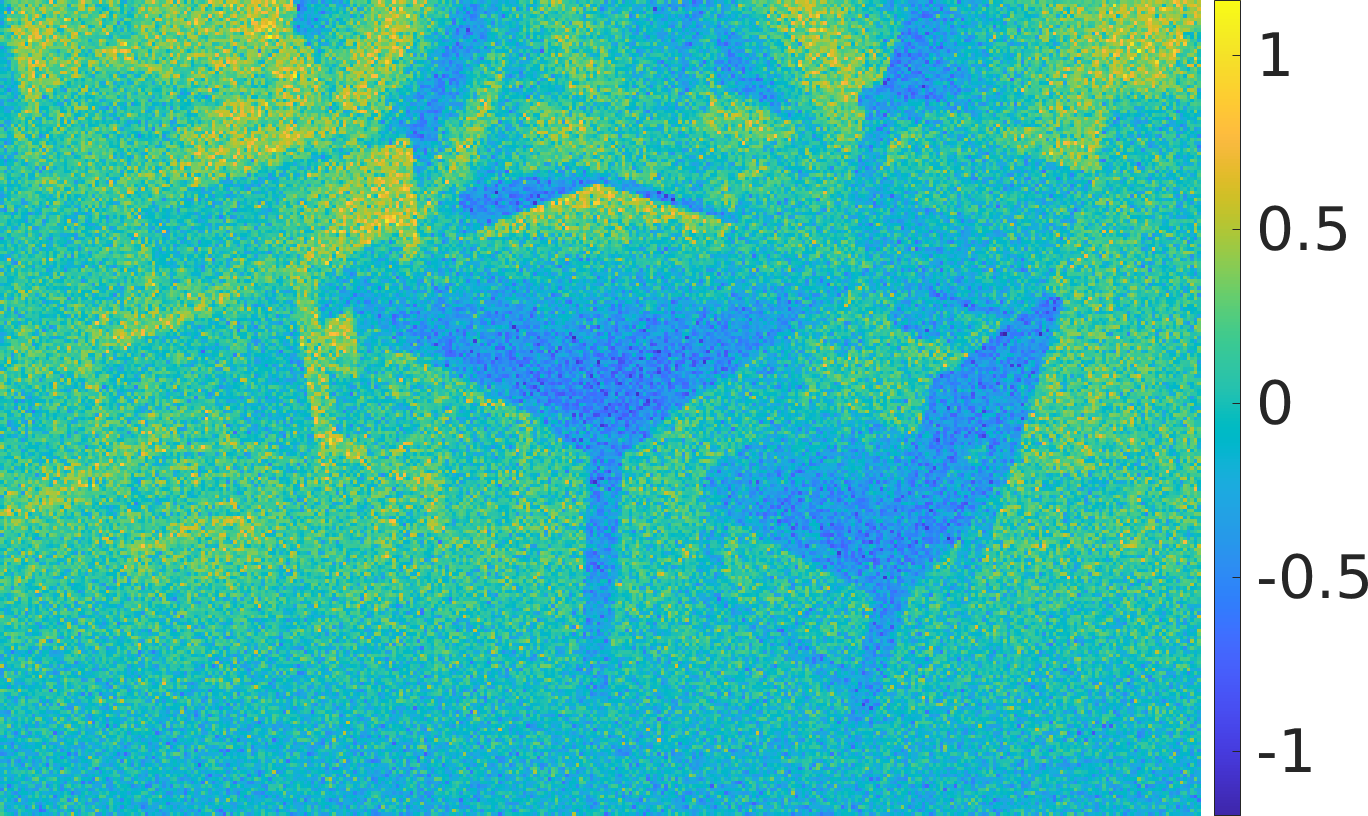}
	\includegraphics[width=0.24\textwidth]{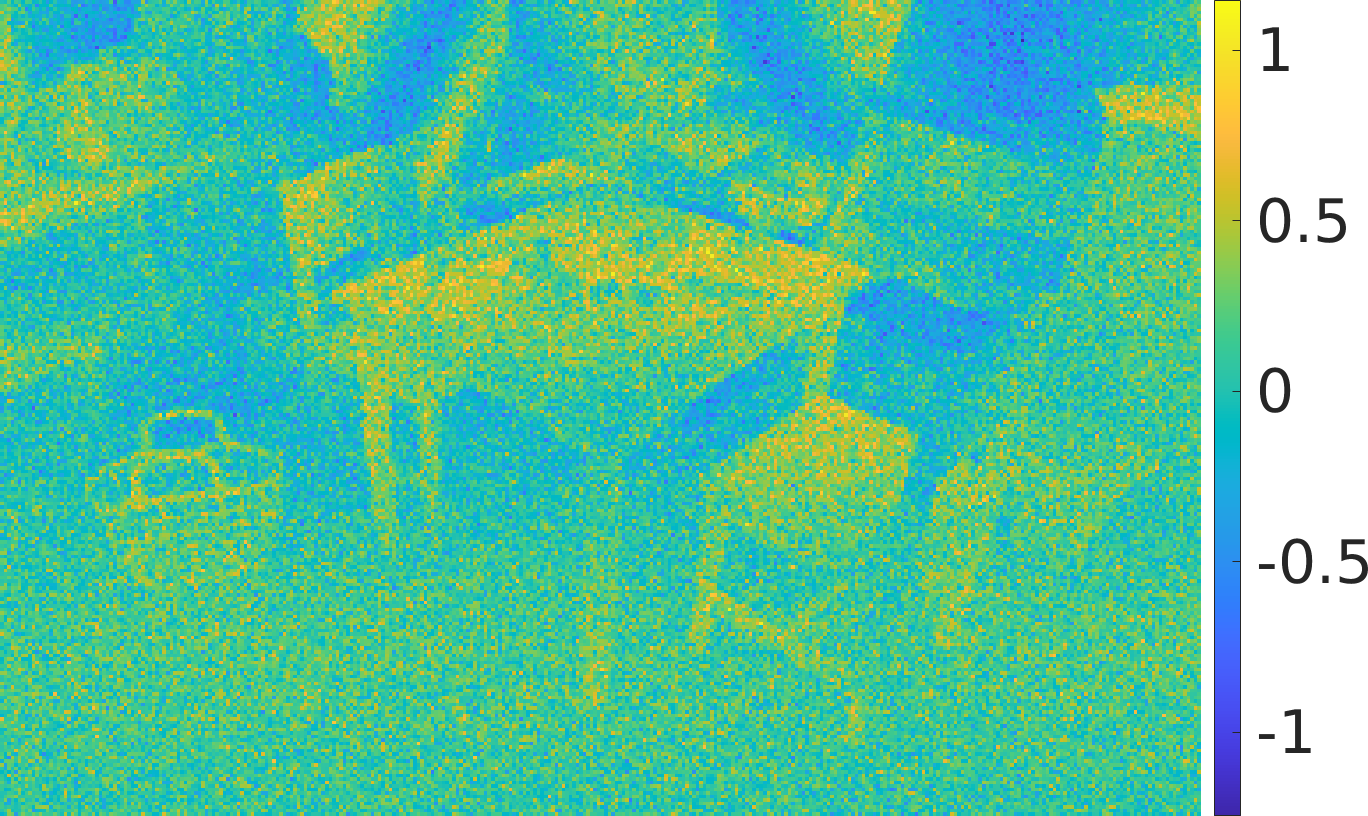} 
	\includegraphics[width=0.24\textwidth]{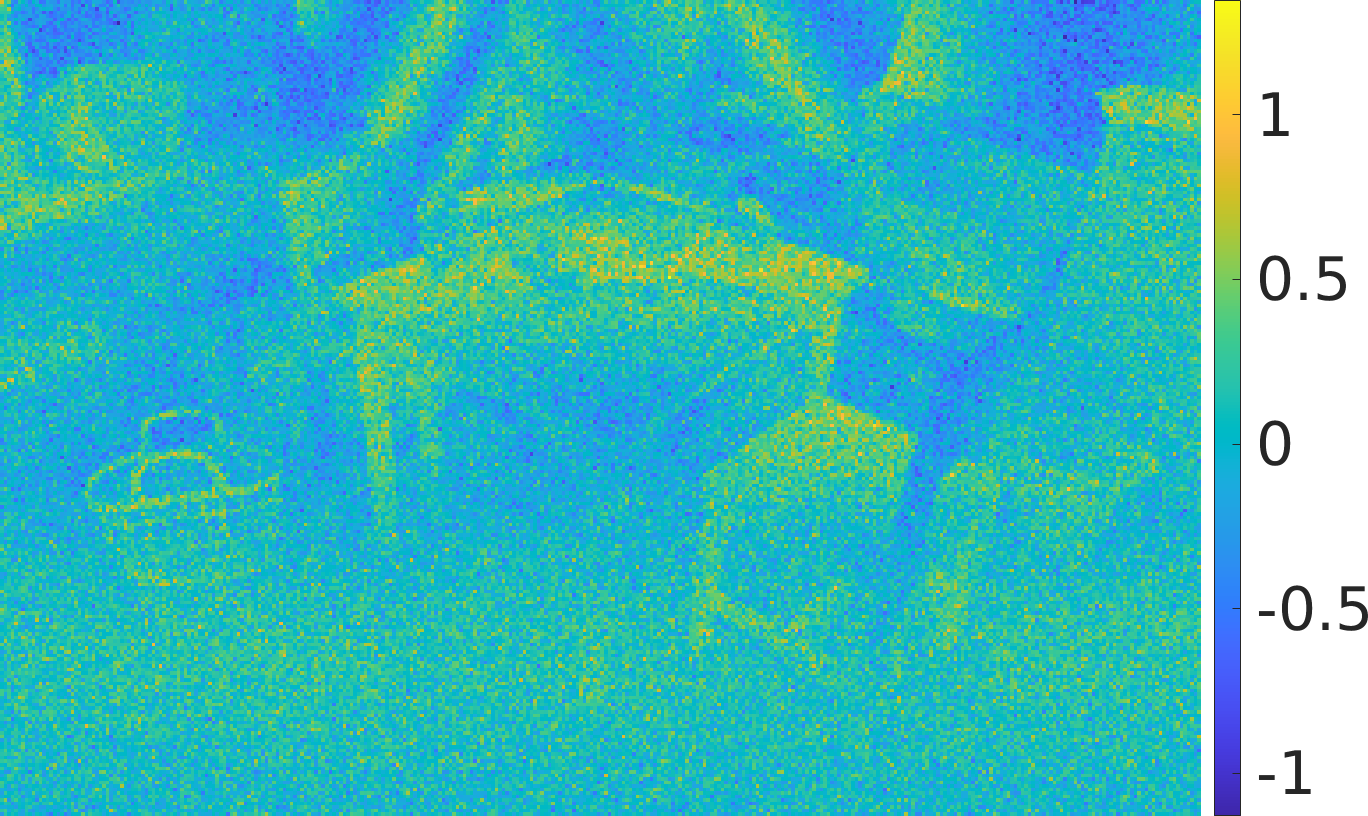}
	\includegraphics[width=0.24\textwidth]{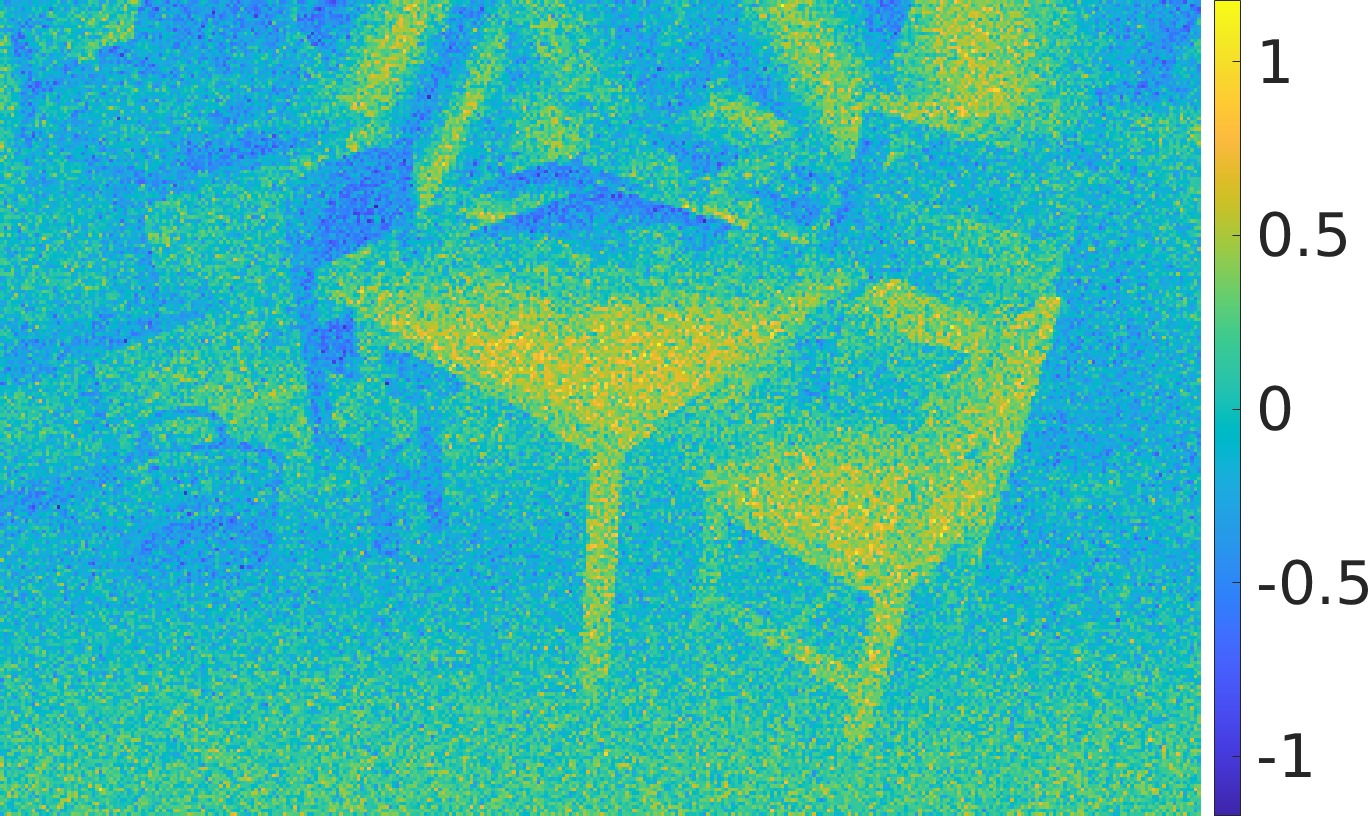}
	\caption{The four difference measurements $y_{ij}$, generated by equation \cref{eq:tof_data} with subsampling and Gaussian noise. \label{fig:data}}
\end{figure}

In the following, we will assume the following imaging model of a time-of-flight system
\begin{equation}
\tilde{y}_{ij} = a_i g_i\left( \frac{4\pi f_i}{\lambda}u + \frac{2\pi j}{n} \right) +b_i =: k_{ij}(u)  +b_i,
\end{equation}
for measurement $j$ in frequency $f_i$ and $g_i$ the $2\pi$-periodic autocorrelation in frequency $i$ that is either calibrated or otherwise known, e.g as a cosine. $a_i$ is the amplitude in frequency $f_i$, $n$ the number of measurements in each frequency $f_i$ and $\lambda$ the speed of light.
To remove the background illumination $b_i$ it is customary in Time-of-Flight literature to consider the difference measurements
\begin{equation}\label{eq:tof_data}
y_{ij} = k_{ij}(u) - k_{i,j+\frac{n}{2}}(u) =: \rho_{ij}(u).
\end{equation}
The problem of recovering a high-resolution depth image $u$ from measurements $y_{ij}$ can now be stated as the energy minimization of 
\begin{equation}\label{eq:tof_energy}
E(u) = \sum_{i=0}^{m-1} \sum_{j=0}^{n/2-1} ||y_{ij}-K\rho_{ij}(u)||^2 + \alpha ||\nabla u||,
\end{equation}
see also \cite{xiao_defocus_2015}. $K$ is the imaging operator, which is here a downsampling operator. The total variation regularization encourages a piecewise-constant depth solution. 

The energy can be solved with the proposed method, because we can identify \cref{eq:tof_energy} with the previously introduced special case of a sum of several composite terms \cref{eq:gen_problem} and a total variation regularization \cref{eq:gen_tv} - we can find a majorizer in each iteration that can be solved by functional lifting.

In practice we solve all subproblems with sub-label accurate lifting as described in \cite{mollenhoff_sublabel-accurate_2016}. 
We initialize the algorithm with a constant depth of 1m and then iteratively update the nonconvex majorizer and solve the sublabel-accurate lifting problem. We use a primal-dual algorithm \cite{chambolle_first-order_2011} to solve our subproblems and 'warm start'  each inner iteration with the primal-dual variables from the previous step. 
We note that this relaxation approach can possibly produce solutions that are convex combinations of global minimizers. To mitigate this problem, we monitor the energy of our inner iterations and terminate the algorithm if lifting cannot successfully minimize the majorizer, i.e. if any iterates $u^{k+1}$ were to violate $E_{u^k}(u^{k+1}) \leq E(u^k)$, which was postulated in Assumption A. However such a violation could not be detected for the Time-of-Flight experiment shown here.

To test this procedure experimentally we generate synthetic data from a depth image of the Middlebury dataset \cite{scharstein_high-resolution_2014} by applying \cref{eq:tof_data}. As a model for $g_i$ we use the autocorrelation of a trapezoid signal. The autocorrelation signal and the ground truth depth are shown in \cref{fig:ground_truth}. We then apply downsampling by a factor of 2 to model the limited sensor size and add significant Gaussian noise to model the sensitivity of common ToF sensors.
We generate two difference measurements in two frequencies, 90 MHz and 120 MHz. The ground truth data covers a depth ranging from 0.5 to 6m. The resulting measurements are outside the unambiguous range of both frequencies, so we expect a wrapping of data, which we want to resolve using both frequencies. We visualize the resulting four difference measurements in \cref{fig:data}. The noise level and severe data wrapping are apparent. 

A classical inversion of the given data by the nonlinear closed form solution for sinusoidal data \cite{lange_3d_2000} is shown for each frequency in \cref{fig:closed_form}. The solution is however contaminated by the nonlinear effects of noise and severe wrapping, note that we adjusted the colors for visualization purposes. Further heuristics would be required in a next step to combine both solutions to a final result, but we omit these due to the already significant distortion.
\begin{figure}
		\centering
		\subfloat[]{\includegraphics[width=0.49\textwidth]{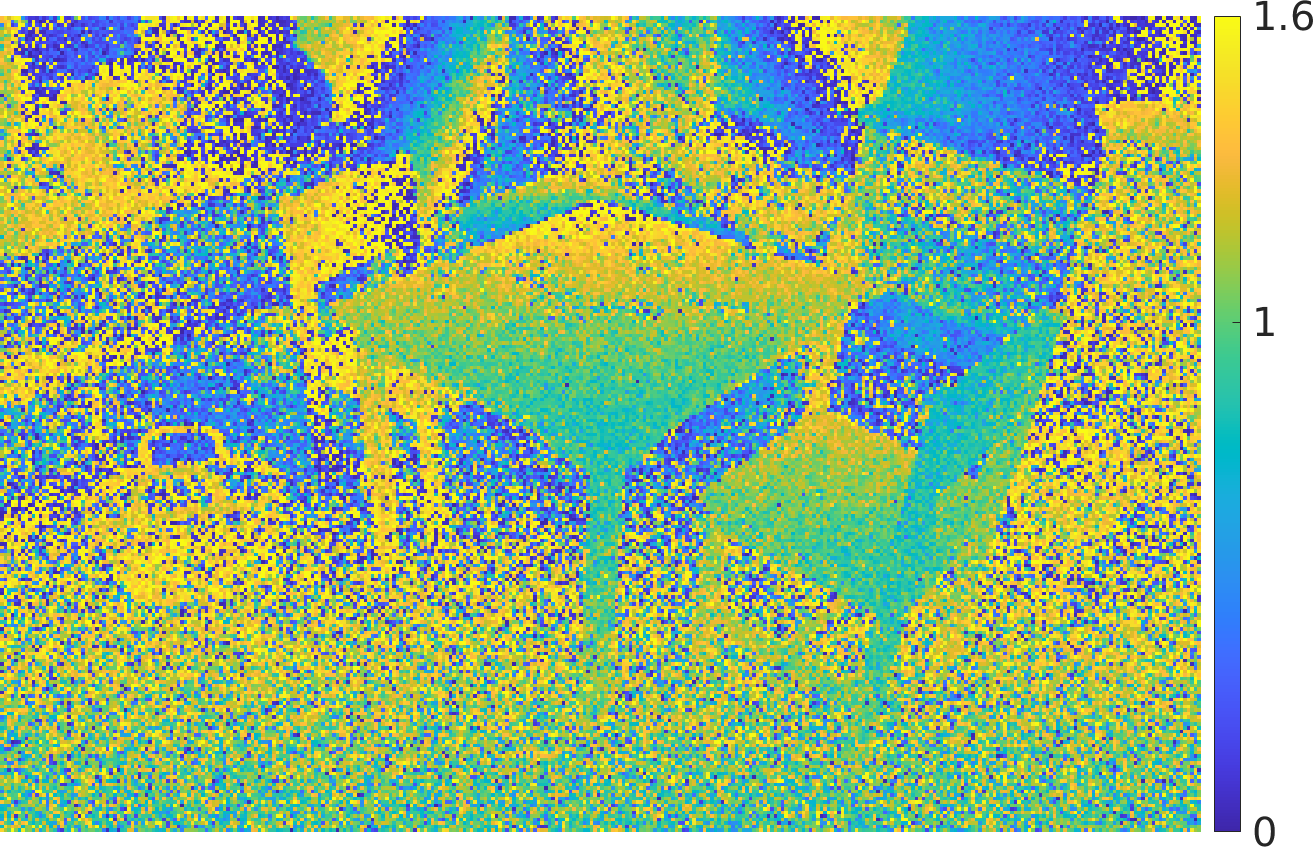}}
		\hspace{1pt}
		\subfloat[]{\includegraphics[width=0.49\textwidth]{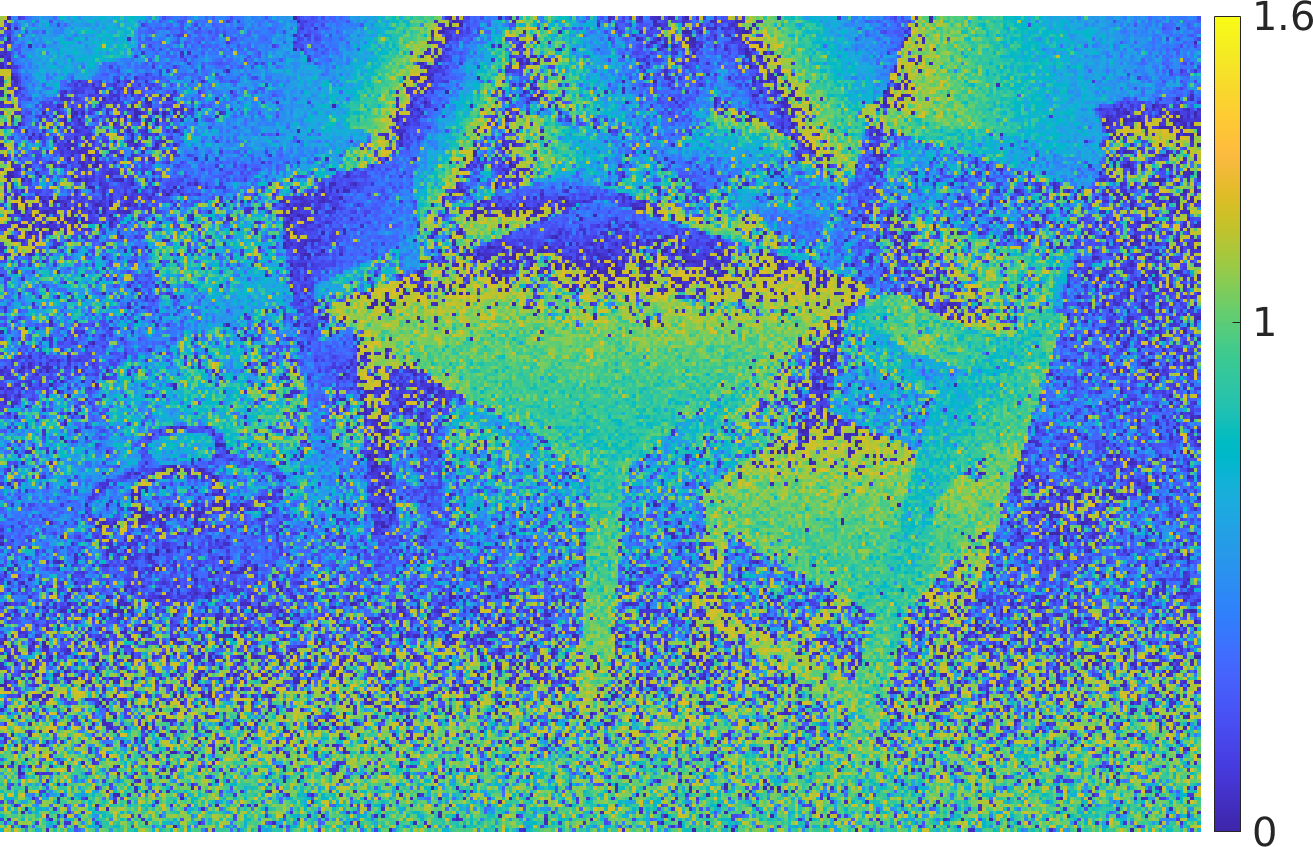}}
	\caption{Classical closed form solution to depth recovery \cite{lange_3d_2000}. 90 MHz data to the left and 120 MHz to the right. \label{fig:closed_form}}
\end{figure}
\begin{figure}
		\centering
		\subfloat[]{\label{fig:solution:a}\includegraphics[width=0.49\textwidth]{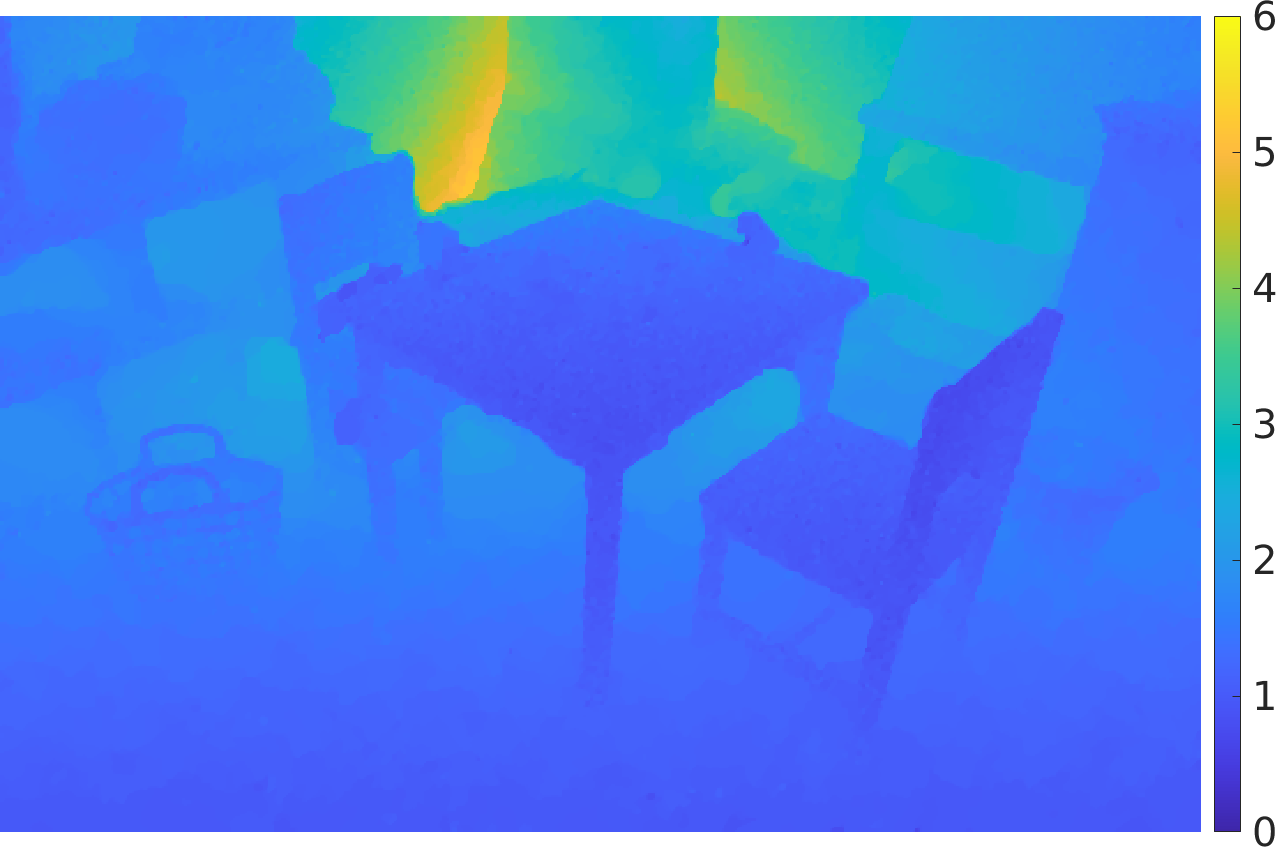}}
		\hspace{1pt}
		\subfloat[]{\label{fig:solution:b}\includegraphics[width=0.49\textwidth]{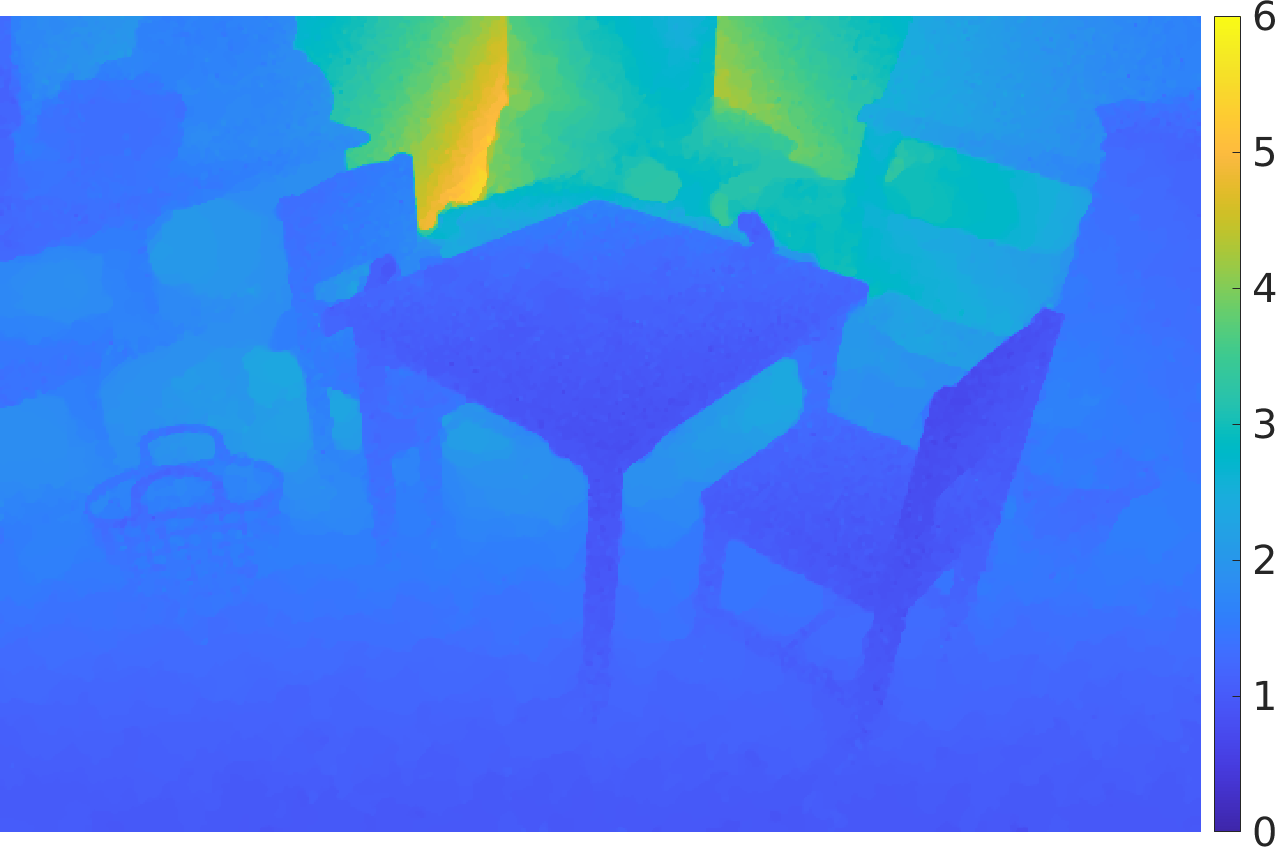}}
	\caption{Solution by the proposed algorithm with lifted subproblems, upsampling factor of 2 to the left. To the right, proposed algorithm initialized with the ground truth, \cref{fig:ground_truth}, for reference. \label{fig:solution}}
\end{figure}

In contrast the solution by our algorithm is shown in \cref{fig:solution:a}. For reference, the solution to the algorithm, initialized with the ground truth data is also visualized in \cref{fig:solution:b}. We see that the recovered solution is near-optimal. The algorithm can accurately unwrap and upsample most of the depth data and only small areas around the left chair are misidentified.  
\begin{remark}
	The presented model assumes the knowledge of the signal amplitude $a_i$ in each frequency by some preceding algorithm to streamline the presentation. If this information cannot be obtained robustly in practice, then the problem still falls into the problem category discussed in this paper, only the optimization variable $v = [u,a]$ is then vector-valued in depth and amplitude at each pixel (see the maximal generalization discussion in \cref{sec:modeling}). Yet, vector-valued variables can still be accounted for efficiently via the vectorial lifting shown in the works \cite{strekalovskiy_convex_2014,laude_sublabel-accurate_2016-1}. The overall algorithm remains unchanged, only the subproblems are solved by vector-valued lifting instead.
\end{remark}
We close this section by mentioning briefly that the presented composition of a matrix and a nonlinear wrapping operator is not entirely unique to Time-of-Flight reconstruction. A very related energy is present in nonlinear MRI reconstruction, see \cite{valkonen_primaldual_2014} for further reading.

\section{Conclusions}
\label{sec:conclusions}

In conclusion we proposed an optimization strategy for composite problems with simple, but highly-nonlinear, inner functions and L-smooth adaptable outer functions. We construct nonconvex majorizing functions and show 
that an iterative minimization of these functions leads to the convergence of energy values under weak assumptions as well as the convergence of the iterates to critical points of the energy under more restrictive assumptions.
Our approach has several attractive properties. 
It generates a set of feasible iterates and it is very easy to use large step-sizes analytically, as these are independent of the Lipschitz properties of the inner function. Our convergence results naturally extend previous work. 
In practice, we extensively analyze the algorithm on synthetic examples, where the sub-problems can be solved globally and show that it can find better minima than related methods. Lastly we show an intended application. The use of recent functional lifting techniques to solve the nonconvex majorizer, critically allows us to find visually appealing solutions to the complicated composite problem of time-of-flight reconstruction from noisy low resolution data.

\subsection*{Acknowledgements}
The authors acknowledge the support of the German Research Foundation (DFG) via the research training group GRK 1564 Imaging New Modalities and the project "Functional Lifting 2.0", as well as the support of the Daimler and Benz Foundation.

\bibliographystyle{plain}

\begin{thebibliography}{10}

\bibitem{alberti_calibration_2003}
{\sc G.~Alberti, G.~Bouchitt\'e, and G.~Dal~Maso}, {\em The calibration method
  for the {{Mumford}}-{{Shah}} functional and free-discontinuity problems},
  Calculus of Variations and Partial Differential Equations, 16 (2003),
  pp.~299--333, \url{https://doi.org/10.1007/s005260100152}.

\bibitem{artina_linearly_2013}
{\sc M.~Artina, M.~Fornasier, and F.~Solombrino}, {\em Linearly {{Constrained
  Nonsmooth}} and {{Nonconvex Minimization}}}, SIAM J. Optim., 23 (2013),
  pp.~1904--1937, \url{https://doi.org/10.1137/120869079}.

\bibitem{attouch_convergence_2013}
{\sc H.~Attouch, J.~Bolte, and B.~F. Svaiter}, {\em Convergence of descent
  methods for semi-algebraic and tame problems: Proximal algorithms,
  forward\textendash{}backward splitting, and regularized
  {{Gauss}}\textendash{{Seidel}} methods}, Math. Program., 137 (2013),
  pp.~91--129, \url{https://doi.org/10.1007/s10107-011-0484-9}.

\bibitem{audet_analysis_2002}
{\sc C.~Audet and J.~Dennis}, {\em Analysis of {{Generalized Pattern
  Searches}}}, SIAM J. Optim., 13 (2002), pp.~889--903,
  \url{https://doi.org/10.1137/S1052623400378742}.

\bibitem{bardaro_nonlinear_2003}
{\sc C.~Bardaro, J.~Musielak, and G.~Vinti}, {\em Nonlinear {{Integral
  Operators}} and {{Applications}}}, {Walter de Gruyter}, Jan. 2003.

\bibitem{bauschke_descent_2017}
{\sc H.~H. Bauschke, J.~Bolte, and M.~Teboulle}, {\em A {{Descent Lemma Beyond
  Lipschitz Gradient Continuity}}: {{First}}-{{Order Methods Revisited}} and
  {{Applications}}}, Mathematics of Operations Research, 42 (2017),
  pp.~330--348, \url{https://doi.org/10.1287/moor.2016.0817}.

\bibitem{bauschke_legendre_1997}
{\sc H.~H. Bauschke and J.~J. Borwein}, {\em Legendre {{Functions}} and the
  {{Method}} of {{Random Bregman Projections}}}, Journal of Convex Analysis, 4
  (1997), pp.~27--67.

\bibitem{bauschke_essential_2001}
{\sc H.~H. Bauschke, J.~M. Borwein, and P.~L. Combettes}, {\em Essential
  {{Smoothness}}, {{Essential Strict Convexity}}, and {{Legendre Functions}} in
  {{Banach Spaces}}}, Communications in Contemporary Mathematics, 03 (2001),
  pp.~615--647, \url{https://doi.org/10.1142/S0219199701000524}.

\bibitem{beck_fast_2009}
{\sc A.~Beck and M.~Teboulle}, {\em A {{Fast Iterative
  Shrinkage}}-{{Thresholding Algorithm}} for {{Linear Inverse Problems}}}, SIAM
  Journal on Imaging Sciences, 2 (2009), pp.~183--202,
  \url{https://doi.org/10.1137/080716542}.

\bibitem{benning_gradient_2017}
{\sc M.~Benning, M.~M. Betcke, M.~J. Ehrhardt, and C.~Sch\"onlieb}, {\em
  Gradient descent in a generalised {{Bregman}} distance framework}, in Joint
  {{Conference Geometric Numerical Integration}} and Its {{Applications}},
  vol.~74, Melbourne, Mar. 2017, {MI Lecture Note Kyushu University},
  pp.~40--45.

\bibitem{bolte_lojasiewicz_2007}
{\sc J.~Bolte, A.~Daniilidis, and A.~Lewis}, {\em The {{\L{}ojasiewicz
  Inequality}} for {{Nonsmooth Subanalytic Functions}} with {{Applications}} to
  {{Subgradient Dynamical Systems}}}, SIAM J. Optim., 17 (2007),
  pp.~1205--1223, \url{https://doi.org/10.1137/050644641}.

\bibitem{bolte_clarke_2007}
{\sc J.~Bolte, A.~Daniilidis, A.~Lewis, and M.~Shiota}, {\em Clarke
  {{Subgradients}} of {{Stratifiable Functions}}}, SIAM J. Optim., 18 (2007),
  pp.~556--572, \url{https://doi.org/10.1137/060670080}.

\bibitem{bolte_majorization-minimization_2016}
{\sc J.~Bolte and E.~Pauwels}, {\em Majorization-{{Minimization Procedures}}
  and {{Convergence}} of {{SQP Methods}} for {{Semi}}-{{Algebraic}} and {{Tame
  Programs}}}, Mathematics of OR, 41 (2016), pp.~442--465,
  \url{https://doi.org/10.1287/moor.2015.0735}.

\bibitem{bolte_proximal_2014}
{\sc J.~Bolte, S.~Sabach, and M.~Teboulle}, {\em Proximal alternating
  linearized minimization for nonconvex and nonsmooth problems}, Math.
  Program., 146 (2014), pp.~459--494,
  \url{https://doi.org/10.1007/s10107-013-0701-9}.

\bibitem{bolte_nonconvex_2018}
{\sc J.~Bolte, S.~Sabach, and M.~Teboulle}, {\em Nonconvex
  {{Lagrangian}}-{{Based Optimization}}: {{Monitoring Schemes}} and {{Global
  Convergence}}}, arXiv:1801.03013 [math],  (2018),
  \url{https://arxiv.org/abs/1801.03013}.

\bibitem{bolte_first_2018}
{\sc J.~Bolte, S.~Sabach, M.~Teboulle, and Y.~Vaisbourd}, {\em First {{Order
  Methods Beyond Convexity}} and {{Lipschitz Gradient Continuity}} with
  {{Applications}} to {{Quadratic Inverse Problems}}}, SIAM J. Optim., 28
  (2018), pp.~2131--2151, \url{https://doi.org/10.1137/17M1138558}.

\bibitem{bonettini_variable_2016}
{\sc S.~Bonettini, I.~Loris, F.~Porta, and M.~Prato}, {\em Variable {{Metric
  Inexact Line}}-{{Search}}-{{Based Methods}} for {{Nonsmooth Optimization}}},
  SIAM Journal on Optimization, 26 (2016), pp.~891--921,
  \url{https://doi.org/10.1137/15M1019325}.

\bibitem{boykov_fast_2001-1}
{\sc Y.~Boykov, O.~Veksler, and R.~Zabih}, {\em Fast approximate energy
  minimization via graph cuts}, IEEE Transactions on Pattern Analysis and
  Machine Intelligence, 23 (2001), pp.~1222--1239,
  \url{https://doi.org/10.1109/34.969114}.

\bibitem{burger_guide_2013}
{\sc M.~Burger and S.~Osher}, {\em A {{Guide}} to the {{TV}} zoo}, in {{PDE}}
  Based {{Reconstruction Methods}} in {{Imaging}}, no.~2090 in Lecture Notes in
  Mathematics, {Springer International Publishing}, Switzerland, 1~ed., 2013.

\bibitem{chambolle_convex_2012}
{\sc A.~Chambolle, D.~Cremers, and T.~Pock}, {\em A {{Convex Approach}} to
  {{Minimal Partitions}}}, SIAM Journal on Imaging Sciences, 5 (2012),
  pp.~1113--1158, \url{https://doi.org/10.1137/110856733}.

\bibitem{chambolle_first-order_2011}
{\sc A.~Chambolle and T.~Pock}, {\em A {{First}}-{{Order Primal}}-{{Dual
  Algorithm}} for {{Convex Problems}} with {{Applications}} to {{Imaging}}}, J
  Math Imaging Vis, 40 (2011), pp.~120--145,
  \url{https://doi.org/10.1007/s10851-010-0251-1}.

\bibitem{chan_algorithms_2006}
{\sc T.~Chan, S.~Esedoglu, and M.~Nikolova}, {\em Algorithms for {{Finding
  Global Minimizers}} of {{Image Segmentation}} and {{Denoising Models}}}, SIAM
  J. Appl. Math., 66 (2006), pp.~1632--1648,
  \url{https://doi.org/10.1137/040615286}.

\bibitem{chen_convergence_1997}
{\sc G.~Chen and R.~Rockafellar}, {\em Convergence {{Rates}} in
  {{Forward}}--{{Backward Splitting}}}, SIAM J. Optim., 7 (1997), pp.~421--444,
  \url{https://doi.org/10.1137/S1052623495290179}.

\bibitem{chouzenoux_variable_2014}
{\sc E.~Chouzenoux, J.-C. Pesquet, and A.~Repetti}, {\em Variable {{Metric
  Forward}}\textendash{{Backward Algorithm}} for {{Minimizing}} the {{Sum}} of
  a {{Differentiable Function}} and a {{Convex Function}}}, J Optim Theory
  Appl, 162 (2014), pp.~107--132,
  \url{https://doi.org/10.1007/s10957-013-0465-7}.

\bibitem{drusvyatskiy_slope_2013}
{\sc D.~Drusvyatskiy}, {\em Slope and Geometry in Variational Mathematics}, PhD
  thesis, Cornell University, 2013.

\bibitem{drusvyatskiy_nonsmooth_2016}
{\sc D.~Drusvyatskiy, A.~D. Ioffe, and A.~S. Lewis}, {\em Nonsmooth
  optimization using {{Taylor}}-like models: Error bounds, convergence, and
  termination criteria}, arXiv:1610.03446 [math],  (2016),
  \url{https://arxiv.org/abs/1610.03446}.

\bibitem{drusvyatskiy_efficiency_2016}
{\sc D.~Drusvyatskiy and C.~Paquette}, {\em Efficiency of minimizing
  compositions of convex functions and smooth maps}, arXiv:1605.00125 [math],
  (2016), \url{https://arxiv.org/abs/1605.00125}.

\bibitem{frerix_proximal_2018}
{\sc T.~Frerix, T.~M\"ollenhoff, M.~Moeller, and D.~Cremers}, {\em Proximal
  {{Backpropagation}}}, in International {{Conference}} on {{Learning
  Representations}} ({{ICLR}}), 2018, \url{https://arxiv.org/abs/1706.04638}.

\bibitem{gabay_dual_1976}
{\sc D.~Gabay and B.~Mercier}, {\em A dual algorithm for the solution of
  nonlinear variational problems via finite element approximation}, Computers
  \& Mathematics with Applications, 2 (1976), pp.~17--40,
  \url{https://doi.org/10.1016/0898-1221(76)90003-1}.

\bibitem{goldberg_genetic_1989}
{\sc D.~E. Goldberg}, {\em Genetic {{Algorithms}} in {{Search}},
  {{Optimization}} and {{Machine Learning}}}, {Addison-Wesley Longman
  Publishing Co., Inc.}, Boston, MA, USA, 1st~ed., 1989.

\bibitem{hansen_global_2004}
{\sc E.~R. Hansen and G.~W. Walster}, {\em Global Optimization Using Interval
  Analysis}, no.~264 in Monographs and textbooks in pure and applied
  mathematics, {Marcel Dekker}, New York, 2nd ed., revised and expanded~ed.,
  2004.

\bibitem{hanzely_accelerated_2018}
{\sc F.~Hanzely, P.~Richtarik, and L.~Xiao}, {\em Accelerated {{Bregman
  Proximal Gradient Methods}} for {{Relatively Smooth Convex Optimization}}},
  arXiv:1808.03045 [math],  (2018), \url{https://arxiv.org/abs/1808.03045}.

\bibitem{hunter_tutorial_2004}
{\sc D.~R. Hunter and K.~Lange}, {\em A {{Tutorial}} on {{MM Algorithms}}}, The
  American Statistician, 58 (2004), pp.~30--37,
  \url{https://doi.org/10.1198/0003130042836}.

\bibitem{huyer_global_1999}
{\sc W.~Huyer and A.~Neumaier}, {\em Global {{Optimization}} by {{Multilevel
  Coordinate Search}}}, Journal of Global Optimization, 14 (1999),
  pp.~331--355, \url{https://doi.org/10.1023/A:1008382309369}.

\bibitem{ingber_adaptive_1996}
{\sc L.~Ingber}, {\em Adaptive simulated annealing ({{ASA}}): {{Lessons}}
  learned}, Control Cybernetics, 25 (1996), pp.~33--54.

\bibitem{ioffe_invitation_2009}
{\sc A.~Ioffe}, {\em An {{Invitation}} to {{Tame Optimization}}}, SIAM J.
  Optim., 19 (2009), pp.~1894--1917, \url{https://doi.org/10.1137/080722059}.

\bibitem{ishikawa_segmentation_1998}
{\sc H.~Ishikawa and D.~Geiger}, {\em Segmentation by {{Grouping Junctions}}},
  in Proceedings of the {{IEEE Computer Society Conference}} on {{Computer
  Vision}} and {{Pattern Recognition}}, CVPR '98, Washington, DC, USA, 1998,
  {IEEE Computer Society}, pp.~125--.

\bibitem{kearfott_rigorous_2013}
{\sc R.~B. Kearfott}, {\em Rigorous {{Global Search}}: {{Continuous
  Problems}}}, {Springer Science \& Business Media}, Mar. 2013.

\bibitem{kennedy_particle_1995}
{\sc J.~Kennedy and R.~C. Eberhardt}, {\em Particle {{Swarm Optimization}}},
  Proceedings of the 1995 IEEE International Conference on Neural Networks
  (Conference proceedings), 4 (1995), pp.~1942--1948.

\bibitem{kingma_adam:_2015}
{\sc D.~P. Kingma and J.~Ba}, {\em Adam: {{A Method}} for {{Stochastic
  Optimization}}}, in International {{Conference}} on {{Learning
  Representations}} ({{ICLR}}), San Diego, May 2015,
  \url{https://arxiv.org/abs/1412.6980}.

\bibitem{kolmogorov_what_2004}
{\sc V.~Kolmogorov and R.~Zabin}, {\em What energy functions can be minimized
  via graph cuts?}, IEEE Transactions on Pattern Analysis and Machine
  Intelligence, 26 (2004), pp.~147--159,
  \url{https://doi.org/10.1109/TPAMI.2004.1262177}.

\bibitem{kuschk_fast_2013-1}
{\sc G.~Kuschk and D.~Cremers}, {\em Fast and {{Accurate Large}}-{{Scale Stereo
  Reconstruction Using Variational Methods}}}, in 2013 {{IEEE International
  Conference}} on {{Computer Vision Workshops}}, Dec. 2013, pp.~700--707,
  \url{https://doi.org/10.1109/ICCVW.2013.96}.

\bibitem{lange_3d_2000}
{\sc R.~Lange}, {\em {{3D Time}}-of-{{Flight Distance Measurement}} with
  {{Custom Solid}}-{{State Image Sensors}} in {{CMOS}}/{{CCD}}-{{Technology}}},
  PhD thesis, University of Siegen, Siegen, June 2000.

\bibitem{laude_sublabel-accurate_2016-1}
{\sc E.~Laude, T.~M\"ollenhoff, M.~Moeller, J.~Lellmann, and D.~Cremers}, {\em
  Sublabel-{{Accurate Convex Relaxation}} of {{Vectorial Multilabel
  Energies}}}, in Computer {{Vision}} \textendash{} {{ECCV}} 2016, Lecture
  Notes in Computer Science, {Springer, Cham}, Oct. 2016, pp.~614--627,
  \url{https://doi.org/10.1007/978-3-319-46448-0_37}.

\bibitem{lewis_proximal_2016-1}
{\sc A.~S. Lewis and S.~J. Wright}, {\em A proximal method for composite
  minimization}, Math. Program., 158 (2016), pp.~501--546,
  \url{https://doi.org/10.1007/s10107-015-0943-9}.

\bibitem{lindner_lateral_2006}
{\sc M.~Lindner and A.~Kolb}, {\em Lateral and depth calibration of
  pmd-distance sensors}, in International {{Symposium}} on {{Visual
  Computing}}, {Springer}, 2006, pp.~524--533.

\bibitem{luenberger_linear_2015}
{\sc D.~G. Luenberger and Y.~Ye}, {\em Linear and {{Nonlinear Programming}}},
  {Springer}, New York, NY, 4th ed. 2016 edition~ed., June 2015.

\bibitem{mairal_optimization_2013-1}
{\sc J.~Mairal}, {\em Optimization with {{First}}-order {{Surrogate
  Functions}}}, in Proceedings of the 30th {{International Conference}} on
  {{International Conference}} on {{Machine Learning}} - {{Volume}} 28,
  ICML'13, Atlanta, GA, USA, 2013, {JMLR.org}, pp.~III--783--III--791.

\bibitem{mollenhoff_sublabel-accurate_2017}
{\sc T.~M\"ollenhoff and D.~Cremers}, {\em Sublabel-{{Accurate Discretization}}
  of {{Nonconvex Free}}-{{Discontinuity Problems}}}, Proceedings of the IEEE
  International Conference on Computer Vision,  (2017), pp.~1183--1191,
  \url{https://doi.org/10.1109/ICCV.2017.134}.

\bibitem{mollenhoff_sublabel-accurate_2016}
{\sc T.~M\"ollenhoff, E.~Laude, M.~Moeller, J.~Lellmann, and D.~Cremers}, {\em
  Sublabel-{{Accurate Relaxation}} of {{Nonconvex Energies}}}, in Proceedings
  of the {{IEEE Conference}} on {{Computer Vision}} and {{Pattern
  Recognition}}, 2016, pp.~3948--3956,
  \url{https://doi.org/10.1109/CVPR.2016.428}.

\bibitem{muhlenbein_parallel_1991}
{\sc H.~M\"uhlenbein, M.~Schomisch, and J.~Born}, {\em The parallel genetic
  algorithm as function optimizer}, Parallel Computing, 17 (1991),
  pp.~619--632, \url{https://doi.org/10.1016/S0167-8191(05)80052-3}.

\bibitem{nesterov_introductory_2004}
{\sc Y.~Nesterov}, {\em Introductory {{Lectures}} on {{Convex Optimization}}},
  vol.~87 of Applied Optimization, {Springer US}, Boston, MA, 2004.

\bibitem{nesterov_gradient_2013}
{\sc Y.~Nesterov}, {\em Gradient methods for minimizing composite functions},
  Math. Program., 140 (2013), pp.~125--161,
  \url{https://doi.org/10.1007/s10107-012-0629-5}.

\bibitem{ochs_unifying_2016}
{\sc P.~Ochs}, {\em Unifying abstract inexact convergence theorems for descent
  methods and block coordinate variable metric {{iPiano}}}, arXiv:1602.07283
  [math],  (2016), \url{https://arxiv.org/abs/1602.07283}.

\bibitem{ochs_ipiano:_2014}
{\sc P.~Ochs, Y.~Chen, T.~Brox, and T.~Pock}, {\em {{iPiano}}: {{Inertial
  Proximal Algorithm}} for {{Nonconvex Optimization}}}, SIAM J. Imaging Sci., 7
  (2014), pp.~1388--1419, \url{https://doi.org/10.1137/130942954}.

\bibitem{ochs_iterated_2013}
{\sc P.~Ochs, A.~Dosovitskiy, T.~Brox, and T.~Pock}, {\em An {{Iterated L1
  Algorithm}} for {{Non}}-smooth {{Non}}-convex {{Optimization}} in {{Computer
  Vision}}}, in Proceedings of the {{IEEE Conference}} on {{Computer Vision}}
  and {{Pattern Recognition}} ({{CVPR}}), {IEEE}, June 2013, pp.~1759--1766,
  \url{https://doi.org/10.1109/CVPR.2013.230}.

\bibitem{ochs_iteratively_2015}
{\sc P.~Ochs, A.~Dosovitskiy, T.~Brox, and T.~Pock}, {\em On {{Iteratively
  Reweighted Algorithms}} for {{Nonsmooth Nonconvex Optimization}} in
  {{Computer Vision}}}, SIAM J. Imaging Sci., 8 (2015), pp.~331--372,
  \url{https://doi.org/10.1137/140971518}.

\bibitem{ochs_non-smooth_2017}
{\sc P.~Ochs, J.~Fadili, and T.~Brox}, {\em Non-smooth {{Non}}-convex {{Bregman
  Minimization}}: {{Unification}} and new {{Algorithms}}}, arXiv:1707.02278
  [cs, math],  (2017), \url{https://arxiv.org/abs/1707.02278}.

\bibitem{pauwels_value_2016-1}
{\sc E.~Pauwels}, {\em The value function approach to convergence analysis in
  composite optimization}, Operations Research Letters, 44 (2016),
  pp.~790--795, \url{https://doi.org/10.1016/j.orl.2016.10.003}.

\bibitem{pock_global_2010}
{\sc T.~Pock, D.~Cremers, H.~Bischof, and A.~Chambolle}, {\em Global
  {{Solutions}} of {{Variational Models}} with {{Convex Regularization}}}, SIAM
  J. Imaging Sci., 3 (2010), pp.~1122--1145,
  \url{https://doi.org/10.1137/090757617}.

\bibitem{precup_methods_2002-1}
{\sc R.~Precup}, {\em Methods in {{Nonlinear Integral Equations}}}, {Springer
  Netherlands}, 2002.

\bibitem{rockafellar_convex_1970}
{\sc R.~T. Rockafellar}, {\em Convex {{Analysis}}}, {Princeton University
  Press}, Princeton, N.J, 1970.

\bibitem{rockafellar_variational_2009}
{\sc R.~T. Rockafellar and R.~J.-B. Wets}, {\em Variational {{Analysis}}},
  vol.~317 of Grundlehren der mathematischen Wissenschaften, {Springer-Verlag
  Berlin Heidelberg}, Berlin Heidelberg, 3rd~ed., June 2009,
  \url{https://doi.org/10.1007/978-3-642-02431-3}.

\bibitem{rudin_nonlinear_1992}
{\sc L.~I. Rudin, S.~Osher, and E.~Fatemi}, {\em Nonlinear total variation
  based noise removal algorithms}, Physica D: Nonlinear Phenomena, 60 (1992),
  pp.~259--268, \url{https://doi.org/10.1016/0167-2789(92)90242-F}.

\bibitem{scharstein_high-resolution_2014}
{\sc D.~Scharstein, H.~Hirschm\"uller, Y.~Kitajima, G.~Krathwohl, N.~Ne{\v
  s}i\'c, X.~Wang, and P.~Westling}, {\em High-{{Resolution Stereo Datasets}}
  with {{Subpixel}}-{{Accurate Ground Truth}}}, in Pattern {{Recognition}},
  Lecture Notes in Computer Science, {Springer, Cham}, Sept. 2014, pp.~31--42,
  \url{https://doi.org/10.1007/978-3-319-11752-2_3}.

\bibitem{scherzer_variational_2009}
{\sc O.~Scherzer, M.~Grasmair, H.~Grossauer, M.~Haltmeier, and F.~Lenzen}, {\em
  Variational {{Methods}} in {{Imaging}}}, no.~167 in Applied Mathematical
  Sciences, {Springer}, New York, 1~ed., 2009.

\bibitem{strekalovskiy_convex_2014}
{\sc E.~Strekalovskiy, A.~Chambolle, and D.~Cremers}, {\em Convex
  {{Relaxation}} of {{Vectorial Problems}} with {{Coupled Regularization}}},
  SIAM Journal on Imaging Sciences, 7 (2014), pp.~294--336,
  \url{https://doi.org/10.1137/130908348}.

\bibitem{sun_majorization-minimization_2017}
{\sc Y.~Sun, P.~Babu, and D.~P. Palomar}, {\em Majorization-{{Minimization
  Algorithms}} in {{Signal Processing}}, {{Communications}}, and {{Machine
  Learning}}}, IEEE Transactions on Signal Processing, 65 (2017), pp.~794--816,
  \url{https://doi.org/10.1109/TSP.2016.2601299}.

\bibitem{tao_convex_1997}
{\sc P.~D. Tao and L.~T.~H. An}, {\em Convex analysis approach to dc
  programming: {{Theory}}, algorithms and applications}, Acta Mathematica
  Vietnamica, 22 (1997), pp.~289--355.

\bibitem{ugray_scatter_2007}
{\sc Z.~Ugray, L.~Lasdon, J.~Plummer, F.~Glover, J.~Kelly, and R.~Mart\'i},
  {\em Scatter {{Search}} and {{Local NLP Solvers}}: {{A Multistart Framework}}
  for {{Global Optimization}}}, INFORMS Journal on Computing, 19 (2007),
  pp.~328--340, \url{https://doi.org/10.1287/ijoc.1060.0175}.

\bibitem{valkonen_primaldual_2014}
{\sc T.~Valkonen}, {\em A primal\textendash{}dual hybrid gradient method for
  nonlinear operators with applications to {{MRI}}}, Inverse Problems, 30
  (2014), p.~055012, \url{https://doi.org/10.1088/0266-5611/30/5/055012}.

\bibitem{wu_convergence_1983}
{\sc C.~F.~J. Wu}, {\em On the {{Convergence Properties}} of the {{EM
  Algorithm}}}, The Annals of Statistics, 11 (1983), pp.~95--103,
  \url{https://doi.org/10.2307/2240463}.

\bibitem{xiao_defocus_2015}
{\sc L.~Xiao, F.~Heide, M.~O'Toole, A.~Kolb, M.~B. Hullin, K.~Kutulakos, and
  W.~Heidrich}, {\em Defocus deblurring and superresolution for time-of-flight
  depth cameras}, in 2015 {{IEEE Conference}} on {{Computer Vision}} and
  {{Pattern Recognition}} ({{CVPR}}), {IEEE}, 2015, pp.~2376--2384,
  \url{https://doi.org/10.1109/CVPR.2015.7298851}.

\end{thebibliography}

\end{document}